\documentclass[11pt,a4paper]{article}

\usepackage[left=2.3cm, top=2.3cm,bottom=2.3cm,right=2.3cm]{geometry}

\usepackage{mathtools,amssymb,amsthm,mathrsfs,calc,graphicx,stmaryrd,xcolor,dsfont,tikz,pgfplots,bbm,float,array,epsfig,hyperref,color}
\usepackage[numbers,square]{natbib}
\usepackage[british]{babel}
\usepackage{amsfonts}              

\numberwithin{equation}{section}
\numberwithin{figure}{section}

\newtheorem {theorem}{Theorem}[section]
\newtheorem {proposition}[theorem]{Proposition}
\newtheorem {lemma}[theorem]{Lemma}
\newtheorem {corollary}[theorem]{Corollary}

{\theoremstyle{definition}

}

{
\theoremstyle{remark}
\newtheorem {remark}[theorem]{Remark}
}

\newcommand{\Beta}{\text{\rm Beta}}
\newcommand{\eqdistr}{\stackrel{d}{=}}

\newcommand{\Vol}{\operatorname{Vol}}

\DeclareMathOperator*{\argmin}{arg\,min}
\renewcommand{\Re}{\operatorname{Re}}  

\newcommand{\ii}{{\rm{i}}}

\def\ba{\begin{array}}
\def\ea{\end{array}}
\def\bea{\begin{eqnarray} \label}
\def\eea{\end{eqnarray}}
\def\be{\begin{equation} \label}
\def\ee{\end{equation}}
\def\bit{\begin{itemize}}
\def\eit{\end{itemize}}
\def\ben{\begin{enumerate}}
\def\een{\end{enumerate}}

\def\lan{\langle}
\def\ran{\rangle}

\def\BB{\mathbb{B}}
\def\CC{\mathbb{C}}
\def\E{\mathbb{E}}

\def\N{\mathbb{N}}
\def\P{\mathbb{P}}

\def\R{\mathbb{R}}
\def\RRd1{\mathbb{R}^{d+1}}
\def\SS{\mathbb{S}}

\def\bS{\mathbb{S}}

\def\b{\beta}

\def\G{\Gamma}

\def\cF{\mathcal{F}}

\def\cH{\mathcal{H}}
\def\cI{\mathcal{I}}

\def\dint{\textup{d}}

\newcommand{\eee}{{\rm e}}

\newcommand{\ind}{\mathbbm{1}}
\newcommand{\eps}{\varepsilon}
\newcommand{\pos}{\mathop{\mathrm{pos}}\nolimits}
\newcommand{\aff}{\mathop{\mathrm{aff}}\nolimits}
\newcommand{\lin}{\mathop{\mathrm{lin}}\nolimits}
\newcommand{\conv}{\mathop{\mathrm{conv}}\nolimits}

\newcommand{\dd}{{\rm d}}
\newcommand{\bsl}{\backslash}
\DeclareMathOperator{\relint}{relint}
\newcommand{\todistr}{\overset{d}{\underset{n\to\infty}\longrightarrow}}

\newcommand{\Mod}[1]{\ (\mathrm{mod}\ #1)}

\begin{document}

\title{\bfseries Beta polytopes and Poisson polyhedra: \\$f$-vectors and angles}

\author{Zakhar Kabluchko, Christoph Th\"ale and Dmitry Zaporozhets}

\date{}

\maketitle

\begin{abstract}
We study random polytopes of the form $[X_1,\ldots,X_n]$ defined as convex hulls of independent and identically distributed random points $X_1,\ldots,X_n$ in $\mathbb{R}^d$ with one of the following densities:
$$
f_{d,\beta} (x) = c_{d,\beta} (1-\|x\|^2)^{\beta}, \qquad \|x\| < 1, \quad \text{(beta distribution, $\beta>-1$)}
$$
or
$$
\tilde f_{d,\beta} (x) = \tilde{c}_{d,\beta} (1+\|x\|^2)^{-\beta}, \qquad x\in\mathbb{R}^d, \quad \text{(beta' distribution, $\beta>d/2$)}.
$$
This setting also includes the uniform distribution on the unit sphere and the standard normal distribution as limiting cases.
We derive exact and asymptotic formulae for the expected number of $k$-faces of $[X_1,\ldots,X_n]$ for arbitrary $k\in\{0,1,\ldots,d-1\}$.  We prove that for any such $k$ this expected number is strictly monotonically increasing with $n$. Also, we compute the expected internal and external angles of these polytopes at faces of every dimension and, more generally, the expected conic intrinsic volumes of their tangent cones. By passing to the large $n$ limit in the beta' case, we compute the expected $f$-vector of the convex hull of Poisson point processes with power-law intensity function.  Using convex duality, we derive exact formulae for the expected number of $k$-faces of the zero cell for a class of isotropic Poisson hyperplane tessellations in $\R^d$. This family includes the zero cell of a classical stationary and isotropic Poisson hyperplane tessellation and the typical cell of a stationary Poisson--Voronoi tessellation as special cases. In addition, we prove precise limit theorems for this $f$-vector in the high-dimensional regime, as $d\to\infty$. Finally, we relate the $d$-dimensional beta and beta' distributions to the generalized Pareto distributions known in  extreme-value theory.

\end{abstract}

\tableofcontents

\section{Introduction and main results}

\subsection{Introduction}
Let $X_1,\ldots,X_n$ be random points chosen independently and uniformly from the unit sphere $\bS^{d-1}$ or the unit ball $\BB^{d}$. Their convex hull $[X_1,\ldots,X_n]$ is a random polytope; see Figure~\ref{fig:beta_polytope}. What is the expected number of vertices, edges, or, more generally, $k$-dimensional faces of this random polytope? What are the expected internal and external angles of this polytope? Does the expected number of $k$-dimensional faces increase if we add one more point to the sample?

In order to address  these questions, it is useful (and probably even necessary) to consider a more general family of distributions including the aforementioned examples as special or limit cases. We say that a random vector in $\R^d$ has a $d$-dimensional \textit{beta distribution} with parameter $\b>-1$ if its Lebesgue density is
\begin{equation}\label{eq:def_f_beta}
f_{d,\b}(x)=c_{d,\beta} \left( 1-\left\| x \right\|^2 \right)^\b\ind_{\{\|x\| <  1\}},\qquad x\in\R^d,\qquad
c_{d,\b}= \frac{ \G\left( \frac{d}{2} + \b + 1 \right) }{ \pi^{ \frac{d}{2} } \G\left( \b+1 \right) }.
\end{equation}
Here, $\|x\| = (x_1^2+\ldots+x_d^2)^{1/2}$ denotes the Euclidean norm of the vector $x= (x_1,\ldots,x_d)\in\R^d$. The uniform distribution on the unit ball $\BB^d$ is recovered by taking $\b=0$, whereas the uniform distribution on the unit sphere $\bS^{d-1}$ is the weak limit of the beta distribution, as $\b\downarrow -1$. Very similar to the beta distributions are the \textit{beta' distributions} with Lebesgue density
\begin{equation}\label{eq:def_f_beta_prime}
\tilde{f}_{d,\b}(x)=\tilde{c}_{d,\b} \left( 1+\left\| x \right\|^2 \right)^{-\b},\qquad
x\in\R^d,\qquad
\tilde{c}_{d,\b}= \frac{ \G\left( \b \right) }{\pi^{ \frac{d}{2} } \G\left( \b - \frac{d}{2} \right) },
\end{equation}
where the parameter $\beta$ should satisfy $\beta>d/2$ to ensure integrability. The standard normal distribution on $\R^d$ can be viewed as a limiting case of both the beta  and beta' family, as $\beta\to +\infty$. In fact, the four $d$-dimensional  distributions mentioned above (i.e.\ the beta distribution, the beta' distribution, the normal distribution and the uniform distribution on the sphere) are characterized by a common underlying property discovered by Ruben and Miles~\cite{ruben_miles}. This characterizing property is also crucial in the present context and will be discussed in more detail below.

\begin{figure}[t]
\begin{center}
\includegraphics[width=0.49\textwidth ]{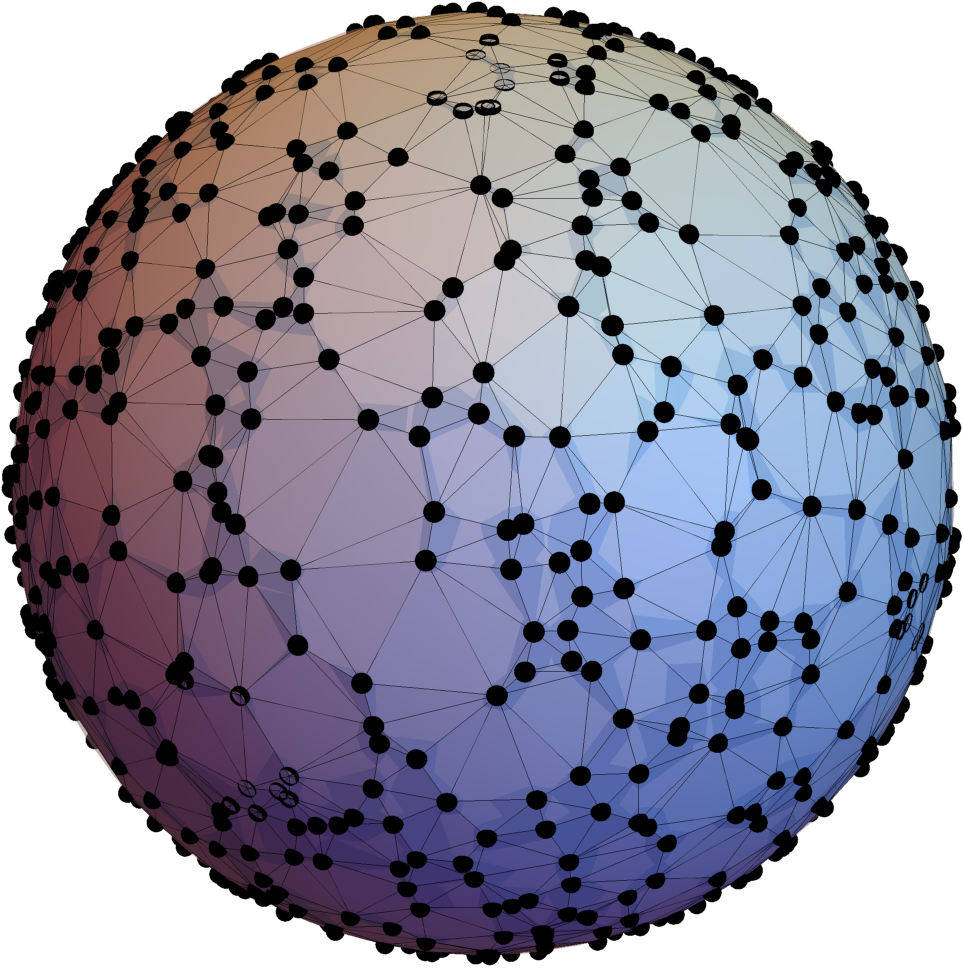}
\includegraphics[width=0.49\textwidth]{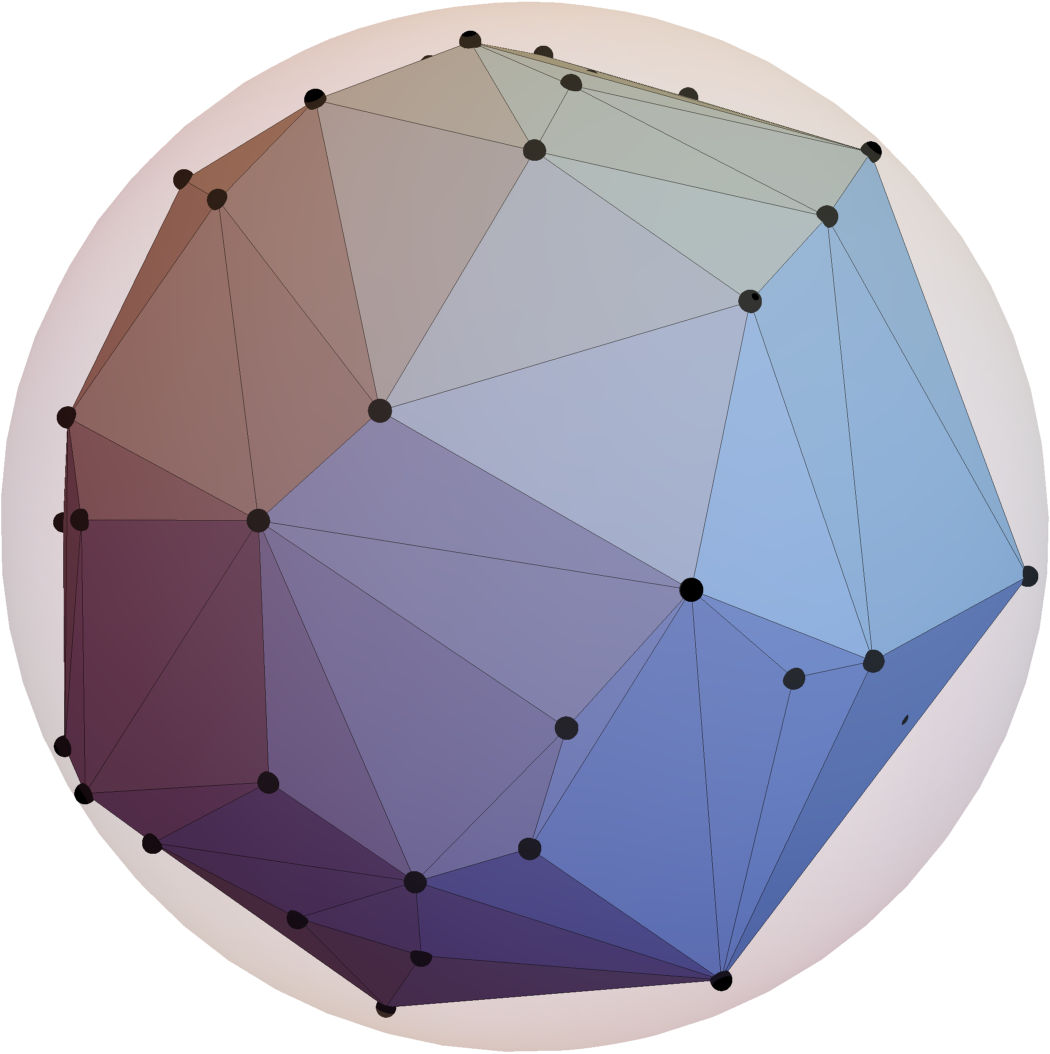}
\end{center}
\caption{
Convex hull of $n=1000$ uniformly distributed points on the sphere (left figure, beta polytope with $\beta=-1$) and the ball (right figure, beta polytope with $\beta=0)$.
}
\label{fig:beta_polytope}
\end{figure}

Convex hulls of $n\geq d+1$ independent random points sampled according to these distributions in $\R^d$ are referred to as \textit{beta} and \textit{beta' polytopes}. Beta and beta' polytopes for the particular case $n=d+1$ (where these polytopes are simplices with probability one) were considered in the works of Miles \cite{miles}, Ruben and Miles \cite{ruben_miles} and, more recently, by Grote, Kabluchko and Th\"ale \cite{beta_simplices}. Asymptotic properties of the beta and beta' polytopes in the general case $n\geq d+1$ were studied by Affentranger \cite{affentranger}, while explicit formulae for some characteristics of these polytopes like the expected intrinsic volumes and the expected number of hyperfaces were derived by Kabluchko, Temesvari and Th\"ale \cite{beta_polytopes}. Let us also point out that the class of beta' polytopes also plays a crucial role in the recent study of spherical convex hulls of random points on half spheres. This connection has been exploited in the works of Bonnet, Grote, Temesvari, Th\"ale, Turchi and Wespi \cite{bonnet_etal} and Kabluchko, Marynych, Temesvari and Th\"ale \cite{convex_hull_sphere}. In this light, the present paper can be regarded as the continuation of our previous works on beta and beta' polytopes. Its main results, which will be presented in Sections~\ref{subsec:MainForBeta} and~\ref{subsec:MainForBetaprime}, can roughly be summarized as follows.

\begin{itemize}
\item[(a)] We provide an explicit formula for the expected number of $k$-dimensional faces of beta and beta' polytopes, for every $k\in\{0,1,\ldots,d-1\}$.
\item[(b)] We prove that the expected number of $k$-dimensional faces strictly increases if new points are added to the sample, again for every $k\in\{0,1,\ldots,d-1\}$.
\item[(c)] We compute the expected external and internal angles of beta and beta' polytopes.
\end{itemize}

In addition, these results have a number of corollaries which are presented in Sections~\ref{subsec:PPP}, \ref{subsec:convex_hulls_half_sphere}, \ref{subsec:PHT}, \ref{subsec:PHT_asympt} and~\ref{subsec:extreme_pareto}. They can be summarized as follows.

\begin{itemize}
\item[(d)] We provide a formula for the expected number of $k$-dimensional faces of the convex hull of a Poisson point process with power-law intensity, for every $k\in\{0,1,\ldots,d-1\}$.
\item[(e)] From (d) we deduce a formula for the expected number of $k$-faces of the zero cell of a Poisson hyperplane tessellation and the typical cell of the Poisson--Voronoi tessellation.
\item[(f)] We provide asymptotic formulae for the expected $f$-vector of these cells in high dimensions, i.e., as $d$ goes to $\infty$.
\item[(g)] We relate the characterizing property of the beta and beta' distributions  which is crucial for obtaining the above results to the properties of the generalized Pareto distributions known in extreme-value theory.
\end{itemize}

As already mentioned above, the standard Gaussian distribution appears as the large $\beta$ limit of both beta and beta' distributions. More concretely, we have the following
\begin{lemma}\label{lem:gauss_limit}
If $X(\beta)$ is a random point in $\R^d$ with density either $f_{d,\beta}$ or $\tilde f_{d, \beta}$, then $\sqrt{2\beta} X(\beta)$ converges weakly to the standard normal distribution on $\R^d$, as $\beta\to +\infty$.
\end{lemma}
\begin{proof}
Write down the density of $\sqrt{2\beta} X(\beta)$, verify that it converges pointwise to the standard normal density and apply Scheff\'e's lemma.
\end{proof}

Most of the results of the present paper can be translated to Gaussian polytopes by taking the limit $\beta\to+\infty$. Since in the Gaussian setting most results are not new and admit simpler and more elegant proofs, see, e.g., \cite{kabluchko_thaele} for the proof of monotonicity, we refrain from considering the Gaussian case here.

\subsection{Main results for beta polytopes}\label{subsec:MainForBeta}
Let $X_1,\ldots,X_n$ be independent and identically distributed (i.i.d.)\ random points in $\R^d$ with beta distribution $f_{d,\beta}$. Their convex hull  is called the beta polytope and will be denoted by
$$
P_{n,d}^\beta := [X_1,\ldots,X_n].
$$
Unless otherwise stated, in all results on beta polytopes the parameters $d\in \N$ and $n\in\N$ satisfy $n\geq d+1$ (so that $P_{n,d}^\beta$ has full dimension $d$ a.s.), while the parameter $\beta$ satisfies $\beta\geq -1$. The value $\beta=-1$ corresponds to the uniform distribution on the unit sphere $\bS^{d-1}$. We are interested in computing  expected values of various characteristics of the beta polytopes $P_{n,d}^\beta$.

Given a polytope $P\subset\R^d$, we denote by $\cF_k(P)$,  where $k\in\{0,\ldots,d\}$, the set of $k$-dimensional faces of $P$ and by $f_k(P)=|\cF_k(P)|$ their total number. Note that the random beta and beta' polytopes considered in the present paper are \textit{simplicial}, that is, all of their faces are simplices, with probability $1$.  If $F$ is a face of $P$, let $\beta(F,P)$ (respectively, $\gamma(F,P)$) be the internal (respectively, external) solid angle at $F$. The normalization is chosen so that the solid angle of the full space is equal to $1$.  For convenience of the reader, we collect the necessary background information from convex and stochastic geometry in Section~\ref{sec:facts}.

\begin{theorem}[Expected $f$-vector]\label{theo:f_vect}
For every $k\in \{0,1,\ldots, d-1\}$, the expected number of $k$-dimensional faces of $P_{n,d}^\beta$ is given by
\begin{equation}\label{eq:f_k_P_main}
\E f_k(P_{n,d}^{\beta})
=
2 \sum_{s=0}^\infty \binom n {d-2s} \binom {d-2s}{k+1} I_{n,d-2s}(2\beta+d) J_{d-2s,k+1}\left(\beta + s + \frac 12\right).
\end{equation}
Here, the quantities $I_{n,k}(\alpha)$  are defined for $n\in\N$, $k\in \{1,\ldots,n\}$ and $\alpha>-1/k$ by the formula
\begin{equation}\label{eq:I_definition}
I_{n,k}(\alpha) := \int_{-1}^{+1}  c_{1, \frac {\alpha k -1}{2}}
(1-t^2)^{\frac {\alpha k - 1}{2}}
\left(\int_{-1}^t c_{1, \frac{\alpha - 1}{2}} (1-s^2)^{\frac{\alpha - 1}{2}}\,\dd s\right)^{n-k} \,\dd t,
\end{equation}
while $J_{m,\ell}(\alpha)$ denotes the expected  internal angle at some $(\ell-1)$-dimensional face of the simplex $[Z_1,\ldots,Z_{m}] \subseteq \R^{m-1}$, where $Z_1,\ldots,Z_m$ are i.i.d.\ points in $\R^{m-1}$ with distribution $f_{m-1,\alpha}$. That is,
\begin{equation}\label{eq:J_definition}
J_{m,\ell}(\alpha) :=  \E\beta([Z_1,\ldots,Z_{\ell}], [Z_1,\ldots,Z_m])
\end{equation}
for $m\in\N$, $\ell\in\{1,\ldots,m\}$ and $\alpha \geq -1$.
\end{theorem}
\begin{remark}
In this paper we shall use the convention that $\binom{a}{b}=0$ whenever $b>a$ or $b<0$. In particular, this implies that the sum in~\eqref{eq:f_k_P_main} contains only finitely many non-zero terms. More concretely, only the terms corresponding to $s\in \{0,1,\ldots\}$ satisfying $d - 2s > k$ appear in the sum. In the trivial case $d=1$ (where the only admissible value of $k$ is $k=0$ and we have $\E f_0(P_{n,1}^\beta) = 2$), it is easy to check that $I_{n,1} (\alpha) = 1/n$ for all $\alpha>-1$, but the value $I_{n,1} (-1)$ is not well-defined. We use the natural convention that $I_{n,1}(-1):= -1/n$.
\end{remark}
\begin{remark}\label{rem:SpecialCasesMainResultBeta}
For faces of dimension $k=d-1$ and $k=d-2$ the expression in~\eqref{eq:f_k_P_main} simplifies considerably, since the only non-vanishing term is the one with $s=0$, and we get
\begin{align*}
\E f_{d-1}(P_{n,d}^{\beta})
=
2  \binom n {d} I_{n,d}(2\beta+d)
\qquad\text{and}\qquad
\E f_{d-2}(P_{n,d}^{\beta})
=
d\binom n {d}  I_{n,d}(2\beta+d),
\end{align*}
where we used that $J_{m,m}(\alpha)=1$ and $J_{m,m-1}(\alpha) = 1/2$. The first formula recovers a result obtained in~\cite[Theorem~2.11, Remark 2.14]{beta_polytopes}, whereas the second one follows from the Dehn--Sommerville relation $2 f_{d-2}(P) = d f_{d-1}(P)$ valid for any $d$-dimensional simplicial polytope $P$.
\end{remark}

In the deterministic setting, it is easy to construct examples which show that adding one more point to the convex hull may increase or decrease the number of $k$-dimensional faces. However, in the setting of random polytopes, it is natural to conjecture that adding one more point should \textit{increase} the \textit{expected} number of $k$-faces, for every $k\in\{0,1,\ldots,d-1\}$. This conjecture is known to hold in several special cases. The work of Devillers, Glisse, Goaoc, Moroz and Reitzner~\cite{devillers} covers the case of the expected vertex number for convex hulls of uniformly distributed points in a planar convex body. For faces of maximal dimension it was established in the work of Beermann and Reitzner~\cite{beermann_diss,beermann_reitzner} for Gaussian polytopes  and in~\cite{bonnet_etal} by Bonnet, Grote, Temesvari, Th\"ale, Turchi and Wespi for beta and beta' polytopes. So far the only model where monotonicity of the expected number of $k$-faces is known for arbitrary $k\in\{0,1,\ldots,d-1\}$ are the Gaussian polytopes \cite{kabluchko_thaele}.
The explicit formula stated in Theorem~\ref{theo:f_vect} allows us to add another positive answer to the conjecture for the $k$-faces of beta polytopes, where $k\in\{0,1,\ldots,d-1\}$.

\begin{theorem}[Monotonicity of the expected $f$-vector]\label{theo:monoton}
For all $d\geq 2$,  $n\geq d+1$ and $k\in \{0,1,\ldots, d-1\}$ we have
$$
\E f_k(P_{n,d}^{\beta}) < \E f_k(P_{n+1,d}^{\beta}).
$$
\end{theorem}

The quantities $I_{n,k}(\alpha)$ and $J_{m,\ell}(\alpha)$ that appeared in \eqref{eq:I_definition} and \eqref{eq:J_definition}, respectively, will play a central role in the sequel.
The next theorem shows that the quantities $I_{n,k}(\alpha)$ can be interpreted as the \textit{expected external angles} of beta simplices (and, more generally, of beta polytopes).

\begin{theorem}[Expected external angles]\label{theo:external}
Fix some $k\in\{1,\ldots,d\}$ and consider the simplex $G := [X_1,\ldots,X_k]$. The expected external angle at $G$ is given by
$$
\E \gamma(G, P_{n,d}^\beta)
=
I_{n,k}(2\beta + d)
$$
with the convention that $\gamma(G, P_{n,d}^\beta)=0$ if $G$ is not a face of $P_{n,d}^\beta$.
Furthermore, the random variable $\gamma(G, P_{n,d}^\beta)$ is stochastically independent of the isometry type of the simplex  $G/\sqrt{1 - h^2}$, where $h:=d(0,\aff G)$ is the distance from the origin to the affine hull of $G$.
\end{theorem}

\begin{remark}\label{rem:I_n_k_alternative}
Let us mention two alternative expressions for $I_{n,k}(\alpha)$:
\begin{align*}
I_{n,k}(\alpha)
&=\int_{-\pi/2}^{+\pi/2} c_{1,\frac{\alpha k - 1}{2}} (\cos \varphi)^{\alpha k} \left(\int_{-\pi/2}^\varphi c_{1,\frac{\alpha-1}{2}}(\cos \theta)^{\alpha} \,\dd \theta \right)^{n-k} \, \dd \varphi\\
&=\int_{-\infty}^{+\infty} c_{1,\frac{\alpha k - 1}{2}} (\cosh \varphi)^{-(\alpha k+1)} \left(\int_{-\infty}^\varphi c_{1,\frac{\alpha-1}{2}}(\cosh \theta)^{-(\alpha+1)} \,\dd \theta \right)^{n-k} \, \dd \varphi.
\end{align*}
The first formula can be obtained from~\eqref{eq:I_definition} by the change of variables $t=\sin\varphi$, $s= \sin \theta$ with $\varphi, \theta\in (-\frac \pi 2, +\frac \pi 2)$, whereas for the second we put $t= \tanh \varphi$, $s= \tanh \theta$ with $\varphi,\theta\in \R$.
\end{remark}

Finding an explicit formula for the \textit{expected internal angles} $J_{m,\ell}(\alpha)$ of beta simplices is a much more difficult question (except for the two trivial cases mentioned in Remark \ref{rem:SpecialCasesMainResultBeta} above and the identity $J_{3,1}(\alpha) = 1/6$ which is valid because the sum of the angles of a triangle is $\pi$). This is not surprising, since even in the limiting case, as $\alpha\to\infty$, where it is possible to show that $J_{m,\ell}(\alpha)$ tends to  the internal angle of an $(m-1)$-dimensional regular simplex at an $(\ell-1)$-dimensional face~\cite{kabluchko_zaporozhets_gauss_simplex,goetze_kabluchko_zaporozhets}, an explicit formula  is not well known to specialists. Explicit and asymptotic (as the dimension goes to $\infty$) formulae for the internal angles of regular simplices can be found in~\cite{boroczky_henk,kabluchko_zaporozhets_absorption,rogers_packing,rogers,vershik_sporyshev}.
The methods used in these papers do not seem to generalize to the finite $\alpha$ case. 
The problem of computing the quantities $J_{m,\ell}(\alpha)$ requires methods going beyond stochastic geometry and will be addressed in the subsequent publications~\cite{kabluchko_algorithm,kabluchko_angles,kabluchko_formula,kabluchko_poisson_zero}.
In particular, it will be shown in~\cite{kabluchko_formula} that
\begin{equation*}
J_{m,\ell}\left(\frac{\alpha - m + 1}{2}\right)
=
\int_{-\pi/2}^{+\pi/2} c_{1,\frac{\alpha m}2} (\cos \varphi)^{\alpha m + 1}
\left(\int_{-\infty}^{\ii\varphi}  c_{1,\frac{\alpha-1}{2}} (\cosh \theta)^{-(\alpha+1)}\dd \theta \right)^{m-\ell} \dd \varphi,
\end{equation*}
for all integer $m\geq 3$, $\ell\in \{1,\ldots,m\}$ and  $\alpha \geq m-3$. This will turn all formulae involving the quantities $J_{m,\ell}(\alpha)$ into fully explicit results.



In the next theorem we analyze the asymptotic behavior of the expected number of $k$-faces of $P_{n,d}^\beta$ when $n\to\infty$ and all other parameters stay fixed.

\begin{theorem}[Asymptotics of the $f$-vector]\label{theo:f_vector_asympt_beta}
For any fixed $d\in\N$ and $k\in \{0,1,\ldots, d-1\}$ we have
\begin{align*}
\lim_{n\to\infty} n^{-\frac{d-1}{2\beta + d + 1}}  \E f_k(P_{n,d}^\beta)
&=
\frac{2}{d!} \binom {d}{k+1}  J_{d,k+1}\left(\beta + \frac 12\right)
\frac{c_{1, \frac {(2\beta+d) d -1}{2}}}{2\beta+d+1}\\
&\qquad\qquad\qquad\times\left(\frac{2\beta+d+1}{c_{1,\frac{2\beta+d -1}{2}}}\right)^{\frac{(2\beta+d)  d + 1}{2\beta+d +1}}
\Gamma\left(\frac{(2\beta+d)  d+1}{2\beta+d+1}\right).
\end{align*}
\end{theorem}

\begin{remark}
In the case $\beta=-1$, which corresponds to the uniform distribution on the sphere $\bS^{d-1}$, the above simplifies to
$$
\lim_{n\to\infty} \frac 1n  \E f_k(P_{n,d}^{-1})
={2^d\pi^{{d\over 2}-1}\over d(d-1)^2}{d\choose k+1}J_{d,k+1}\left(-{1\over 2}\right){\Gamma(1+{d(d-2)\over 2})\over\Gamma({(d-1)^2\over 2})}\left({\Gamma({d+1\over 2})\over\Gamma({d\over 2})}\right)^{d-1}.
$$
Except for the case $k=d-1$, where $J_{d,d}(-1/2)=1$ and which is mentioned in Buchta, M\"uller and Tichy \cite{buchta_mueller_tichy}, such an explicit result seems to be new, although the order of $\E f_k(P_{n,d}^{-1})$ in $n$ was determined in the thesis \cite{stemeseder_phd} using entirely different tools. Similarly, in the case $\beta=0$ corresponding to the uniform distribution on the ball $\BB^d$, we  obtain
\begin{equation}\label{eq:Limitf_kBall}
\begin{split}
\lim_{n\to\infty} n^{-\frac{d-1}{d+1}}  \E f_k(P_{n,d}^0) &= {2\pi^{d(d-1)\over 2(d+1)}\over (d+1)!}{d\choose k+1}J_{d,k+1}\left({1\over 2}\right)\\
&\qquad\qquad\times{\Gamma(1+{d^2\over 2})\Gamma({d^2+1\over d+1})\over\Gamma({d^2+1\over 2})}\left({(d+1)\Gamma({d+1\over 2})\over \Gamma(1+{d\over 2})}\right)^{d^2+1\over d+1}.
\end{split}
\end{equation}
Again, except for the case $k=d-1$, which is treated in \cite{affentranger}, such an explicit result seems new.
\end{remark}

Let us point out the following connection to a question of Reitzner. In \cite{ReitznerCombinatorialStructure} he has shown that if $K_n$ is the convex hull of $n\geq d+1$ uniformly distributed random points in a convex body $K\subseteq\R^d$ with twice differentiable boundary $\partial K$ and everywhere positive Gaussian curvature $\kappa(\,\cdot\,)$ then, for every $k\in\{0,1,\ldots,d-1\}$,
\begin{equation}\label{eq:ReitznerExpectation}
\lim_{n\to\infty}n^{-{d-1\over d+1}}\E f_k(K_n)
=
c_{d,k}\Omega(K)/\Vol_d(K)^{\frac{d-1}{d+1}}
\end{equation}
with
\begin{equation*}
\Omega(K):=\int_{\partial K}\kappa(x)^{1\over d+1}\,\cH^{d-1}(\dint x)
\end{equation*}
being the so-called \textit{affine surface area} of $K$ ($\cH^{d-1}$ is the $(d-1)$-dimensional Hausdorff measure) and where $c_{d,k}$ is a constant only depending on $d$ and on $k$. In~\cite{ReitznerCombinatorialStructure}, Equation~\eqref{eq:ReitznerExpectation} is stated under the assumption that $K$ has unit volume. The general case of~\eqref{eq:ReitznerExpectation} is a consequence of the scaling property of the affine surface area
$$
\Omega(r K) = r^{\frac{d(d-1)}{d+1}}\Omega (K), \qquad r>0,
$$
which can be found, e.g., in Theorem~3.6 of~\cite{hug_affine}.
Unfortunately and as pointed out in \cite[p.\ 181]{ReitznerCombinatorialStructure} it was not possible so far to determine the constant $c_{d,k}$ explicitly and in an accessible form. But since \eqref{eq:ReitznerExpectation} is true in particular for $K=\BB^d$ and since the affine surface area of $\BB^d$ is $\Omega(\BB^d)=2\pi^{d/2}/\Gamma({d\over 2})$, we can identify $c_{d,k}$ with  $\Vol_d(\BB^d)^{\frac{d-1}{d+1}}/\Omega(\BB^d)$ times the right hand side in \eqref{eq:Limitf_kBall}. We summarize these findings in the next proposition.

\begin{proposition}
For all $k\in \{0,1,\ldots, d-1\}$,  the constant $c_{d,k}$ in \eqref{eq:ReitznerExpectation} is given by
\begin{equation*}
c_{d,k}
=
\frac{\Gamma(1+\frac d2)^{\frac 2 {d+1}}} {d (d+1)!\pi^{\frac d {d+1}}}
{d\choose k+1}J_{d,k+1}\left({1\over 2}\right){\Gamma(1+{d^2\over 2})\Gamma({d^2+1\over d+1})\over\Gamma({d^2+1\over 2})}\left({(d+1)\Gamma({d+1\over 2})\over \Gamma(1+{d\over 2})}\right)^{d^2+1\over d+1}.
\end{equation*}
\end{proposition}

\begin{remark}
We remark that for Gaussian polytopes, a representation of this type, involving the internal angle of a regular simplex, is well known from \cite[Equations (4.1) and (4.2)]{HMR04}.
\end{remark}

In the next theorem we evaluate the expected conic intrinsic volumes of the tangent cones at faces of the beta polytope. The definition of tangent cones and conic intrinsic volumes (which include internal and external solid angles as special cases), together with a list of their properties, will be given in Section~\ref{sec:facts}.
\begin{theorem}[Expected conic intrinsic volumes of tangent cones]\label{theo:expected_conic_tangent}
Fix some $k\in\{1,\ldots,d\}$ and consider the simplex $G := [X_1,\ldots,X_k]$. Then, for every $j\in\{k-1,\ldots,d-1\}$, the expected $j$-th conic intrinsic volume of the tangent cone $T(G, P_{n,d}^\beta)$ at $G$ is given by
\begin{align*}
\E \upsilon_{j} (T(G, P_{n,d}^{\beta}))
&=
{n-k\choose j-k+1}  I_{n,j+1}(2\beta+d) J_{j+1,k}\left(\beta + \frac{d - j}2\right),
\end{align*}
with the convention that $T(G, P_{n,d}^{\beta})=\R^d$ if $G$ is not a face of $P_{n,d}^{\beta}$.
\end{theorem}
Taking $j=k-1$  and observing that $\upsilon_{k-1} (T(G, P_{n,d}^{\beta})) = \gamma(T(G, P_{n,d}^{\beta}))$   provided $G$ is a face (this is because the tangent cone contains the $(k-1)$-dimensional affine  hull of $G$ shifted to $0$ as its lineality space), we recover the first claim of Theorem~\ref{theo:external} as a special case of Theorem~\ref{theo:expected_conic_tangent}.
Another interesting special case, $j=d$, is not included in Theorem~\ref{theo:expected_conic_tangent}. For $j=d$, we have the following result.

\begin{theorem}[Expected internal angles]\label{theo:internal}
Fix some $k\in\{1,\ldots,d\}$ and consider the simplex $G := [X_1,\ldots,X_k]$. The expected internal angle of $P_{n,d}^\beta$ at $G$ is given by
\begin{align*}
\E \left[\beta(G, P_{n,d}^\beta) \ind_{\{G\in\cF_{k-1}(P_{n,d}^\beta)\}} \right]
&=
\sum_{m=k}^d (-1)^{d-m} \binom {n-k}{m-k} I_{n,m}(2\beta+d) J_{m,k}\left(\beta +\frac{d-m+1}{2}\right),
\\
\E \beta(G, P_{n,d}^\beta)
&=
1-  \sum_{m=k}^d \binom {n-k}{m-k} I_{n,m}(2\beta+d) J_{m,k}\left(\beta +\frac{d-m+1}{2}\right)
,
\end{align*}
where in the second formula we use the convention that $\beta(G, P_{n,d}^\beta)=1$ if $G$ is not a face of $P_{n,d}^{\beta}$.
\end{theorem}


\subsection{Main results for beta' polytopes}\label{subsec:MainForBetaprime}
In this section we present our results for beta' polytopes. Let $X_1,\ldots,X_n$ be i.i.d.\ random points in $\R^d$ with density $\tilde f_{d,\beta}$.  Their convex hull will be denoted by
$$
\tilde P_{n,d}^\beta := [X_1,\ldots,X_n]
$$
and called the beta' polytope.  Unless otherwise stated, in all results on beta' polytopes we assume that the parameters $n\in\N$  and $d\in\N$ satisfy $n\geq d+1$, so that $\tilde P_{n,d}^\beta$ has full dimension $d$ a.s., while the parameter $\beta$ satisfies $\beta > d/2$. The following is the analogue of Theorem \ref{theo:f_vect}.
\begin{theorem}[Expected $f$-vector]\label{theo:f_vect_prime}
For every $k\in \{0,1,\ldots, d-1\}$, the expected number of $k$-dimensional faces of $\tilde P_{n,d}^\beta$ is given by
\begin{equation}
\E f_k(\tilde P_{n,d}^{\beta})
=
2 \sum_{s=0}^\infty \binom n {d-2s} \binom {d-2s}{k+1} \tilde I_{n,d-2s}(2\beta-d) \tilde J_{d-2s,k+1}\left(\beta - s - \frac 12\right).
\end{equation}
Here, the quantities $\tilde I_{n,k}(\alpha)$ are defined for $n\in\N$, $k\in\{1,\ldots,n\}$ and $\alpha>0$ by the formula
\begin{equation}\label{eq:I_definition_prime}
\tilde I_{n,k}(\alpha)
:=
\int_{-\infty}^{+\infty} \tilde c_{1, \frac {\alpha k + 1}{2}}
(1+t^2)^{-\frac {\alpha k + 1}{2}}
\left(\int_{-\infty}^t \tilde c_{1, \frac{\alpha + 1}{2}} (1+s^2)^{-\frac{\alpha + 1}{2}}\,\dd s\right)^{n-k} \dd t,
\end{equation}
while $\tilde J_{m,\ell}(\alpha)$ denotes the expected  internal angle at some $(\ell-1)$-dimensional face of the simplex $[Z_1,\ldots,Z_{m}] \subset \R^{m-1}$, where $Z_1,\ldots,Z_m$ are i.i.d.\ points in $\R^{m-1}$ with density $\tilde f_{m-1,\alpha}$. That is,
\begin{equation}\label{eq:J_definition_prime}
\tilde J_{m,\ell}(\alpha) :=  \E \beta([Z_1,\ldots,Z_{\ell}], [Z_1,\ldots,Z_m])
\end{equation}
for $m\in\N$, $\ell\in\{1,\ldots,m\}$ and  $\alpha>\frac{m-1}{2}$.
\end{theorem}

The next theorem is the analogue of Theorem \ref{theo:monoton} and shows that the expected $f$-vector is strictly monotonically increasing as a function of the number $n$ of points.

\begin{theorem}[Monotonicity of the expected $f$-vector]\label{theo:monoton_prime}
For all $d\geq 2$,  $n\geq d+1$ and $k\in\{0,1,\ldots, d-1\}$ we have
$$
\E f_k(\tilde P_{n,d}^{\beta}) < \E f_k(\tilde P_{n+1,d}^{\beta}).
$$
\end{theorem}


Our next result on beta' polytopes is the following analogue of Theorem~\ref{theo:external}.

\begin{theorem}[Expected external angles]\label{theo:external_prime}
Fix some $k\in\{1,\ldots,d\}$ and consider the simplex $G := [X_1,\ldots,X_k]$. The expected external angle at $G$ is given by
$$
\E \gamma(G, \tilde P_{n,d}^\beta)
=
\tilde I_{n,k}(2\beta - d)
$$
with the convention that $\gamma(G, \tilde P_{n,d}^\beta)=0$ if $G$ is not a face of $\tilde P_{n,d}^\beta$. Furthermore, the random variable $\gamma(G, \tilde P_{n,d}^\beta)$ is stochastically independent of the isometry type of the simplex  $G/\sqrt{1 + h^2}$, where $h:=d(0,\aff G)$ is the distance from the origin to the affine hull of $G$.
\end{theorem}

\begin{remark}
As in the beta case, we have two alternative expressions for $\tilde I_{n,k}(\alpha)$:
\begin{align*}
\tilde I_{n,k}(\alpha)
&=\int_{-\pi/2}^{+\pi/2} \tilde c_{1,\frac{\alpha k + 1}{2}} (\cos \varphi)^{\alpha k-1} \left(\int_{-\pi/2}^\varphi \tilde c_{1,\frac{\alpha+1}{2}}(\cos \theta)^{\alpha-1} \,\dd \theta \right)^{n-k} \, \dd \varphi\\
&=\int_{-\infty}^{+\infty} \tilde c_{1,\frac{\alpha k + 1}{2}} (\cosh \varphi)^{-\alpha k} \left(\int_{-\infty}^\varphi \tilde c_{1,\frac{\alpha+1}{2}}(\cosh \theta)^{-\alpha} \,\dd \theta \right)^{n-k} \, \dd \varphi.
\end{align*}
These formulae can be obtained from~\eqref{eq:I_definition_prime} by the changes of variables $t=\tan \varphi$, $s= \tan \theta$,  and $t= \sinh \varphi$, $s= \sinh \theta$, respectively.
The problem of computing $\tilde J_{m,\ell}(\alpha)$ requires methods going beyond the scope of the present paper and will be addressed in~\cite{kabluchko_formula}, where it will be shown that
\begin{equation*}
\tilde J_{m,\ell}\left(\frac{\alpha + m - 1}{2}\right)
=
\int_{-\pi/2}^{+\pi/2} \tilde c_{1,\frac{\alpha m}2} (\cos \varphi)^{\alpha m - 2} \left(\int_{-\infty}^{\ii \varphi} \tilde c_{1,\frac{\alpha+1}{2}}(\cosh \theta)^{-\alpha}\dd \theta \right)^{m-\ell} \dd \varphi,
\end{equation*}
for all $\alpha>0$, $m\in\N$ and $\ell\in \{1,\ldots,m\}$ such that $\alpha m>1$.
\end{remark}

Finally, we present a formula for the expected conic intrinsic volumes of the tangent cones at the faces of a beta' polytope.
\begin{theorem}[Expected conic intrinsic volumes of tangent cones]\label{theo:expected_conic_tangent_prime}
Fix some $k\in\{1,\ldots,d\}$ and consider the simplex $G := [X_1,\ldots,X_k]$. Then, for every $j\in\{k-1,\ldots,d-1\}$, the expected $j$-th conic intrinsic volume of the tangent cone $T(G, \tilde P_{n,d}^\beta)$ at $G$ is given by
\begin{align*}
\E \upsilon_{j} (T(G, \tilde P_{n,d}^{\beta}))
=
{n-k\choose j-k+1} \tilde I_{n,j+1}(2\beta-d) \tilde J_{j+1,k}\left(\beta - \frac{d - j}2\right),
\end{align*}
with the convention that $T(G, \tilde P_{n,d}^{\beta})=\R^d$ if $G$ is not a face of $\tilde P_{n,d}^{\beta}$.
\end{theorem}

With the choice $j=k-1$ we recover the first claim of Theorem~\ref{theo:external_prime}.
In the case $j=d$, which is not covered by Theorem~\ref{theo:expected_conic_tangent_prime}, we have the following analogue of Theorem~\ref{theo:internal}.

\begin{theorem}[Expected internal angles]\label{theo:internal_prime}
Fix some $k\in\{1,\ldots,d\}$ and consider the simplex $G := [X_1,\ldots,X_k]$. The expected internal angle of $\tilde P_{n,d}^\beta$ at $G$ is given by
\begin{align*}
\E \left[\beta(G, \tilde P_{n,d}^\beta) \ind_{\{G\in\cF_{k-1}(\tilde P_{n,d}^\beta)\}} \right]
&=
\sum_{m=k}^d (-1)^{d-m} \binom {n-k}{m-k} \tilde I_{n,m}(2\beta-d) \tilde J_{m,k}\left(\beta - \frac{d-m+1}{2}\right),
\\
\E \beta(G, \tilde P_{n,d}^\beta)
&=
1-  \sum_{m=k}^d \binom {n-k}{m-k} \tilde I_{n,m}(2\beta - d) \tilde J_{m,k}\left(\beta - \frac{d-m+1}{2}\right)
,
\end{align*}
with the convention that $\beta(G, \tilde P_{n,d}^\beta)=1$ if $G$ is not a face of $\tilde P_{n,d}^{\beta}$.
\end{theorem}

\begin{remark}
The methods of the present paper can be adapted to treat the \textit{symmetric} beta and beta' polytopes which are defined as the convex hulls of $\pm X_1,\ldots,\pm X_n$, where $X_1,\ldots,X_n$ are i.i.d.\ with beta  or beta' distribution. However, we refrain from considering symmetric polytopes here.
\end{remark}

%

\subsection{Poisson point processes with power-law intensity}\label{subsec:PPP}
In the large $n$ limit, rescaled samples from the beta' distribution converge to the Poisson point process with a power-law intensity function. This can be used to obtain results on the convex hull of this class of Poisson point process. For $\alpha>0$ let $\Pi_{d,\alpha}$ be a Poisson point process on $\R^d\backslash\{0\}$ with power-law intensity function
$$
x\mapsto \|x\|^{-d-\alpha},\qquad x\in \R^d \bsl \{0\}.
$$
The number of points of $\Pi_{d,\alpha}$ outside any ball centered at the origin is finite, but the total number of points is infinite, and, in fact, the origin is an accumulation point for the atoms of $\Pi_{d,\alpha}$, with probability $1$; see the left panel of Figure~\ref{fig:poisson}.  The convex hull of the atoms of $\Pi_{d,\alpha}$ will be denoted by $\conv \Pi_{d,\alpha}$. In~\cite{convex_hull_sphere} it was shown that $\conv \Pi_{d,\alpha}$ is almost surely a polytope, and explicit formulae for its expected intrinsic volumes  and expected number of $(d-1)$-dimensional faces were given. Using the results obtained in Section \ref{subsec:MainForBetaprime} we can now provide an explicit formula for the expected number of $k$-dimensional faces of $\conv \Pi_{d,\alpha}$ for any $k\in\{0,1,\ldots,d-1\}$.

\begin{theorem}\label{theo:f_vect_poisson}
For every $d\in\N$ and $k\in \{0,1,\ldots,d-1\}$, the expected number of $k$-faces of $\conv \Pi_{d,\alpha}$ is given by
\begin{equation} \label{eq:E_f_k_beta_prime_to_poisson}
\begin{split}
&\E f_k(\conv \Pi_{d,\alpha})
=
\lim_{n\to\infty}\E f_k\left(\tilde P_{n,d}^{\frac{d+\alpha}{2}}\right)
\\
&\quad=2 \sum_{\substack{m\in \{k+1,\ldots,d\}\\ m\equiv d \Mod{2}}} \frac{\Gamma\left(\frac{m\alpha + 1}{2}\right) \Gamma\left(\frac \alpha 2\right)^{m}}{\Gamma\left(\frac{m\alpha}{2}\right)\Gamma\left(\frac{\alpha+1}{2}\right)^{m}}\frac{(\sqrt \pi \alpha)^{m-1}}{m}  \binom {m}{k+1} \tilde J_{m,k+1}\left(\frac{m-1+\alpha}{2}\right),
\end{split}
\end{equation}
where $\tilde J_{m,k+1}(\alpha)$ is defined as in Theorem~\ref{theo:f_vect_prime}.
\end{theorem}

\begin{remark}\label{rem:f_d-1_d-2}
For faces of dimensions  $k=d-1$ and $k=d-2$ the result simplifies to
\begin{align*}
\E f_{d-1}(\conv \Pi_{d,\alpha})
&=
\frac 2d (\sqrt \pi \alpha)^{d-1} \frac{\Gamma\left(\frac{d\alpha + 1}{2}\right) \Gamma\left(\frac \alpha 2\right)^{d}}{\Gamma\left(\frac{d\alpha}{2}\right)\Gamma\left(\frac{\alpha+1}{2}\right)^{d}},\\
\E f_{d-2}(\conv \Pi_{d,\alpha})
&=
\frac d2 \E f_{d-1}(\conv \Pi_{d,\alpha}).
\end{align*}
The first formula was obtained in~\cite[Corollary 2.13]{convex_hull_sphere}, whereas the second one is valid for every simplicial polytope (even almost surely without expectations) by the Dehn--Sommerville equations. Note that although the intensity function used here differs by a multiplicative constant from that used in~\cite{convex_hull_sphere}, the expected $f$-vector is the same in both cases because in the case of power-law intensity, multiplying intensity by a constant is equivalent to spatial rescaling the Poisson point process, which does not affect the $f$-vector of the convex hull.
\end{remark}

\subsection{Convex hulls on the half-sphere}\label{subsec:convex_hulls_half_sphere}

Let us mention an application of the above results to the random spherical convex hulls first studied by  B\'ar\'any, Hug, Reitzner and Schneider~\cite{barany_etal}. Let $U_1,\ldots,U_n$ be independent random points  distributed uniformly on the $d$-dimensional upper half-sphere $\bS^d_+ = \bS^d\cap \{x_0\geq 0\}\subset\RRd1$.  Let $C_n:=\pos (U_1,\ldots,U_n)$ be the random cone generated by these points.
The $f$-vector of the random spherical polytope  $C_n\cap \bS^d_+$ has the same distribution as the $f$-vector of $\tilde P_{n,d}^{\beta}$ with $\beta= (d+1)/2$; see~\cite{bonnet_etal,convex_hull_sphere}.
Theorem~\ref{theo:f_vect_prime} with $m:=d-2s$ yields
$$
\E f_k(C_n\cap \bS^d_+)
=
2 \sum_{\substack{m\in \{k+1,\ldots,d\}\\ m\equiv d \Mod{2}}} \binom n {m} \binom {m}{k+1} \tilde I_{n,m}(1) \tilde J_{m,k+1}\left(\frac{m}{2}\right)
$$
for all $k\in\{0,\ldots,d-1\}$. Further, Theorem~\ref{theo:monoton_prime} implies that the expected $f$-vector of the random spherical polytope $C_n\cap \bS^d_+$ increases component-wise with $n$. In fact, the limits to which these vectors converge, as $n\to\infty$, are finite. Namely, in~\cite{convex_hull_sphere} it was shown that
$$
\lim_{n\to\infty} \E f_{k+1}^\ell(C_n) = \lim_{n\to\infty} \E f_{k}^\ell(C_n\cap \bS^d_+) = \E f_{k}^\ell(\conv \Pi_{d,1}),
$$
for $k\in\{0,1,\ldots,d-1\}$ and any $\ell\in\N$.
Using Theorem~\ref{theo:f_vect_poisson}, we arrive at the following asymptotic formula for the particular case $\ell=1$:
$$
\lim_{n\to\infty} \E f_{k}(C_n\cap \bS^d_+)
=
2\sqrt \pi \sum_{\substack{m\in \{k+1,\ldots,d\}\\ m\equiv d \Mod{2}}} \frac{\Gamma\left(\frac{m + 1}{2}\right)}{\Gamma\left(\frac{m}{2}\right)}  \frac{\pi^{m-1}}{m}   \binom {m}{k+1} \tilde J_{m,k+1}\left(\frac{m}{2}\right).
$$
The cases $k\in\{0, d-1,d-2\}$ were treated in~\cite{barany_etal}. In particular, the limit for $k=0$ was expressed in~\cite[Theorem~7.1]{barany_etal} in terms of certain constant $C(d)$ given as a multiple integral in~\cite[Equation (22)]{barany_etal}. The limit of the complete expected $f$-vector was expressed in~\cite[Theorem~2.4]{convex_hull_sphere} in terms of multiple integrals that can be interpreted as absorption probabilities of the Poisson point process $\Pi_{d,1}$. Our approach provides an alternative formula in terms of the quantities $\tilde J_{m,\ell}(\alpha)$.

Let us finally comment on the first equality in~\eqref{eq:E_f_k_beta_prime_to_poisson}. It follows from standard results in extreme-value theory that if $X_1,X_2,\ldots$ are i.i.d.\ points in $\R^d$ with density $\tilde f_{d,\frac {d+\alpha}{2}}$, then the point process
$$
\sum_{j=1}^n \delta_{n^{-1/\alpha} X_j}
$$
converges, as $n\to\infty$, to the Poisson point process $\Pi_{d,\alpha}$ weakly on the space of locally finite integer-valued measures on $\R^d\backslash\{0\}$ endowed with the vague topology, see \cite[Equation (4.6)]{convex_hull_sphere}. From this one can deduce the  distributional convergence
$$
f_k (\tilde P_{n,d}^{\frac{d+\alpha}{2}}) \todistr f_k (\conv \Pi_{d,\alpha})
$$
together with the convergence of all moments from the continuous mapping theorem as in~\cite{convex_hull_sphere}. In fact, in~\cite{convex_hull_sphere} we considered only the case $\alpha=1$ (which was tailored towards the application to convex hulls on the half-sphere~\cite{barany_etal}), but the same method of proof applies to any $\alpha>0$. In the proof of Theorem~\ref{theo:f_vect_poisson}, which will be given in Section~\ref{sec:proof_poisson_limit}, we shall prove the second line of~\eqref{eq:E_f_k_beta_prime_to_poisson} by using the explicit formula for the expected $f$-vector of a beta' polytope.


\begin{figure}[t]
\begin{center}
\includegraphics[width=0.49\textwidth ]{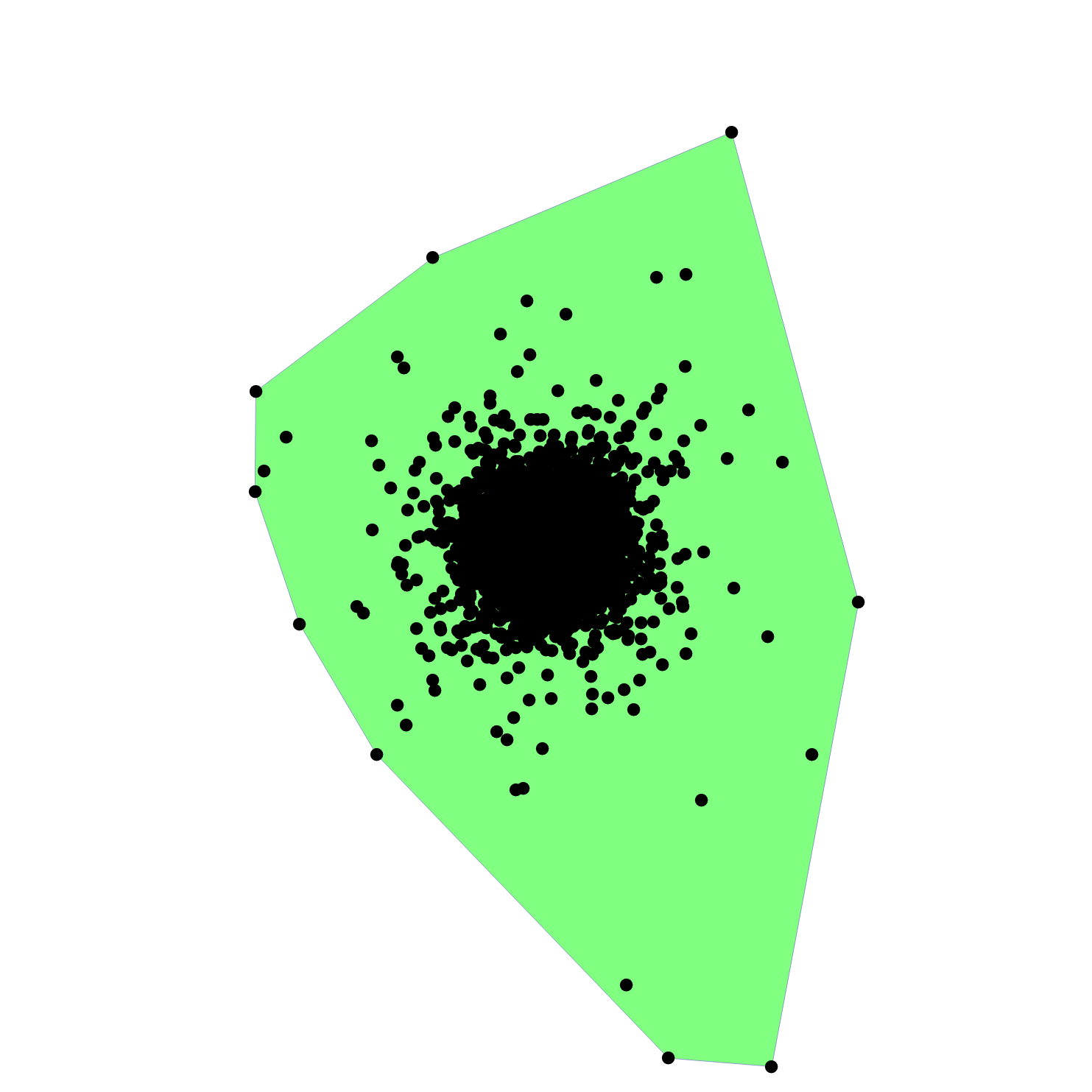}
\includegraphics[width=0.49\textwidth]{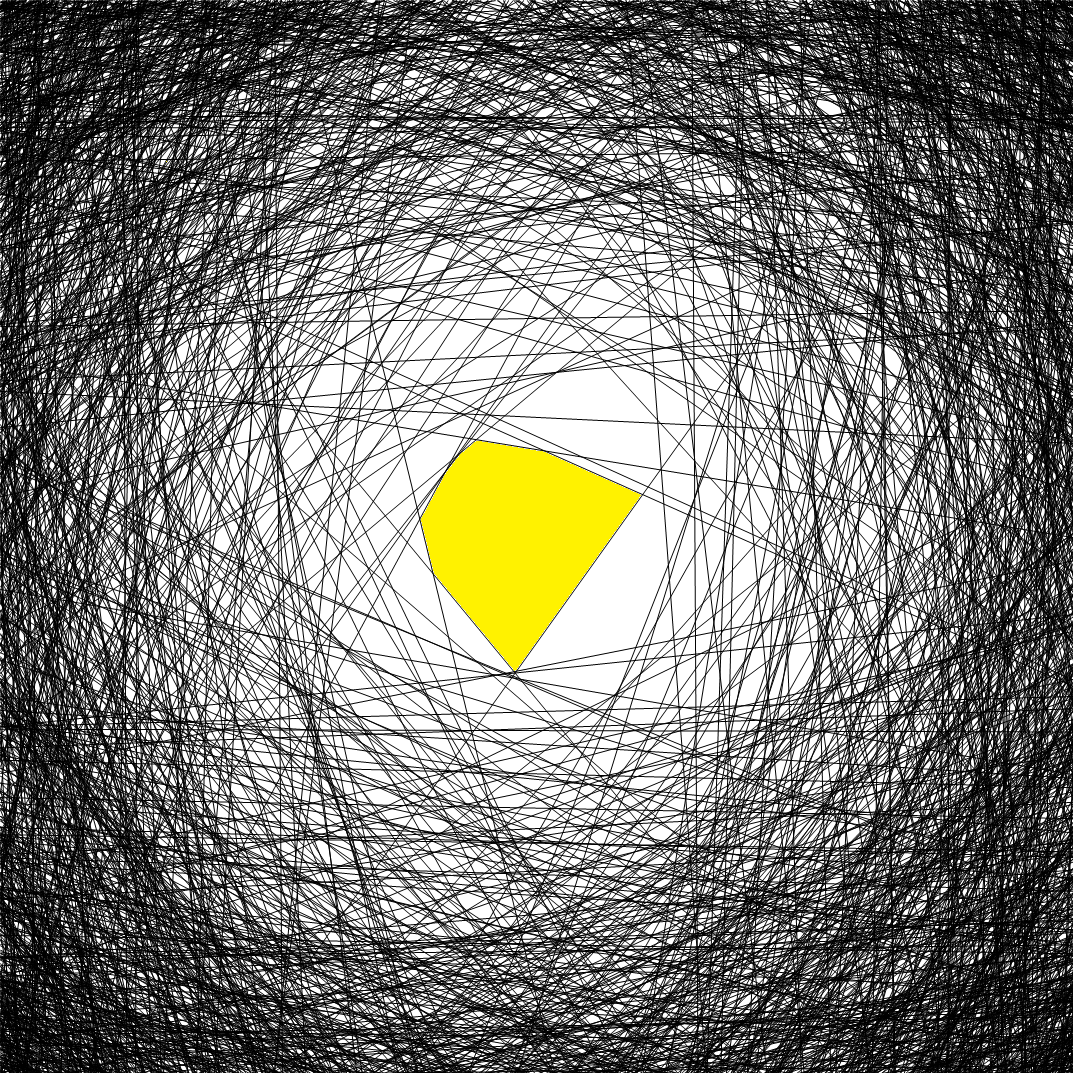}
\end{center}
\caption{
Left: The Poisson point process $\Pi_{2,3}$ on $\R^2$ with intensity $\|x\|^{-5}$, together with its convex hull. Right: The dual Poisson line tessellation, together with the corresponding zero cell.
}
\label{fig:poisson}
\end{figure}

\subsection{Poisson hyperplane tessellations}\label{subsec:PHT}
Using essentially convex duality, Poisson point processes can be transformed into Poisson hyperplane processes. To state this precisely,  fix a space dimension $d\geq 2$ as well as a parameter $\alpha>0$, the so-called \textit{distance exponent}. For an intensity parameter $\gamma\in(0,\infty)$ we define a $\sigma$-finite measure $\Theta_{d,\alpha,\gamma}$ on the affine Grassmannian $A(d,d-1)$ by
\begin{equation}\label{eq:Theta_def}
\Theta_{d,\alpha,\gamma}(\,\cdot\,)
:=
\frac {2\gamma}{\omega_d}\int_{\SS^{d-1}} \int_{0}^\infty \ind_{\{H(u,t)\in\,\cdot\,\}}\,t^{\alpha-1}\,\dint t \sigma(\dint u),
\end{equation}
where $H(u,t)$ is the hyperplane $H(u,t)=\{x\in\R^d:\langle x,u\rangle=t\}$ and $\sigma$ denotes the spherical Lebesgue measure on $\SS^{d-1}$ with total mass $\omega_d=2\pi^{d/2}/\Gamma(\frac d2)$.
Note that $\Theta_{d,1,1}$ coincides with the motion invariant Haar measure $\mu_{d-1}$ on $A(d,d-1)$ to be defined in~\eqref{eq:DefMeasureMuk} below.
In this paper, by a \textit{Poisson hyperplane process} with distance exponent $\alpha$ and intensity parameter $\gamma$ we understand a Poisson point process $\eta_{d,\alpha,\gamma}$ on the space $A(d,d-1)$ with intensity measure $\Theta_{d,\alpha,\gamma}$; see the right panel of Figure~\ref{fig:poisson}.  The random hyperplanes in $\eta_{d,\alpha,\gamma}$ dissect $\R^d$ into almost surely countably many random convex polyhedra, which are called cells in the sequel. The collection of these random polyhedra is known as a \textit{Poisson hyperplane tessellation}. Our focus lies on the \textit{zero cell}
$$
Z_{d,\alpha,\gamma} := \bigcap_{H\in\eta_{d,\alpha,\gamma}} H^-
$$
of such a random tessellation, where for a hyperplane $H\in A(d,d-1)$ we denote by $H^-$ the closed half-space determined by $H$ that contains the origin. We emphasize that the probability law of $Z_{d,\alpha,\gamma}$ is invariant under rotations and that $Z_{d,\alpha,\gamma}$ is almost surely bounded and hence a random polytope. Zero cells of Poisson hyperplane tessellations of this type have attracted considerable attention in the literature, see \cite{HoermannHugReitznerThaele,HugSchneider07LargeCells} as well as the references cited therein. In particular, this class contains two prominent special cases. Namely, $Z_{d,1,\gamma}$ corresponds to the zero cell of a stationary and isotropic Poisson hyperplane tessellation with intensity
$$
{1\over 2}\E\sum_{H\in\eta_{d,\alpha,\gamma}}\ind_{\{H\cap\BB^d\neq\varnothing\}} = {1\over 2}\Theta_{d,1,\gamma}(\{H\in A(d,d-1):H\cap\BB^d\neq\varnothing\}) = {\gamma\over \omega_d}\int_{\SS^{d-1}}\int_0^1 \dd t\sigma(\dint u) = \gamma
$$
(see \cite[Equation (4.27)]{SW08}), while $Z_{d,d,\gamma}$ has the same distribution as the typical cell of a stationary Poisson--Voronoi tessellation of a suitable constant intensity that can be expressed in terms of $\gamma$ and $d$. Both models are classical objects in stochastic geometry and well studied; we refer to \cite{LastPenrosePPPBook,SW08} for further background material.

It is a crucial observation that the zero cells $Z_{d,\alpha,\gamma} $ are \textit{dual} to convex hulls of Poisson point processes of the type discussed in Section~\ref{subsec:PPP}. To make this precise, we recall from \cite[Chapter 5.1]{MatousekDiscreteGeometryBook} that if $K\subset\R^d$ is a convex body, its dual (or polar body) $K^\circ$ is defined as
$$
K^\circ:=\{y\in\R^d:\langle x,y\rangle\leq 1\text{ for all }x\in K\}.
$$
In particular, if $P\subset\R^d$ is a polytope with $0$ in its interior, it is well known that, for all $k\in\{0,1,\ldots,d-1\}$,
\begin{equation}\label{eq:DualityFvector}
f_k(P) = f_{d-k-1}(P^\circ),
\end{equation}
see~\cite[Corollary~2.13]{ziegler_book_lec_on_poly}. 

To state the next theorem we recall from Section~\ref{subsec:PPP} that $\Pi_{d,\alpha}$ denotes a Poisson point process on $\R^d\setminus\{0\}$ with power-law intensity function $x\mapsto\|x\|^{-d-\alpha}$.

\begin{theorem}\label{thm:PoissonHyperplanes}
Fix $\alpha>0$. Then, $Z_{d,\alpha,\omega_d/2}^\circ$ has the same distribution as $\conv\Pi_{d,\alpha}$. Also,
$$
\E f_k(Z_{d,\alpha,\gamma}) = \E f_{d-k-1}(\conv\Pi_{d,\alpha})
$$
for every $\gamma\in(0,\infty)$.
\end{theorem}

\begin{remark}
The first part of Theorem \ref{thm:PoissonHyperplanes} can be rephrased as follows. The polytope $Z_{d,\alpha,1}^\circ$ has the same distribution as the convex hull of a Poisson point process on $\R^d\setminus\{0\}$ whose intensity function is given by $x\mapsto{2\over\omega_d}\|x\|^{-d-\alpha}$. More generally, it becomes clear from the proof of Theorem~\ref{thm:PoissonHyperplanes} that $Z_{d,\alpha,\gamma}^\circ$ has the same distribution as the convex hull of a Poisson point process on $\R^d\setminus\{0\}$ with intensity function $x\mapsto{2\gamma\over\omega_d}\|x\|^{-d-\alpha}$, for all $\gamma\in(0,\infty)$.
\end{remark}

Theorem~\ref{thm:PoissonHyperplanes} together with Theorem~\ref{theo:f_vect_poisson} yields an explicit description of the expected $f$-vector of the zero cells $Z_{d,\alpha,\gamma}$ for any $\gamma\in(0,\infty)$.
For example, Remark~\ref{rem:f_d-1_d-2} yields
\begin{align}
\E f_0(Z_{d,\alpha,\gamma})
&= \E f_{d-1}(\conv\Pi_{d,\alpha})
= \frac 2d (\sqrt \pi \alpha)^{d-1} \frac{\Gamma\left(\frac{d\alpha + 1}{2}\right) \Gamma\left(\frac \alpha 2\right)^{d}}{\Gamma\left(\frac{d\alpha}{2}\right)\Gamma\left(\frac{\alpha+1}{2}\right)^{d}},
\label{eq:E_f_0_Z_alpha1}\\
\E f_1(Z_{d,\alpha,\gamma})
&= \frac d2 \E f_0(Z_{d,\alpha,\gamma}),\label{eq:E_f_0_Z_alpha2}
\end{align}
whereas the formulae for the remaining components are more complicated and involve terms of the form $\tilde J_{m,d-k}(\beta)$, namely
$$
\E f_k(Z_{d,\alpha,\gamma})
= 2\sum_{\substack{m\in \{d-k,\ldots,d\}\\ m\equiv d \Mod{2}}} {\Gamma({m\alpha+1\over 2})\over\Gamma({m\alpha\over 2})}\left({\Gamma({\alpha\over 2})\over\Gamma({\alpha+1\over 2})}\right)^{m}{(\sqrt{\pi}\alpha)^{m-1}\over m}{m\choose d-k}\tilde{J}_{m,d-k}\left({m-1+\alpha\over 2}\right).
$$
This adds to the existing literature, where only  formulae for $\E f_0(Z_{d,\alpha,\gamma})$ with $\alpha\in \{1,d\}$ (and, since $Z_{d,\alpha,\gamma}$ is a simple polytope with probability one, also for $\E f_1(Z_{d,\alpha,\gamma})$) are available; see~\cite[Theorem 10.4.9]{SW08} and~\cite[Theorem~7.2]{moller}.
Also, in~\cite[Corollary 3.3]{HoermannHugReitznerThaele} a formula for $\E f_0(Z_{d,\alpha,\gamma})$ with general $\alpha>0$ was given in terms of a certain multiple integral which was not clear how to evaluate.

Additionally to the above formulae for $\E f_k(Z_{d,\alpha,\gamma})$, we claim that for the zero cell of a stationary and isotropic Poisson hyperplane tessellation in $\R^d$ it holds that
\begin{equation}\label{eq:E_f_d-2}
\E f_{d-2} (Z_{d,1,\gamma}) = \E f_{1}(\conv\Pi_{d,1}) = \frac{1}{2} \binom {d+1}{3} \pi^2.
\end{equation}
Indeed, while the first equality is a particular case of Theorem~\ref{thm:PoissonHyperplanes}, the second one was obtained in~\cite[Theorem~2.4, Remark~2.5]{convex_hull_sphere} by combining a formula from~\cite{barany_etal} with an Efron-type identity proved in~\cite[Theorem~2.8]{convex_hull_sphere}.

Since the expected intrinsic volumes $\E V_{k}(Z_{d,1,\gamma})$ are proportional to $\E f_{d-k}(Z_{d,1,\gamma})$ by an identity due to Schneider~\cite[p.~693]{SchneiderWeightedFaces}, the above yields also formulae for $\E V_k(Z_{d,1,\gamma})$. 

\subsection{Asymptotic results for Poisson hyperplane tessellations}\label{subsec:PHT_asympt}
Next, we shall consider for fixed $k\in\{0,1,2,\ldots\}$ the asymptotic behaviour of $\E f_k(Z_{d,\alpha,\gamma})$, as $d\to\infty$. Since $f_k(Z_{d,\alpha,\gamma})$ is independent from $\gamma\in(0,\infty)$, we just take $\gamma=1$ and write $Z_{d,\alpha}$ for $Z_{d,\alpha,1}$ from now on. While this has already been investigated in \cite{HoermannHugReitznerThaele} on a logarithmic scale, we are able to prove \textit{exact} asymptotic formulae, which strengthen these results. We shall write $a_d\sim b_d$ for two sequences $(a_d)_{d\in\N}$ and $(b_d)_{d\in\N}$ whenever $a_d/b_d\to 1$, as $d\to\infty$.

\begin{theorem}\label{thm:PoissonHyperplanesDtoInfinity}
Fix $k\in\{0,1,2,\ldots\}$. If the distance exponent $\alpha=\alpha(d)$ is such that 
$\inf_{d\in\N} \alpha(d) > 0$, then
$$
\E f_k(Z_{d,\alpha}) \sim {\sqrt{\alpha}\over 2^{k-{1\over 2}}}\left({\Gamma({\alpha\over 2})\over\Gamma({\alpha+1\over 2})}\right)^{d}{(\sqrt{\pi}\,\alpha)^{d-1}\over k!}d^{k-{1\over 2}},
\qquad d\to\infty.
$$
\end{theorem}

We remark that Theorem~\ref{thm:PoissonHyperplanesDtoInfinity} is consistent with Theorems~1.2 and~3.21 of~\cite{HoermannHugReitznerThaele}, which yield the limit relation
$$
\lim_{d\to\infty}\sqrt[d]{\E f_k(Z_{d,\alpha})} = {\sqrt{\pi}\alpha\Gamma({\alpha\over 2})\over\Gamma({\alpha+1\over 2})}
$$
for a fixed $\alpha>0$ (corresponding in the special case that $\alpha=1$ to the zero-cell of a stationary and isotropic Poisson hyperplane tessellation) as well as
$$
\lim_{d\to\infty}d^{-{1\over 2}}\sqrt[d]{\E f_k(Z_{d,d})} = \sqrt{2\pi}
$$
in the case that $\alpha=d$ (which corresponds to the typical cell of a stationary Poisson--Voronoi tessellation). Of course, both relations easily follow from Theorem~\ref{thm:PoissonHyperplanesDtoInfinity} as well, but Theorem~\ref{thm:PoissonHyperplanesDtoInfinity} is in fact much more precise. For example,  we obtain the asymptotic formulae
$$
k! \,\E f_k(Z_{d,\alpha}) \sim
\begin{cases}
\pi^{d-\frac 12} \left(\frac d 2\right)^{k-\frac 12},
&\text{ if } \alpha=1,\\
\eee^{1/4} {2^{{d+1\over 2}-k}\,\pi^{d-1\over 2}}\,d^{{d\over 2}+k-1},
&\text{ if } \alpha=d,
\end{cases}
$$
for any fixed $k\in\{0,1,2,\ldots\}$. The term $\eee^{1/4}$ in the second line appears because of the expansion
$$
\frac{\Gamma({d\over 2})}{\Gamma({d+1\over 2})} = \sqrt{\frac 2 d} \left(1 + \frac{1+o(1)}{4d} \right), \qquad d\to\infty.
$$

\subsection{Organization of the paper}
The rest of the paper is organized as follows.
In Section~\ref{sec:facts} we introduce the necessary notation and recall some facts from stochastic and integral geometry. Section~\ref{sec:properties_beta} contains the canonical decomposition for beta and beta' distribution which is of major importance in our proofs, which in turn are collected in Section~\ref{sec:proofs}.

\section{Notation and facts from stochastic and integral geometry}\label{sec:facts}

\subsection{General notation}

For $d\geq 1$ we let $\R^d$ be the $d$-dimensional Euclidean space with the standard scalar product $\lan\,\cdot\,,\,\cdot\,\ran$ and the associated norm $\|\,\cdot\,\|$. We let $\BB^d=\{x\in\R^d:\|x\|\leq 1\}$ be the Euclidean unit ball and $\SS^{d-1}=\{x\in\R^d:\|x\|=1\}$ be the corresponding $(d-1)$-dimensional unit sphere. Let also $\lambda_d$ denote the $d$-dimensional Lebesgue measure and $\sigma$ denote the spherical Lebesgue measure on $\SS^{d-1}$ which is normalized in such a way that $\sigma(\SS^{d-1})=\omega_d:={2\pi^{d/2}\over\Gamma(d/2)}$.

The \textit{convex} (respectively, \textit{positive}, \textit{linear}, \textit{affine}) \textit{hull} of a set $A\subset\R^d$ is the smallest convex set (respectively, convex cone, linear subspace, affine subspace) containing the set $A$ and is denoted by $\conv A$ (respectively, $\pos A$, $\lin A$, $\aff A$). The convex hull of  finitely  many points $x_1,\ldots,x_n$ is also denoted by $[x_1,\ldots,x_n]$.

We let $(\Omega,\cF,\P)$ be our underlying probability space, which we implicitly assume to be rich enough to carry all the random objects we consider. Expectation (i.e.\ integration) with respect to $\P$ is denoted by $\E$. For two random variables $X$ and $Y$ we write $X\overset{d}{=}Y$ if $X$ and $Y$ have the same probability law. Moreover, for random variables $X,X_1,X_2,\ldots$ we shall write $X_n\overset{d}{\longrightarrow}X$ if $X_n$ converges to $X$ in distribution, as $n\to\infty$.

\subsection{Polytopes and their faces}
A \textit{polytope} is a convex hull of finitely many points, while a \textit{polyhedron} is an intersection of finitely many closed half-spaces. We recall that a bounded polyhedron is also a polytope. The dimension $\dim P$ of a polyhedron $P$ is the dimension of its affine hull $\aff P$.
The \textit{$f$-vector} of a $d$-dimensional polyhedron $P\subseteq\R^d$ is defined by
$$
\mathbf{f} (P) := (f_0(P), \ldots,f_{d-1}(P)),
$$
where $f_k(P)$ is the number of $k$-dimensional faces of $P$.
The set of $k$-dimensional faces of a polyhedron $P$ is denoted by $\cF_k(P)$, so that $f_k(P)$ is the cardinality of $\cF_k(P)$.

\subsection{Grassmannians and the Blaschke--Petkantschin formula}

We denote by $G(d,k)$, respectively $A(d,k)$, the set of $k$-dimensional linear, respectively affine, subspaces of $\R^d$, where $k\in \{0,\ldots,d\}$.
The unique probability measure on $G(d,k)$ which is invariant under the action of the orthogonal group ${\rm SO}(d)$ is denoted by $\nu_k$. The affine Grassmannian $A(d,k)$ is endowed with the infinite measure $\mu_k$ defined by
\begin{equation}\label{eq:DefMeasureMuk}
\mu_k(\,\cdot\,) := \int_{G(d,k)}\int_{L^\bot}\ind_{\{L+x\in\,\cdot\,\}}\,\lambda_{L^\bot}(\dd x) \nu_{k}(\dd L),
\end{equation}
where $L^\bot$ is the orthogonal complement of $L$ and $\lambda_{L^\bot}$ is the Lebesgue measure on $L^\bot$, see~\cite[pp.~168--169]{SW08}.

The next theorem, to be found in~\cite[Theorem 7.2.7]{SW08}, allows to replace integration over all $k$-tuples of points in $\R^d$ by the double integration first over all $(k-1)$-dimensional affine subspaces $E$ and then over all $k$-tuples inside $E$. An important feature is the appearance of a term involving $\Delta(x_1,\ldots,x_k)$, the $(k-1)$-dimensional volume of the simplex $[x_1,\ldots,x_k]$.

\begin{proposition}[Affine Blaschke--Petkantschin formula]\label{theo:blaschke_petk}
For all $k\in \{1,\ldots,d+1\}$ and every non-negative Borel function $f:(\R^d)^{k}\to\R$ we have
\begin{multline*}
\int_{(\R^d)^{k}}f(x_1,\ldots,x_k) \, \lambda_d^k(\dd(x_1,\ldots,x_k))
\\
= B(d,k) \int_{A(d,k-1)}\int_{E^{k}}f(x_1,\ldots,x_k)\,\Delta^{d-k+1}(x_1,\ldots,x_k)\, \lambda_E^k(\dd(x_1,\ldots,x_k))
\mu_{k-1}(\dd E).
\end{multline*}
Here, $\lambda_E$ is the Lebesgue measure on the affine subspace $E$, and
$$
B(d,k) = ((k-1)!)^{d-k+1}\,{\omega_{d-k+2}\cdots\omega_d\over\omega_1\cdots\omega_{k-1}},
\qquad
B(d,1) = 1.
$$
\end{proposition}

\subsection{Cones and solid angles}
In this paper, the term cone always refers  to a \textit{polyhedral cone}, that is an intersection of finitely many closed half-spaces whose boundaries pass through the origin. In particular any polyhedral cone is a polyhedron. The \textit{solid angle} of a cone $C\subset \R^d$ is defined as
$$
\alpha(C) := \P[N\in C],
$$
where $N$ is a random vector having a standard normal distribution on the linear hull of $C$. The \textit{polar} (or \textit{dual}) \textit{cone} of $C$ is defined by
$$
C^\circ := \{v\in \R^d\colon \langle v, z\rangle\leq 0 \text{ for all } z\in C\}.
$$
The \textit{tangent cone} $T(F,P)$ at a face $F$ of a full-dimensional polytope $P\subseteq\R^d$ is defined as
$$
T(F,P) := \{v\in\R^d\colon  x_0 + v\eps \in P \text{ for some } \eps>0\},
$$
where $x_0$ is any point in the relative interior of $F$ (the definition does not depend on the choice of $x_0$). The \textit{normal cone} of $F$ is the polar to the tangent cone, that is
$$
N(F,P) := T^\circ (F,P) = \{v\in \R^d\colon \langle v, z-x_0\rangle \leq 0 \text{ for all } z\in P\}.
$$
The \textit{internal} and \textit{external} angles at a face $F$ of $P$ are defined as the solid angles of the tangent and the normal cones, respectively:
$$
\beta(F,P) := \alpha(T(F,P)), \qquad \gamma(F,P) := \alpha (N(F,P)).
$$
For further background material we refer, for example, to \cite{AmelunxenLotzDCG17,GruenbaumGA,GruenbaumBook}.

\subsection{Conic intrinsic volumes and Grassmann angles}

In this section we recall the definitions of the conic intrinsic volumes and Grassmann angles of cones and refer to \cite{AmelunxenLotzDCG17,ALMT14,glasauer_phd,SW08} for further information.
For a polyhedral cone $C\subset\R^{d}$ we denote by $\cF_k(C)$ the set of its $k$-dimensional faces,  where $k\in\{0,\ldots,d\}$.  Note that $C$ is the disjoint union of the relative interiors of its faces, where the \textit{relative interior} $\relint F$ of a face $F$ is the interior of $F$ with respect to its affine hull $\aff F$ as the ambient space.

If $x\in\R^{d}$ is a point,  we let $\pi_C(x)$ denote the \textit{metric projection} of $x$ onto $C$, that is the uniquely determined point $y\in C$ minimizing the distance $\|x-y\|$.
For $k\in\{0,\ldots,d\}$, the $k$-th \textit{conic intrinsic volume} $\upsilon_k(C)$ is defined by
$$
\upsilon_k(C) := \sum_{F\in\cF_k(C)} \P[\pi_C(N)\in\relint(F)],
$$
where $N$ is a standard Gaussian random vector in $\R^{d}$.  If $\cF_k(C)=\varnothing$, we define $\upsilon_k(C):=0$.  In other words, $\upsilon_k(C)$ is the probability that the metric projection of $N$ to $C$ lies in the relative interior of a $k$-dimensional face of $C$.  For convenience also define $\upsilon_k(C):=0$ for all integers $k\notin\{0,\ldots,d\}$.
For example, if $C$ is a $k$-dimensional linear subspace, then $\upsilon_{k}(C)=1$, while all other conic intrinsic volumes vanish.

By definition, the conic intrinsic volumes are non-negative and their sum equals one. Moreover, they satisfy the so-called \textit{Gauss--Bonnet formula} \cite[Equation~(5.3)]{ALMT14}
\begin{equation}\label{eq:gauss_bonnet}
\upsilon_0(C)+ \upsilon_2(C) + \ldots = \upsilon_1(C)+ \upsilon_3(C) +\ldots = \frac 12,
\end{equation}
provided $C$ is not a linear subspace. Observe that  $\upsilon_d(C)$ is just the solid angle of $C$, provided that $\dim\aff C=d$.
The so-called \textit{half-tail functionals}~\cite{ALMT14} of a polyhedral cone $C\subset\R^d$ are defined as
\begin{equation}\label{eq:h_def}
h_{k}(C) := \upsilon_{k}(C) + \upsilon_{k+2}(C) + \ldots, \quad k\in \{0,\ldots, d\}.
\end{equation}
Note that the above sums only contain finitely many non-zero terms. It is also natural to put
$$
h_{d+1}(C) := h_{d+2}(C) := \ldots := 0.
$$
If $C$ is not a linear subspace, the \textit{conic Crofton formula} states that
\begin{equation}\label{eq:ConicalIntVolGrassmannAngle}
h_{k}(C) = {1\over 2}\P[C\cap L_{d+1-k}\neq\{0\}], \qquad k\in\{1,\ldots,d+1\},
\end{equation}
where $L_{d+1-k}\in G(d,d+1-k)$ is a random linear subspace distributed according to the probability measure $\nu_{d+1-k}$;
see~\cite[p.~257]{mcmullen}, \cite[pages~261--262]{SW08} or~\cite[Equation (2.10)]{AmelunxenLotzDCG17}.
The numbers on the right-hand side of~\eqref{eq:ConicalIntVolGrassmannAngle} are called the \textit{Grassmann angles} of the cone $C$ and were introduced by Gr\"unbaum~\cite{GruenbaumGA}.




\subsection{Random projections of polytopes}

Let $P\subset\R^N$ be a polytope and $L_d\in G(N,d)$ be a random subspace distributed according to the probability measure $\nu_d$, where $d\in \{1,\ldots,N\}$. Then $\Pi_{d}P$ stands for the random polytope in $L_d$ that arises as the orthogonal projection of $P$ onto $L_d$. The next result we recall is due to \citet{AS92}, its proof is based on the conic Crofton formula~\eqref{eq:ConicalIntVolGrassmannAngle}. It says that the expected $f$-vector of the random polytope $\Pi_d P$ can be expressed in terms of the internal and external angles of the original polytope $P$. We emphasize that the sum on the right hand side of \eqref{eq:AffSchn} below only contains finitely many non-zero terms.

\begin{proposition}[Expected $f$-vectors of random projections]\label{theo:affentranger_schneider}
Let $P\subset \R^N$ be a polytope. Then, for all $d\in \{1,\ldots,\dim P\}$ and $k\in\{0,\ldots,d-1\}$,
\begin{equation}\label{eq:AffSchn}
\E f_k(\Pi_d P) = 2 \sum_{s=0}^\infty \sum_{G\in \cF_{d-1-2s}(P)} \gamma(G,P) \sum_{F\in \cF_k(G)} \beta(F,G).
\end{equation}
\end{proposition}

\section{Properties of beta and beta' distributions}\label{sec:properties_beta}

\subsection{Identification of affine subspaces}\label{subsec:identification}
Sometimes it will be convenient to identify every affine subspace of $\R^d$ with the Euclidean space of the corresponding dimension. To make this precise, we recall that $A(d,k)$ is the set of $k$-dimensional affine subspaces of $\R^d$. For an affine subspace $E \in A(d, k)$ we denote by $\pi_E: \R^d\to E$ the orthogonal projection onto $E$ and by $p(E)=\pi_E(0)=\argmin_{x\in E}\|x\|$ the projection of the origin on $E$. For every affine subspace $E\in A(d,k)$ let us fix an isometry $I_E:E\to \R^k$ such that $I_E(p(E))=0$.  The exact choice of the isometries $I_E$ is not important (essentially due to the rotational invariance of the beta and beta' distributions). We only require that $(x,E) \mapsto I_E(\pi_E(x))$ defines  a Borel measurable map from $\R^d \times A(d,k)$ to $\R^k$, where we supply $\R^d$, $\R^k$ and $A(d,k)$ with their standard Borel $\sigma$-algebras; see~\cite[Chapter 13.2]{SW08} for the case of $A(d,k)$.


\subsection{Projections and distances}
The next lemma, taken from  \cite[Lemma 4.3]{beta_polytopes}, states that the beta and beta$^\prime$-distributions on $\R^d$ yield distributions of the same type (but with different parameters) when projected orthogonally onto arbitrary linear subspaces.

\begin{lemma}[Orthogonal projections]\label{lem:projection}
Denote by $\pi_L: \R^d\rightarrow L$ the orthogonal projection onto a $k$-dimensional linear subspace $L\in G(d,k)$, where $k\in \{1,\ldots,d\}$.
	\begin{itemize}
		\item[(a)] If the random point $X$ has distribution $f_{d,\beta}$ for some $\beta\geq -1$, then $I_L(\pi_L (X))$ has distribution $f_{k,\b+\frac{d-k}{2}}$.
		\item[(b)] If the random point $X$ has density $\tilde{f}_{d,\beta}$ for some $\beta>\frac d2$, then $I_L(\pi_L(X))$ has density $\tilde{f}_{k,\b-\frac{d-k}{2}}$.
	\end{itemize}
\end{lemma}

Note that the case $\beta = -1$ is included in Part (a) and follows from the case $\beta>-1$ by weak continuity. In fact, as $\beta\downarrow -1$, the beta distribution on $\R^d$ converges weakly to the uniform distribution on the unit sphere $\SS^{d-1}$.

The next lemma describes the distribution of the squared norm of a random vector with $d$-dimensional beta or beta' distribution. The squared norm turns out to have the usual, one-dimensional beta or beta' distribution.
Recall that a random variable has a classical \textit{beta distribution} with parameters $\alpha_1>0, \alpha_2>0$, denoted by $\Beta(\alpha_1,\alpha_2)$, if its Lebesgue density on $\R$ is
$$
g_{\alpha_1,\alpha_2}(t) = \frac{\Gamma(\alpha_1+\alpha_2)}{\Gamma(\alpha_1)\Gamma(\alpha_2)} t^{\alpha_1-1} (1-t)^{\alpha_2-1}\ind_{\{0<t<1\}}, \qquad t\in \R.
$$
Similarly, a random variable has a classical \textit{beta' distribution} with parameters $\alpha_1>0$, $\alpha_2>0$, denoted by $\Beta'(\alpha_1,\alpha_2)$, if its Lebesgue density on $\R$ is
$$
\tilde g_{\alpha_1,\alpha_2}(t) = \frac{\Gamma(\alpha_1+\alpha_2)}{\Gamma(\alpha_1)\Gamma(\alpha_2)} t^{\alpha_1-1} (1+t)^{-\alpha_1-\alpha_2}\ind_{\{t>0\}}, \qquad t\in \R.
$$
Observe that, up to reparametrization and rescaling, $\Beta'(\alpha_1,\alpha_2)$ coincides with the Fisher--Snedecor $F$-distribution.  The following fact can be directly verified using polar integration.

\begin{lemma}[Squared norm]\label{lem:squared_norm}
Let $X$ be a random vector in $\R^d$.
\begin{itemize}
\item[(a)] If $X$ has the beta density $f_{d,\beta}$ with $\beta>-1$, then $\|X\|^2 \sim \text{Beta}(\frac d2, \beta + 1)$.
\item[(b)] If $X$ has the beta' density $\tilde f_{d,\beta}$ with $\beta>\frac d2$, then $\|X\|^2 \sim \text{Beta}'(\frac d2, \beta - \frac{d}{2})$.
\end{itemize}
\end{lemma}

\subsection{Canonical decomposition of Ruben and Miles}
Let $X_1,\ldots,X_k$ be i.i.d.\ random points in $\R^d$ with the beta distribution $f_{d,\beta}$. Let $k \in \{1,\ldots, d+1\}$, so that $[X_1,\ldots,X_k]$ is a simplex.  We need a description of the positions of these points inside their own affine hull $A= \aff(X_1,\ldots,X_k)$, together with the position of $A$ inside $\R^d$. The next theorem is due to Ruben and Miles~\cite{ruben_miles}. Since this result is of central importance for what follows and since in \cite{ruben_miles} a different notation is used, we give a streamlined proof.

\begin{theorem}[Canonical decomposition in the beta case]\label{theo:ruben_miles}
Let $X_1,\ldots,X_k$ be i.i.d.\ random points in $\R^d$ with density $f_{d,\beta}$, where $\beta>-1$ and $k\in \{1,\ldots, d+1\}$. Let $A=\aff(X_1,\ldots,X_k)$ be the affine subspace spanned by $X_1,\ldots,X_k$. Let also $p(A)$ be the orthogonal projection of the origin on $A$ and let $h(A) = \|p(A)\|$ denote the distance from the origin to $A$. Observe that $A\cap \BB^d$ is a $(k-1)$-dimensional  ball of radius $\sqrt{1-h^2(A)}$ and consider the points
$$
Z_i := \frac{I_A(X_i)}{\sqrt{1-h^2(A)}} \in \BB^{k-1}, \quad i=1,\ldots,k.
$$
Then,
\begin{itemize}
\item[(a)] The joint Lebesgue density of the random vector $(Z_1,\ldots,Z_k)$ is a constant multiple of
$$\Delta^{d-k+1} (z_1,\ldots,z_k)\prod_{i=1}^k f_{k-1,\beta}(z_i),$$
\item[(b)] the random vector $(Z_1,\ldots,Z_k)$ is stochastically independent of $A$, 
\item[(c)] the density of $I_{A^\bot}(p(A))\in \BB^{d-k+1}$ is $f_{d-k+1, \frac{(k-1)(d+1)}{2} + k \beta}$,
\item[(d)] $I_{A^\bot}(p(A))$ is stochastically independent of $A^\bot$.
\end{itemize}
\end{theorem}
\begin{remark}
For $k=1$, Part~(c) states that $p(A) = X_1$ has density $f_{d,\beta}$ (which is trivial), whereas the other assertions are empty statements.
\end{remark}
\begin{proof}[Proof of Theorem~\ref{theo:ruben_miles}.]
Let $\varphi:\R^k\to [0,\infty)$ and $\psi:A(d,k-1)\to[0,\infty)$ be Borel measurable functions.  We are interested in the following quantity:
\begin{align*}
B_{\varphi,\psi}
:&=
\E \left[\varphi\left(\frac{I_A(X_1)}{\sqrt{1-h^2(A)}},\ldots,\frac{I_A(X_{k})}{\sqrt{1-h^2(A)}}\right) \psi(A)\right]
\\
& =
\int_{(\R^d)^k} \varphi\left(\frac{I_A(x_1)}{\sqrt{1-h^2(A)}},\ldots,\frac{I_A(x_{k})}{\sqrt{1-h^2(A)}}\right) \psi(A)
\left(\prod_{i=1}^k f_{d,\beta}(x_i)\right) \left(\prod_{i=1}^k\lambda_d(\dd x_i)\right),
\end{align*}
where $A$ is used to denote $\aff (x_1,\ldots,x_k)$ without risk of confusion.
For the rest of the proof, let $C_1,C_2,\ldots$ be constants depending only on $d,k,\beta$. By the affine Blaschke--Petkantschin formula stated in Proposition~\ref{theo:blaschke_petk}, we have
\begin{multline*}
B_{\varphi,\psi}
=
C_1 \int_{A(d,k-1)}\int_{A^k} \varphi\left(\frac{I_A(x_1)}{\sqrt{1-h^2(A)}},\ldots,\frac{I_A(x_{k})}{\sqrt{1-h^2(A)}}\right) \psi(A)
\\
\times \Delta^{d-k+1}(x_1,\ldots,x_k) \left(\prod_{i=1}^k f_{d,\beta}(x_i)\right) \left(\prod_{i=1}^k\lambda_A(\dd x_i)\right) \mu_{k-1}(\dd A).
\end{multline*}
Using the substitution $y_i=I_A(x_i)\in\R^{k-1}$,  $1\leq i\leq k$, recalling that $I_A:A\to \R^{k-1}$ is an isometry such that $I_A(p(A)) =0$ and observing that  $\|x_i\|^2 = h^2(A) + \|y_i\|^2$, we arrive at
\begin{align*}
B_{\varphi,\psi}&=
C_2 \int_{A(d,k-1)} \int_{(\R^{k-1})^k} \varphi\left(\frac{y_1}{\sqrt{1-h^2(A)}},\ldots,\frac{y_{k}}{\sqrt{1-h^2(A)}}\right)  \psi(A)
\\
&  \times\Delta^{d-k+1}(y_1,\ldots,y_k)   \left( \prod_{i=1}^k (1-h^2(A)-\|y_i\|^2)^\beta\,\ind_{\{h^2(A)+\|y_i\|^2<1\}}\right)  \left(\prod_{i=1}^k \lambda_{k-1}(\dd y_i)\right) \mu_{k-1}(\dd A),
\end{align*}
where we also used the definition \eqref{eq:def_f_beta} of the beta density.
Next, we apply the substitution $z_i = y_i/\sqrt{1-h^2(A)}\in \BB^{k-1}$, $1\leq i\leq k$, and write
\begin{align}
&(1-h^2(A)-\|y_i\|^2)^\beta
=
(1-h^2(A))^\beta \left(1- \frac{\|y_i\|^2}{1-h^2(A)}\right)^\beta
=
(1-h^2(A))^\beta \left(1-\|z_i\|^2\right)^\beta, \label{eq:self_similar_beta}\\
\nonumber &\lambda_{k-1}(\dint y_i) = (1-h^2(A))^{\frac 12 (k-1)} \lambda_{k-1}(\dint z_i), \\
\nonumber &\Delta(y_1,\ldots,y_k) =  (1-h^2(A))^{\frac {k-1}2} \Delta(z_1,\ldots,z_k),
\end{align}
to conclude that
\begin{multline*}
B_{\varphi,\psi}
=
C_3 \int_{A(d,k-1)} \int_{(\BB^{k-1})^k} \varphi(z_1,\ldots,z_k) \psi(A) \; (1-h^2(A))^{\frac 12 k(k-1) + \frac 12(d-k+1)(k-1) + k \beta} \ind_{\{h(A)<1\}}
\\
\times\Delta^{d-k+1}(z_1,\ldots,z_k)
\left(\prod_{i=1}^k  (1 - \|z_i\|^2)^\beta \right) \left(\prod_{i=1}^k \lambda_{k-1}(\dd z_i) \right)\mu_{k-1}(\dd A).
\end{multline*}
Finally, some elementary transformations including the use of~\eqref{eq:def_f_beta} lead to
\begin{multline*}
B_{\varphi,\psi}
=
C_4 \left(\int_{A(d,k-1)}  \psi(A)\;  (1-h^2(A))^{\gamma} \ind_{\{h(A)<1\}}\; \mu_{k-1}(\dd A)\right)
\\
\times \left(\int_{(\R^{k-1})^k} \varphi(z_1,\ldots,z_k)\Delta^{d-k+1} (z_1,\ldots,z_k)  \left(\prod_{i=1}^k f_{k-1,\beta}(z_i)\right) \left(\prod_{i=1}^k \lambda_{k-1}(\dd z_i) \right)\right),
\end{multline*}
where we used the notation
$$
\gamma:= \frac 12 k(k-1) + \frac 12(d-k+1)(k-1) + k \beta
=\frac {(k-1)(d+1)}2 + k\beta.
$$
The form of the second integral and the product structure of the formula  imply that the random points $Z_1,\ldots,Z_k$ have the required joint density and are independent of $A$, thus proving claims (a) and (b) of the theorem.

Next, we prove parts (c) and (d) of the theorem. To this end, we take $\varphi(z_1,\ldots,z_k)=1$ and write the above result as
$$
\E \psi(A) = C_5 \int_{A(d,k-1)}  \psi(A)\;  (1-h^2(A))^{\gamma} \ind_{\{h(A)<1\}}\; \mu_{k-1}(\dd A).
$$
Now we take $\psi(A) = \psi_1(I_{A^\bot}(p(A)))\, \psi_2 (A^\bot)$ for some Borel functions $\psi_1:\R^{d-k+1} \to [0,\infty)$ and $\psi_2: G(d,d-k+1) \to [0,\infty)$, so that the above identity takes the form
$$
\E \psi(A) = C_5 \int_{A(d,k-1)}  \psi_1(I_{A^\bot}(p(A)))\; \psi_2 (A^\bot)\;  (1-h^2(A))^{\gamma} \ind_{\{h(A)<1\}}\; \mu_{k-1}(\dd A).
$$
The definition of the measure $\mu_{k-1}$ on $A(d,k-1)$ given in~\eqref{eq:DefMeasureMuk} states that for every Borel function $f: A(d,k-1)\to [0,\infty)$ we have
$$
\int_{A(d,k-1)} f(A) \mu_{k-1}(\dint A)
=
\int_{G(d,k-1)}\int_{L^\bot}f(L+x) \lambda_{L^\bot}(\dd x) \;\nu_{k-1}(\dd L).
$$
Observing that for every $L\in G(d,k-1)$ and $x\in L^\bot$ we have $(L+x)^\bot = L^\bot$, $p(L+x) = x$ and $h(L+x) = \|x\|$, we arrive at
$$
\E \psi(A)
=
C_5 \int_{G(d,k-1)} \psi_2 (L^\bot) \left(\int_{L^\bot} \psi_1(I_{L^\bot}(x)) \;  (1-\|x\|^2)^{\gamma} \ind_{\{\|x\|<1\}}\; \lambda_{L^\bot}(\dd x)\right) \;\nu_{k-1}(\dd L).
$$
Writing $y:= I_{L^\bot}(x)\in \R^{d-k+1}$ and using that $\|y\| = \|x\|$ since $I_{L^\bot}:L^\bot\to \R^{d-k+1}$ is an isometry, we obtain
\begin{multline*}
\E \left[\psi_1(I_{A^\bot}(p(A))) \; \psi_2 (A^\bot)\right]
\\=
C_5 \left(\int_{G(d,k-1)} \psi_2 (L^\bot)\; \nu_{k-1}(\dd L)  \right) \left(\int_{\R^{d-k+1}} \psi_1(y)   \; (1-\|y\|^2)^{\gamma} \ind_{\{\|y\|<1\}}\; \dd y \right).
\end{multline*}
The product structure of the right-hand side implies that $I_{A^\bot}(p(A))$ and $A^\bot$ are independent, thus proving part (d) of the theorem. Taking $\psi_2 \equiv 1$, we arrive at
$$
\E \psi_1(I_{A^\bot}(p(A)))
=
C_5  \int_{\R^{d-k+1}} \psi_1(y) \;  (1-\|y\|^2)^{\gamma}\ind_{\{\|y\|<1\}} \; \dd y
=
C_6 \int_{\R^{d-k+1}} \psi_1(y) \; f_{d-k+1, \gamma}(y)\; \dd y.
$$
It follows that $I_{A^\bot}(p(A))$ has density $f_{d-k+1, \gamma}$ on $\R^{d-k+1}$, thus proving claim (c).
\end{proof}

\begin{remark}
Theorem~\ref{theo:ruben_miles} continues to hold for $\beta=-1$ (corresponding to the uniform distribution on the $(d-1)$-dimensional unit sphere), but in this case we have to replace (a) by
\begin{itemize}
\item[(a)] The joint distribution of $(Z_1,\ldots,Z_k)$ has density proportional to $\Delta^{d-k+1} (z_1,\ldots,z_k)$ with respect to the $d$-th power of the spherical Lebesgue measure on $\bS^{d-1}$.
\end{itemize}
\end{remark}

A result similar to Theorem~\ref{theo:ruben_miles} holds in the beta' case as well and is also due to Ruben and Miles~\cite{ruben_miles}. Since the proof is similar, we don't present the details.

\begin{theorem}[Canonical decomposition in the beta' case]\label{theo:ruben_miles_prime}
Let $X_1,\ldots,X_k$ be i.i.d.\ points in $\R^d$ with density $\tilde f_{d,\beta}$,  where $k\in \{1,\ldots, d+1\}$. Let $A=\aff(X_1,\ldots,X_k)$ be the affine subspace spanned by $X_1,\ldots,X_k$. Let also $p(A)$ be the orthogonal projection of the origin on $A$ and let and $h(A) = \|p(A)\|$ denote the distance from the origin to $A$. Consider the points
$$
Z_i := \frac{I_A(X_i)}{\sqrt{1 + h^2(A)}} \in \R^{k-1}, \quad i=1,\ldots,k.
$$
Then,
\begin{itemize}
\item[(a)] the joint Lebesgue density of the random vector $(Z_1,\ldots,Z_k)$ is proportional to
$$\Delta^{d-k+1} (z_1,\ldots,z_k)\prod_{i=1}^k \tilde f_{k-1,\beta}(z_i),$$
\item[(b)] the random vector $(Z_1,\ldots,Z_k)$ is stochastically independent of $A$,
\item[(c)] the density of $I_{A^\bot}(p(A))\in \R^{d-k+1}$ is $\tilde f_{d-k+1, k \beta - \frac{(k-1)(d+1)}{2}}$,
\item[(d)] $I_{A^\bot}(p(A))$ is stochastically independent of $A^\bot$.
\end{itemize}
\end{theorem}
\begin{proof}
The computations are analogous to those done in the proof of Theorem~\ref{theo:ruben_miles}, but instead of~\eqref{eq:self_similar_beta} we use the identity
\begin{equation*}
(1+h^2(A)+\|y_i\|^2)^{-\beta}
=
(1+h^2(A))^{-\beta} \left(1 + \frac{\|y_i\|^2}{1+h^2(A)}\right)^{-\beta}
=
(1+h^2(A))^{-\beta} \left(1 + \|z_i\|^2\right)^{-\beta}.
\end{equation*}
Correspondingly, in the formula for $B_{\varphi,\psi}$ the term $(1+h^2(A))^{-\tilde \gamma}$ appears, where $\tilde \gamma$ is given by
$\tilde \gamma = k\beta - \frac{(k-1)(d+1)}{2}$.
\end{proof}

Applying Lemma~\ref{lem:squared_norm} to $I_{A^\bot}(p(A))$, we obtain the following result, which is also contained in~\cite{beta_simplices} as Theorem 2.7.

\begin{corollary}[Distances to affine subspaces]\label{cor:distance_distr}
Let $X_1,\ldots,X_{k}$ be i.i.d.\ random points in $\R^d$,  where $k\in \{1,\ldots,d\}$, and denote by $h$ the distance from the origin to the affine subspace $\aff(X_1,\ldots,X_k)$ spanned by $X_1,\ldots,X_{k}$.
\begin{itemize}
\item[(a)] If $X_1,\ldots,X_k$ have the beta density $f_{d,\beta}$, then  $h^2(A)\sim \text{Beta}(\frac{d-k+1}{2},\frac{(k-1)(d+1)}{2} + k \beta + 1)$.
\item[(b)] If $X_1,\ldots,X_k$ have the beta' density $\tilde f_{d,\beta}$, then  $h^2(A)\sim \text{Beta}'(\frac{d-k+1}{2}, k(\beta - \frac d2))$.
\end{itemize}
\end{corollary}

\begin{remark}
A result similar to Theorems~\ref{theo:ruben_miles} and~\ref{theo:ruben_miles_prime} holds for the isotropic normal distribution in $\R^d$ if we define $Z_i=I_A(X_i)$. In this case, $I_{A^\bot}(p(A))$ has a standard normal distribution on $\R^{d-k+1}$.
\end{remark}

\subsection{Relation to the extreme-value theory}\label{subsec:extreme_pareto}
The aim of the present section is to show that the beta and the beta' distributions, together with the normal distribution considered as their limiting case, are in one-to-one correspondence with the generalized Pareto distributions. The latter are, in turn,  in one-to-one correspondence with the extreme-value distributions. According to the classical  Fisher--Tippett--Gnedenko theorem~\cite[Theorem~1.1.3]{de_haan_book}, there are $3$ possible families of extreme-value distributions: Gumbel, Fr\'echet and Weibull,  which, as we shall argue, correspond to the normal distribution, the beta' family and the beta family, respectively.
Since the results of this section will not be used in the sequel, we shall only sketch the arguments.

Consider a random vector $X$ in $\R^d$ whose density is a spherically symmetric function of the form $p(\|x\|)$, $x\in\R^d$.
The beta and beta' distributions as well as the normal distribution are characterized by the following remarkable property discovered by Miles~\cite[Section~12]{miles}. Namely, for every $h,r>0$ for which $p(h)>0$ the relation
\begin{equation}\label{eq:miles_functional_eq}
p\left(\sqrt{h^2+r^2}\right) = c_1(h) p\left(\frac{r}{c_2(h)}\right)
\end{equation}
holds, where $c_1(h)>0$ and $c_2(h)>0$ are certain functions.
That is, the restriction of the density to any affine hyperplane at distance $h$ from the origin has the same radial component as the original density, up to rescaling. This property is crucial for the proof of the canonical decomposition, recall~\eqref{eq:self_similar_beta}.

Miles~\cite[Section~12]{miles} solved the functional equation~\eqref{eq:miles_functional_eq} under additional smoothness assumptions on $p$.
Let us show how~\eqref{eq:miles_functional_eq} can be reduced to the classification of the generalized Pareto distributions which does not require smoothness. In fact, it is even possible to drop the assumption of the absolute continuity of $X$, but we refrain from doing this since this would lead to intransparent notation.
Consider the function $g(y) := p(\sqrt y)$. Then, \eqref{eq:miles_functional_eq}  takes the form
$$
g(h^2+r^2) = c_1(h) g\left(\frac{r^2}{c_2^2(h)}\right).
$$
Equivalently, with $a:=h^2$, $s:=r^2$ and with $\psi_1(a)= c_1(\sqrt a)$, $\psi_2(a)=c_2^2(\sqrt a)$, we have
$$
g(a+s) = c_1(\sqrt a) g\left(\frac{s}{c_2^2(\sqrt a)}\right) = \psi_1(a) g\left(\frac{s}{\psi_2(a)}\right).
$$
Since $\int_0^\infty g(y) \dd y = 2\int_0^\infty p(r) r \dd r <\infty$, provided we assume that $d\geq 2$, we can normalize $g$ to be a probability density. Let $Z$ be a random variable with density $g$. Then, the above equation can probabilistically be rewritten as
\begin{equation}\label{eq:pareto_definition}
Z-a \,|\, Z\geq a \eqdistr  \psi_2(a) Z \qquad\text{ for all } a>0 \text{ such that } \P[Z\geq a] >0.
\end{equation}
Non-degenerate distributions having this property are known as generalized Pareto distributions and appear in extreme-value theory as limit distributions for residual life given that the current age is high. These distributions were classified in~\cite[Theorem~2]{balkema_de_haan} and are in one-to-one correspondence with the extreme-value distributions; see Theorem 1.1.6 (in particular, Claim~4) in~\cite{de_haan_book}. There are three possible types of generalized Pareto distributions:
\begin{enumerate}
\item[(a)] the exponential distribution $g(y) = \text{const}\cdot \eee^{-\lambda y}$, $y>0$, with parameter $\lambda>0$, which corresponds to the normal distribution with radial component $p(r)=\text{const}\cdot\eee^{-\lambda r^2}$, $r>0$. Here and below $\textit{const}$ denotes a suitable normalization constant, which may change from occasion to occasion.
\item[(b)] the Pareto distribution of Weibull type $g(y) = \text{const}\cdot(1 - y/A)^\beta$, $0 < y < A$, where $\beta>-1$ and $A>0$ are parameters. They correspond to the beta-type densities with radial component $p(r)= \text{const}\cdot (1 - r^2/A)^\beta$, $0<r<\sqrt{A}$.
\item[(c)] the Pareto distribution of Fr\'echet type $g(y) = \text{const}\cdot (1 + y/A)^{-\beta}$, $y > 0$, where $\beta>1$ and $A>0$ are parameters. They correspond to the beta'-type densities with radial component $p(r)= \text{const}\cdot (1 + r^2/A)^{-\beta}$, $r>0$.
\end{enumerate}
Besides, the degenerate distribution, where $Z$ is a positive constant, also satisfies~\eqref{eq:pareto_definition}. The corresponding multivariate distribution is the uniform distribution on a sphere.

\section{Proofs}\label{sec:proofs}

\subsection{Expected external angles} \label{subsec:external_angles_proofs}
\begin{proof}[Proof of Theorem~\ref{theo:external} and Theorem~\ref{theo:external_prime}]
Since the proofs in the beta and beta' cases are similar, let us write $P$ for both $P_{n,d}^\beta$ and $\tilde P_{n,d}^\beta$. The following first part of the proof applies to both cases.

\begin{figure}[ht]
\begin{center}
\includegraphics[width=0.8\textwidth, height=0.6\textwidth ]{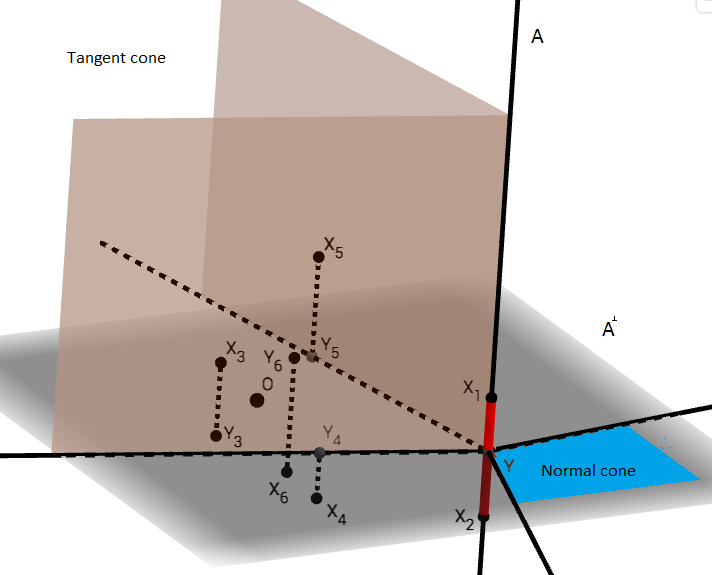}
\end{center}
\caption{
Idea of the proof of Theorems~\ref{theo:external} and~\ref{theo:external_prime}.
The red interval is the face $G=[X_1,X_2]$, with $k=2$.  The vertical line passing through $X_1$ and $X_2$ is the affine subspace $A$. The grey horizontal plane is its orthogonal complement $A^\bot$. The points $Y_3,\ldots, Y_6$ are orthogonal projections of $X_3,\ldots,X_6$ on $A^\bot$.  The figure also shows the tangent cone (the solid angle bounded by the brown half-planes) and the normal cone (the blue two-dimensional angle in $A^\bot$). The reader should keep in mind that both angles should in fact be translated to $0$.
}
\label{fig:projections}
\end{figure}

\vspace*{2mm}
\noindent
\textit{Representation of cones and angles.}
Consider the affine subspace $A = \aff G = \aff (X_1,\ldots,X_k)$ and let $A^\bot$ be the orthogonal complement of $A$; see Figure~\ref{fig:projections}.  Note that $\dim A = k-1$ and $\dim A^\bot = d-k+1$ with probability $1$. Observe also that $A^\bot$ is by definition a linear subspace, whereas $A$ need not pass through the origin.  In the following, we shall identify $A^\bot$ with $\R^{d-k+1}$ by means of the isometry $I_{A^\bot} : A^{\bot} \to \R^{d-k+1}$, as explained in Section~\ref{subsec:identification}.   Let $\pi_{A^\bot}:\R^d \to A^\bot$ be the orthogonal projection onto $A^\bot$.
Consider the points
$$
Y_1:= \pi_{A^\bot}(X_{k+1}) \in A^\bot,
\;\; \ldots, \;\;
Y_{n-k} := \pi_{A^\bot}(X_n)\in A^\bot,
\;\;
Y := \pi_{A^\bot}(X_1)= \ldots = \pi_{A^\bot}(X_k)\in A^\bot.
$$
Let us assume that $G=[X_1,\ldots,X_k]$ is a face of $P$.  Then the tangent cone of  $P$ at $G$ is given by
\begin{align*}
T(G, P)
=
\pos\left(X_{1} - \bar X, \ldots, X_k - \bar X, X_{k+1}- \bar X, \ldots, X_{n} - \bar X\right),
\end{align*}
where the centre $\bar X = (X_1+\ldots+X_k)/k$ is almost surely contained in the relative interior of $G$.
Since the positive hull of $X_1- \bar X,\ldots, X_k - \bar X$ is $A-\bar X$, we arrive at
$$
T(G, P)
=
(A-\bar X) \oplus \pos(Y_1-Y,\ldots,Y_{n-k}-Y),
$$
where the direct sum $\oplus$ is orthogonal.
Since $A-\bar X$ is a linear space, it follows that the normal cone at $G$, defined as the polar of the tangent cone, is the polar cone of $\pos(Y_1-Y,\ldots,Y_{n-k}-Y)$ taken inside $A^\bot$ as the ambient space. Let us now map all our points to $\R^{d-k+1}$ by considering $Y_i':=I_{A^\bot}(Y_i)\in \R^{d-k+1}$ and $Y' := I_{A^\bot}(Y)\in \R^{d-k+1}$. From the isometry property of $I_{A^\bot}$ it follows that the internal and the external angles at $G$ are given by
\begin{align}
\beta(G, P)  &= \alpha (\pos(Y_1'-Y',\ldots,Y_{n-k}'-Y')), \label{eq:gamma_G_P_alpha_internal}\\
\gamma(G, P) &= \alpha (\pos^\circ (Y_1'-Y',\ldots,Y_{n-k}'-Y')). \label{eq:gamma_G_P_alpha}
\end{align}
The above holds if $G$ is a face of $P$. At this point let us observe that  $G$ is \textit{not} a face of $P$ if and only if $\pos(Y_1-Y,\ldots,Y_{n-k}-Y)= A^\bot$. This condition means that the angles on the right-hand sides of~\eqref{eq:gamma_G_P_alpha_internal} and~\eqref{eq:gamma_G_P_alpha} are equal to $1$ and $0$, respectively, which corresponds to our convention that $\beta(G,P)=1$ and $\gamma(G,P)=0$ if $G$ is not a face of $P$.

If $N$ denotes a vector with standard normal distribution on $\R^{d-k+1}$ that is independent of everything else, then the definitions of the solid angle and the polar cone imply that
\begin{equation}\label{eq:E_gamma_proof0}
\gamma(G, P)
=
\P[\langle Y_1'-Y', N\rangle \leq 0 ,\ldots, \langle Y_{n-k}'-Y', N\rangle \leq 0 \;|\; Y',Y_1',\ldots,Y_{n-k}'].
\end{equation}
Averaging over $X_1,\ldots,X_n$, we arrive at
\begin{equation}\label{eq:E_gamma_proof}
\E \gamma(G, P)
=
\P[\langle Y_1'-Y', N\rangle \leq 0 ,\ldots, \langle Y_{n-k}'-Y', N\rangle \leq 0].
\end{equation}
The above considerations are valid both for beta and beta' polytopes. In the following, we consider the beta case. Changes needed in the beta' case will be indicated at the end of the proof.

\vspace*{2mm}
\noindent
\textit{Proof of the independence.}
Observe that by~\eqref{eq:E_gamma_proof0}, the random variable $\gamma(G, P)$ is certain function of the random points $Y',Y_1',\ldots,Y_{n-k}'$. Let us argue that this collection is independent of $I_A(X_1)/\sqrt{1-h^2(A)},\ldots, I_A(X_k)/\sqrt{1-h^2(A)}$, where $h(A) = \|Y\|$ is the distance from the origin to $A$, and $I_A: A\to \R^{k-1}$ is an isometry satisfying $I_A(Y)=0$.   This would prove the independence statement of Theorem~\ref{theo:external}. Recall that $Y_i'=I_{A^\bot}(\pi_{A^\bot}(X_{k+i}))$, $1\leq i \leq n-k$, hence $Y_1',\ldots,Y_{n-k}'$ are functions  of $X_{k+1}, \ldots,X_n$ and $A^\bot$. Since $I_A(X_1)/\sqrt{1-h^2(A)},\ldots, I_A(X_k)/\sqrt{1-h^2(A)}$ are stochastically independent of $A$ by part (b) of Theorem~\ref{theo:ruben_miles}, these random points are independent of $Y_1',\ldots,Y_{n-k}'$.  They are also independent of $Y'=I_{A^\bot}(Y)$ because $\{Y\} = A\cap A^\bot$ is function of $A$ only.

\vspace*{2mm}
\noindent
\textit{Joint distribution of the projected points.}
Let us now describe the joint distribution of the points $Y',Y_1',\ldots,Y_{n-k}'$.
We claim that
\begin{itemize}
\item[(a)]  $Y', Y_1',\ldots,Y_{n-k}'$ are independent points in $\R^{d-k+1}$,
\item[(b)]  $Y_1',\ldots,Y_{n-k}'$ are i.i.d.\ with density $f_{d-k+1, \frac{2\beta+k-1} 2}$,
\item[(c)]  $Y'$ has density $f_{d-k+1,\gamma}$ with $\gamma = \frac {(2\beta+d)k+ k-d-1}{2}$.
\end{itemize}
To prove (a), observe that conditionally on $A^\bot$, the points $Y_i = I_{A^\bot} (\pi_{A^\bot} (X_{k+i}))$, $1\leq i\leq n-k$, form an i.i.d.\ sample with density $f_{d-k+1, \frac{2\beta+k-1} 2}$ by Lemma~\ref{lem:projection} (a).  Again conditionally on $A^{\bot}$, the point $Y' = I_{A^\bot}(p(A))$ (where $Y=p(A)$ is the projection of the origin onto $A$) has the density $f_{d-k+1, \gamma}$ by Theorem~\ref{theo:ruben_miles} (c) and (d). Still conditioning on $A^\bot$, we observe that  $Y'= I_{A^\bot}(\pi_{A^\bot}(X_1))$ is stochastically independent of the points $Y_i = I_{A^\bot} (\pi_{A^\bot} (X_{k+i}))$, $1\leq i\leq n-k$. Thus, properties  (a), (b), (c) hold conditionally on $A^\bot$. Since the joint conditional distribution of $Y',Y_1',\ldots,Y_{n-k}'$ does not depend on $A^\bot$, the statements hold in the unconditional sense, too.

\vspace*{2mm}
\noindent
\textit{Proof of the formula for the external angle.}
We are finally ready to compute the expected external angle. Since the joint distribution of $Y',Y_1',\ldots,Y_{n-k}'$ does not change under orthogonal transformations of $\R^{d-k+1}$, we may rewrite~\eqref{eq:E_gamma_proof} in the following form:
$$
\E \gamma(G, P)
=
\P[\langle Y_1'-Y', e\rangle \leq 0 ,\ldots, \langle Y_{n-k}'-Y', e\rangle \leq 0],
$$
where $e\in \R^{d-k+1}$ is any unit vector. Introducing the random variables $Z_i:= \langle Y_i',e\rangle$ and $Z := \langle Y,e\rangle$, we obtain
\begin{equation}\label{eq:gamma_Z}
\E \gamma(G, P)
=
\P[Z_1\leq Z ,\ldots, Z_{n-k}\leq Z].
\end{equation}
Projecting $Y'$ and $Y_i'$ to $Z'$ and $Z_i'$ reduces the dimension by $d-k$.  Now, by the above description of the joint law of $Y',Y_1',\ldots,Y_{n-k}'$ and by  Lemma~\ref{lem:projection} (a), we have that
\begin{itemize}
\item[(a)]  $Z, Z_1,\ldots,Z_{n-k}$ are independent random variables,
\item[(b)]  $Z_1,\ldots,Z_{n-k}$ are i.i.d.\ with density $f_{1, \frac{2\beta+d-1} 2}$,
\item[(c)]  $Z$ has density $f_{1,\frac {(2\beta+d) k -1}{2}}$.
\end{itemize}
Conditioning on the event that $Z=t$ in the right-hand side of~\eqref{eq:gamma_Z} and integrating, we obtain
\begin{multline*}
\E \gamma(G, P)
=
\P[Z_1\leq Z ,\ldots, Z_{n-k}\leq Z]\\
=
\int_{-1}^{+1} c_{1, \frac {(2\beta+d) k - 1}{2}}
(1-t^2)^{\frac {(2\beta+d) k - 1}{2}}
\left(\int_{-1}^t c_{1, \frac{2\beta+d - 1}{2}} (1-s^2)^{\frac{2\beta + d  - 1}{2}}\dd s\right)^{n-k} \dd t
=
I_{n,k}(2\beta+d),
\end{multline*}
where we used~\eqref{eq:I_definition} in the last equality.  This completes the proof of the formula for the expected external angle in the beta case.

\vspace*{2mm}
\noindent
\textit{The beta' case} is analogous to the beta case, but this time everything is based on Theorem~\ref{theo:ruben_miles_prime} and part (b) of Lemma~\ref{lem:projection}. The joint distribution of $Y',Y_1',\ldots,Y_{n-k}'$ is as follows:
\begin{itemize}
\item[(a)]  $Y', Y_1',\ldots,Y_{n-k}'$ are independent points in $\R^{d-k+1}$,
\item[(b)]  $Y_1',\ldots,Y_{n-k}'$ are i.i.d.\ with density $\tilde f_{d-k+1, \frac{2\beta-k+1} 2}$,
\item[(c)]  $Y'$ has density $\tilde f_{d-k+1,\gamma}$ with $\gamma = \frac {(2\beta-d) k + d-k+1}{2}$.
\end{itemize}
By Lemma~\ref{lem:projection} (b), the joint distribution of the one-dimensional projections $Z_i:= \langle Y_i',e\rangle$ and $Z := \langle Y',e\rangle$ is as follows:
\begin{itemize}
\item[(a)]  $Z, Z_1,\ldots,Z_{n-k}$ are independent random variables,
\item[(b)]  $Z_1,\ldots,Z_{n-k}$ are i.i.d.\ with density $\tilde f_{1, \frac{2\beta-d+1} 2}$,
\item[(c)]  $Z$ has density $\tilde f_{1,\frac {(2\beta-d) k +1}{2}}$.
\end{itemize}
Recalling~\eqref{eq:gamma_Z}, conditioning on the event that $Z=t$ and integrating, we obtain
\begin{multline*}
\E \gamma(G, P)
=
\P[Z_1\leq Z ,\ldots, Z_{n-k}\leq Z]\\
=
\int_{-\infty}^{+\infty} \tilde c_{1, \frac {(2\beta-d) k + 1}{2}}
(1+t^2)^{-\frac {(2\beta-d) k + 1}{2}}
\left(\int_{-\infty}^t \tilde c_{1, \frac{2\beta-d + 1}{2}} (1+s^2)^{-\frac{2\beta-d + 1}{2}}\dd s\right)^{n-k} \dd t
=
\tilde I_{n,k}(2\beta-d),
\end{multline*}
where we used~\eqref{eq:I_definition_prime} in the last equality. This completes the proof in the beta' case.
\end{proof}

\subsection{Internal angles under change of dimension}

To motivate the next theorem, consider a $d$-dimensional simplex $[Z_1,\ldots,Z_{d+1}]$ in a Euclidean space $\R^{d+\ell}$, where $\ell\in\N_0$. Let first $Z_1,\ldots,Z_{d+1}$ be i.i.d.\ with the beta distribution $f_{d+\ell,\beta}$.  Na\"ively, one might conjecture that the expected internal angle at a face of some fixed dimension $k$ does not depend on the choice of $\ell\in\N_0$. Indeed, this angle does not depend on whether we consider the simplex as embedded into $\R^{d+\ell}$ or into its own $d$-dimensional affine hull $A=\aff (Z_1,\ldots,Z_{d+1})$, and the beta density preserves its form when restricted to affine subspaces (up to scaling, which does not change the angle).  However, as we know from Theorem~\ref{theo:ruben_miles}, the joint distribution of $Z_1,\ldots,Z_{d+1}$ inside their own affine hull involves an additional ``Blaschke--Petkantschin term'' $\Delta^\ell(Z_1,\ldots,Z_{d+1})$, which is why the above argument breaks down.
In the next theorem we show that in order to make the expected internal angle independent of the dimension of the space the simplex is embedded in, we have to decrease the parameter of the beta distribution by $\frac 12$ each time we increase the dimension by $1$.

\begin{theorem}\label{theo:internal_angle_projection}
Let $X_1,\ldots,X_{d+1}$ be i.i.d.\ random points in $\R^{d+\ell}$  with the beta  distribution $f_{d+\ell, \beta - \frac \ell 2}$, where $\ell\in\N_0$ and $\beta -\frac \ell 2\geq -1$. Then, for all $k\in \{1,\ldots,d\}$, we have
\begin{equation*}
\E \beta([X_1,\ldots,X_{k}], [X_1,\ldots,X_{d+1}]) = J_{d+1,k}(\beta)
\end{equation*}
where $J_{d+1,k}(\beta)$ is given by~\eqref{eq:J_definition}.
That is, the expected internal angle does not depend on $\ell\in \N_0$ as long as $\beta-\frac \ell 2\geq -1$.
Similarly, if $X_1,\ldots,X_{d+1}$ are i.i.d.\ points in $\R^{d+\ell}$ with the beta'-type density $\tilde f_{d+\ell, \beta + \frac \ell 2}$, where $\ell\in\N_0$ and $\beta > \frac d2$,  then the above expected internal angle equals $\tilde J_{d+1,k}(\beta)$ defined in~\eqref{eq:J_definition_prime} and thus does not depend on the choice of $\ell\in\N_0$.
\end{theorem}
\begin{proof}
For concreteness, we consider the beta case.  The main tool in the proof is Lemma~\ref{lem:projection} that states that the projection of $[X_1,\ldots,X_{d+1}]$ to $\R^d$ is a full-dimensional simplex whose vertices are i.i.d.\ with distribution $f_{d,\beta}$. We have to relate the expected internal angles of $[X_1,\ldots,X_{d+1}]$ to those of its projection.

In Section~\ref{subsec:external_angles_proofs}, especially in Equation~\eqref{eq:gamma_G_P_alpha_internal}, we have shown (with a different notation) that
$$
\beta([X_1,\ldots,X_{k}], [X_1,\ldots,X_{d+1}]) = \alpha (\pos (V_1-V,\ldots,V_{d+1-k}-V)),
$$
where
\begin{itemize}
\item[(a)]  $V, V_1,\ldots,V_{d+1-k}$ are independent points in $\R^{d+\ell-k+1}$ such that
\item[(b)]  $V_1,\ldots,V_{d+1-k}$ are i.i.d.\ with distribution $f_{d+\ell-k+1, \frac{2\beta-\ell+k-1} 2}$ and
\item[(c)]  $V$ has distribution $f_{d+\ell-k+1,\gamma}$ with $\gamma = \frac {(2\beta- \ell+d +\ell)k+ k-d-\ell-1}2 = \frac {(2\beta+d)k+ k-d-1}{2} -\frac \ell 2$.
\end{itemize}
Note that an increase of the dimension by $\ell$ is always accompanied by a decrease of the beta-parameter by $\frac \ell2$. Let $\Pi: \R^{d+\ell-k+1}\to\R^{d-k+1}$ be the orthogonal projection defined by
\begin{equation}\label{eq:def_projection_Pi}
\Pi (x_0,x_1,\ldots,x_{d+\ell-k}) = (x_0,x_{\ell+1}\ldots,x_{d+\ell-k}),\qquad (x_0,\ldots,x_{d+\ell-k})\in\R^{d+\ell-k+1}.
\end{equation}
By Lemma~\ref{lem:projection} (a), the joint distribution of the points $W:=\Pi V$, $W_1:=\Pi V_1,\ldots,W_{d+1-k}:=\Pi V_{d+1-k}$ can be described as follows:
\begin{itemize}
\item[(a)]  $W, W_1,\ldots,W_{d+1-k}$ are independent points in $\R^{d-k+1}$,
\item[(b)]  $W_1,\ldots,W_{d+1-k}$ are i.i.d.\ with distribution $f_{d-k+1, \frac{2\beta+k-1} 2}$,
\item[(c)]  $W$ has distribution $f_{d-k+1,\gamma'}$ with $\gamma' =  \frac {(2\beta+d)k+ k-d-1}{2}$.
\end{itemize}
In particular, their distribution does not depend on $\ell$. To prove the theorem, it suffices to show that
\begin{equation}\label{eq:E_alpha_V_eq_E_alpha_W}
\E \alpha (\pos (V_1-V,\ldots,V_{d+1-k}-V)) = \E \alpha (\pos (W_1-W,\ldots,W_{d+1-k}-W)).
\end{equation}
Since this identity becomes trivial for $\ell=0$, we shall henceforth assume that $\ell\in\N$. Using the definition of the solid angle, we have
\begin{align*}
2\,\E \alpha (\pos (W_1-W,\ldots,W_{d+1-k}-W))
=
\P[\pos (W_1-W,\ldots,W_{d+1-k}-W) \cap L_{1} \neq \{0\}],
\end{align*}
where $L_1 \in G(d-k+1, 1)$ is a uniformly distributed random line passing through the origin which is independent of everything else. Since the probability law of the cone $\pos (W_1-W,\ldots,W_{d+1-k}-W)$ is invariant under orthogonal transformations, we can replace $L_1$ by any fixed line, which leads to
$$
2\,\E \alpha (\pos (W_1-W,\ldots,W_{d+1-k}-W))
=
\P[\pos (W_1-W,\ldots,W_{d+1-k}-W) \cap \lin (e_0) \neq \{0\}],
$$
where $e_0,e_1,\ldots, e_{d-k}$ is the standard orthonormal basis of $\R^{d-k+1}$.
On the other hand, using the properties of conic intrinsic volumes and the conic Crofton formula, see, in particular, \eqref{eq:h_def} and~\eqref{eq:ConicalIntVolGrassmannAngle}, we can write
\begin{align*}
2\,\E \alpha (\pos (V_1-V,\ldots,V_{d+1-k}-V))
&=
2\,\E \upsilon_{d-k+1} (\pos (V_1-V,\ldots,V_{d+1-k}-V)) \\
&=
2\,\E h_{d-k+1} (\pos (V_1-V,\ldots,V_{d+1-k}-V))\\
&=
\P[\pos (V_1-V,\ldots,V_{d+1-k}-V) \cap L_{\ell + 1}' \neq \{0\}],
\end{align*}
where $L_{\ell + 1}'\in G(d+\ell-k+1,\ell+1)$ is a random, uniformly distributed $(\ell+1)$-dimensional linear subspace of $\R^{d+\ell-k+1}$ that is independent of everything else. Once again by rotational invariance of the probability law of the random cone $\pos (V_1-V,\ldots,V_{d+1-k}-V)$, we can replace $L_{\ell+1}'$ by an arbitrary deterministic $(\ell+1)$-dimensional linear subspace of our choice, which leads to
\begin{equation}\label{eq:alpha_pos_as_intersect}
\begin{split}
&2\,\E \alpha (\pos (V_1-V,\ldots,V_{d+1-k}-V))\\
&\qquad=
\P[\pos (V_1-V,\ldots,V_{d+1-k}-V) \cap \lin(e_0,e_1,\ldots,e_\ell) \neq \{0\}],
\end{split}
\end{equation}
where $e_0,e_1,\ldots, e_{d+\ell-k}$ is the standard orthonormal basis of $\R^{d+\ell-k+1}$. Recalling that $W=\Pi V$ and $W_i=\Pi V_i$ for $i\in\{1,\ldots,d-k+1\}$, and using the definition of the orthogonal projection $\Pi$ given in~\eqref{eq:def_projection_Pi}, we arrive at
\begin{align*}
&2\,\E \alpha (\pos (W_1-W,\ldots,W_{d+1-k}-W))\\
&\qquad=
\P[\pos (W_1-W,\ldots,W_{d+1-k}-W) \cap \lin (e_0) \neq \{0\}]\\
&\qquad=
\P[\Pi \pos (V_1-V,\ldots,V_{d+1-k}-V) \cap (\lin (e_0)\backslash \{0\}) \neq \varnothing]\\
&\qquad=
\P[\pos (V_1-V,\ldots,V_{d+1-k}-V) \cap \Pi^{-1}(\lin (e_0)\backslash \{0\}) \neq \varnothing].
\end{align*}
Now, $ \Pi^{-1}(\lin (e_0)\backslash \{0\}) = \lin(e_0,e_1,\ldots,e_\ell) \backslash \lin (e_1,\ldots,e_\ell)$. Thus, we can write
\begin{multline*}
2\,\E \alpha (\pos (W_1-W,\ldots,W_{d+1-k}-W))
\\=
\P[\exists v \in (\lin(e_0,e_1,\ldots,e_\ell)\backslash\{0\}) \backslash (\lin (e_1,\ldots,e_\ell)\backslash\{0\})\colon v\in \pos (V_1-V,\ldots,V_{d+1-k}-V)].
\end{multline*}
However, by rotational invariance of the involved distributions, the $(d+1-k)$-dimensional linear space $\lin (V_1-V,\ldots,V_{d+1-k}-V)$ has the uniform distribution $\nu_{d+1-k}$ on the Grassmannian $G(d+\ell-k+1,d+1-k)$. So, \cite[Lemma 13.2.1]{SW08} implies that the intersection of $\lin (V_1-V,\ldots,V_{d+1-k}-V)$ with the $\ell$-dimensional linear space $E:=\lin (e_1,\ldots,e_\ell)$ in $\R^{d+\ell-k+1}$ is $\{0\}$ with probability $1$. Indeed,
\begin{align*}
&\P[\lin (V_1-V,\ldots,V_{d+1-k}-V)\cap E\neq\{0\}] \\
&\qquad= \int_{{\rm SO}(d+\ell-k+1)}\ind_{\{\dim(\vartheta\lin(e_{\ell+1},\ldots,e_{d+\ell-k+1})\cap E)>0\}}\,\nu(\dint \vartheta) = 0,
\end{align*}
where ${\rm SO}(d+\ell-k+1)$ is the special orthogonal group in $\R^{d+\ell-k+1}$ with its unique invariant Haar probability measure $\nu$.
It follows that
\begin{align*}
&2\,\E \alpha (\pos (W_1-W,\ldots,W_{d+1-k}-W))\\
&\qquad=
\P[\pos (V_1-V,\ldots,V_{d+1-k}-V) \cap \lin(e_0,e_1,\ldots,e_\ell) \neq \{0\}]\\
&\qquad=
2\,\E \alpha (\pos (V_1-V,\ldots,V_{d+1-k}-V)),
\end{align*}
where we used~\eqref{eq:alpha_pos_as_intersect} in the last step. This proves~\eqref{eq:E_alpha_V_eq_E_alpha_W} and completes the proof of Theorem~\ref{theo:internal_angle_projection} in the beta case.

The proof in the beta' case is similar, but this time an increase of the dimension by $\ell$ is always accompanied by an \textit{increase} of the beta'-parameter by $\frac \ell2$, see Lemma~\ref{lem:projection} (b).
\end{proof}

\begin{remark}
Returning to the discussion at the beginning of this section, we can equivalently restate Theorem~\ref{theo:internal_angle_projection} as follows. Let $Y_1,\ldots,Y_{d+1}$ be  (in general, stochastically dependent) random points in $\R^d$ whose joint density is proportional to
$$
\Delta^{\ell} (y_1,\ldots,y_{d+1})\,\prod_{i=1}^{d+1} f_{d,\beta - \frac{\ell}{2} }(y_i).
$$
Then, the expected internal angle $\E \beta([Y_1,\ldots,Y_k], [Y_1,\ldots,Y_{d+1}])$ does not depend on the choice of $\ell\in\N_0$, as long as $\beta-\frac \ell 2 \geq -1$.  Indeed, by Theorem~\ref{theo:ruben_miles}, the joint distribution of $X_1,\ldots,X_{d+1}$ inside their own affine hull $\aff(X_1,\ldots,X_{d+1})$ is the same as the joint distribution of $Y_1,\ldots,Y_{d+1}$ up to rescaling, which does not change internal angles. A similar statement also holds in the beta' case.
\end{remark}

\subsection{Analytic continuation}
One of the main ideas used in our proofs is to raise the dimension. More precisely, we shall view the beta polytope $P_{n,d}^\beta\subset \R^d$ as a projection of $P_{n,d+1}^{\beta-1/2} \subset \R^{d+1}$; see Lemma~\ref{lem:projection} (a).  Since raising the dimension must be accompanied by lowering the parameter $\beta$, such a representation is possible for $\beta \geq -\frac 12$ only. For example the uniform distribution on $\bS^{d-1}$ (corresponding to $\beta=-1$) cannot be represented as a projection of a higher-dimensional beta distribution. It is for this reason that our proofs work for $\beta > -\frac 12$ only. In order to extend the results to the full range $\beta\geq -1$, we shall use analytic continuation. To this end, we need to show that the functionals  we are interested in, such as the expected internal angles of beta simplices, can be viewed as analytic functions of $\beta$.
The following lemma makes this precise and will be applied several times below.
Observe that for any fixed $x\in\BB^d$, we can consider
$$
f_{d,z}(x)= \frac{ \Gamma\left( \frac{d}{2} + z + 1 \right) }{ \pi^{ \frac{d}{2} } \Gamma\left( z+1 \right) }(1-\|x\|^2)^z
$$
as an analytic function of the complex variable $z$ on the half-plane $H_{-1}:=\{z\in\CC: \Re z>-1\}$.

\begin{lemma}\label{lem:Analytic}
Fix $d\in\N$, $n\in\N$, and let $\varphi:(\BB^d)^{n}\to \R$ be a bounded measurable function.   Then the function
$$
\cI(z):=\int_{(\BB^d)^{n}}\varphi(x_1,\ldots,x_{n})\left(\prod_{i=1}^{n} f_{d,z}(x_i)\right)\,
\lambda_d(\dint x_1) \ldots \lambda_d(\dint x_{n})
$$
is analytic on  the half-plane $H_{-1}$.
\end{lemma}
\begin{proof}
If $K\subset H_{-1}$ is a compact set, then there is a constant $C(K)$ depending only on $K$ such that
\begin{equation}\label{eq:est222}
|f_{d,z}(x)| = \left|\frac{ \Gamma\left( \frac{d}{2} + z + 1 \right) }{ \pi^{ \frac{d}{2} } \Gamma\left( z+1 \right) }(1-\|x\|^2)^z\right| \leq C(K) (1-\|x\|^2)^{\Re z}
\end{equation}
for all $x\in \BB^d$ and $z\in K$.
Since $\varphi$ is bounded and the function $(1-\|x\|^2)^{\Re z}$ is integrable over $\BB^d$ for $\Re z >-1$, the function $\cI(z)$ is well-defined.

\vspace*{2mm}
\noindent
\textit{Continuity.}
In a next step,  we claim that $\cI(z)$ is continuous on $H_{-1}$. To prove this, take a sequence $(z_k)_{k\in\N}\subset H_1$ with $z_k\to z\in H_1$, as $k\to\infty$. Then,
$$
|\cI(z) - \cI(z_k)| \leq \int_{(\BB^d)^{n}} |\varphi(x_1,\ldots,x_{n})| \left|\prod_{i=1}^{n}f_{d,z}(x_i)-\prod_{i=1}^{n}f_{d,z_k}(x_i)\right|\,
\lambda_d(\dint x_1) \ldots \lambda_d(\dint x_{n}).
$$
For every fixed $x_1,\ldots,x_{n}$ and as $k\to\infty$, the integrand converges to $0$,  because $\lim_{k\to\infty} f_{d,z_k}(x_i) = f_{d,z}(x_i)$ for all $x_1,\ldots,x_n\in\BB^d$. Moreover, recall that $\varphi$ is bounded and observe that by the triangle inequality and~\eqref{eq:est222},
\begin{equation*}
\left|\prod_{i=1}^{n}f_{d,z}(x_i)-\prod_{i=1}^{n}f_{d,z_k}(x_i)\right|
\leq
C(K)^{n} \prod_{i=1}^{n} (1-\|x_i\|^2)^{\Re z} + C(K)^{n} \prod_{i=1}^{n} (1-\|x_i\|^2)^{a}
\end{equation*}
with $K = \{z, z_1,\ldots\}$ being compact and $a := \inf_{k\in\N} \Re z_k > -1$. Since the function $(1-\|x_i\|^2)^a$ is integrable over $\BB^{d}$ for $a>-1$, the dominated convergence theorem applies, thus proving that $\cI(z_k)\to \cI(z)$, as $k\to\infty$. Hence, $\cI(z)$ is continuous.

\vspace*{2mm}
\noindent
\textit{Analyticity.}
To prove that $\cI(z)$ is analytic, let $\gamma\subset H_{-1}$ be any triangular contour. By Morera's theorem \cite[Theorem 10.17]{RudinRealComplexAnalysis} it suffices to show that
$$
\oint_\gamma\cI(z)\,\dint z = 0.
$$
But since the function $z\mapsto \prod_{i=1}^{n} f_{d,z}(x_i)$ is analytic on $H_{-1}$ for all $x_1,\ldots,x_n\in\BB^d$, Cauchy's integral theorem \cite[Theorem 10.14]{RudinRealComplexAnalysis} implies that, for all $x_1,\ldots,x_{n}\in\BB^d$,
$$
\oint_\gamma \prod_{i=1}^{n}f_{d,z}(x_i)\,\dint z = 0.
$$
Since $\varphi$ is bounded, for every $z\in\gamma$ and $x_1,\ldots,x_{n} \in\BB^d$ we have
$$
\left| \varphi(x_1,\ldots,x_{n})\left(\prod_{i=1}^{n}f_{d,z}(x_i)\right)\right|
\leq
\sup_{x_1,\ldots,x_n\in\BB^d} |\varphi(x_1,\ldots,x_n)| \cdot C(\gamma)^{n} \prod_{i=1}^{n}(1-\|x\|_i^2)^{b},
$$
where $b:=\inf_{z\in\gamma} \Re z>-1$.
Since the function $(1-\|x_i\|^2)^b$ is integrable over $\BB^{d}$ for $b>-1$, we may interchange the order of integration by Fubini's theorem, which  yields
\begin{align*}
\oint_\gamma\cI(z)\,\dint z = \int_{(\BB^d)^{n}}\varphi(x_1,\ldots,x_{n})
\left(\oint_\gamma \prod_{i=1}^{n}f_{d,z}(x_i)\,\dint z \right) \lambda_d (\dint x_1) \ldots \lambda_{d}(\dint x_{n})
= 0.
\end{align*}
Note that since $\gamma$ is a triangle, the contour integral can be reduced to usual Lebesgue integrals, which justifies the above use of Fubini's theorem.  The argument is complete.
\end{proof}

\begin{corollary}\label{cor:analytic_J}
The function $J_{m,\ell}(\alpha)$,  originally defined in Theorem~\ref{theo:f_vect} for real $\alpha>-1$, admits an extension to an  analytic function on the half-plane $H_{-1}=\{z\in\CC: \Re z>-1\}$.
\end{corollary}
\begin{proof}
Apply Lemma~\ref{lem:Analytic} with $d=m-1$, $n=m$ and $\varphi(x_1,\ldots,x_m) = \beta ([x_1,\ldots,x_{\ell}], [x_1,\ldots,x_m])$, which is bounded by $1$ and measurable.
\end{proof}

\begin{corollary}\label{cor:analytic_f_k}
The function $\beta\mapsto \E f_k(P_{n,d}^\beta)$,  originally defined for real $\beta>-1$, admits an extension to an  analytic function on the half-plane $H_{-1}=\{z\in\CC: \Re z>-1\}$.
\end{corollary}
\begin{proof}
Apply Lemma~\ref{lem:Analytic} with  $\varphi(x_1,\ldots,x_{n}) = f_k([x_1,\ldots,x_n])$, which is bounded by $\binom {n}{k+1}$.
\end{proof}

\begin{remark}
Observe that the problem mentioned at the beginning of the section does not arise in the beta' case since by Lemma~\ref{lem:projection} (b) we can represent $\tilde P_{n,d}^{\beta}$ as a projection of $\tilde P_{n,d+1}^{\beta+1/2}$ for any $\beta>\frac d2$ and the new parameters also satisfy $\beta+ \frac 12 >\frac{d+1}{2}$. This is why we only treated the beta case here.
\end{remark}

\subsection{Expected \texorpdfstring{$f$}{f}-vector}
In this section we prove Theorems \ref{theo:f_vect} and \ref{theo:f_vect_prime}.

\begin{proof}[Proof of Theorem~\ref{theo:f_vect}.]
We are going to compute the expected $f$-vector of $P_{n,d}^\beta$. To this end, we shall represent this polytope as a random projection of a higher-dimensional polytope and then use the formula from Proposition \ref{theo:affentranger_schneider}.

\vspace*{2mm}
\noindent
\textit{Geometric argument.}
We take some $\ell\in\N$, assume that $\beta-\frac \ell2 >-1$ and consider the random polytope $P_{n,d+\ell}^{\beta-\frac \ell 2}$ in $\R^{d+\ell}$. Independently, let $L_d$ be a random, uniformly distributed $d$-dimensional linear subspace in $\R^{d+\ell}$ and denote by $\Pi_d$ the orthogonal projection on $L_d$. By Lemma~\ref{lem:projection} (a) we have that
\begin{equation}\label{eq:P_is_projection}
f_k(\Pi_d P_{n,d+\ell}^{\beta - \frac \ell 2}) \eqdistr f_k(P_{n,d}^{\beta}).
\end{equation}
In particular, the expectations of these quantities are equal. On the other hand, by Proposition~\ref{theo:affentranger_schneider}, we have that
$$
\E \left[ f_k(\Pi_d P_{n,d+\ell}^{\beta-\frac \ell  2}) \Big| P_{n,d+\ell}^{\beta-\frac \ell2} \right]
=
2 \sum_{s=0}^\infty \sum_{G\in \cF_{d-2s-1}(P_{n,d+\ell}^{\beta-\frac \ell2})} \gamma(G, P_{n,d+\ell}^{\beta-\frac \ell2}) \sum_{F\in\cF_{k}(G)} \beta(F,G).
$$
In the following we consider only terms with $d-2s\geq 1$ because all remaining terms are equal to $0$.
All $(d-2s-1)$-dimensional faces of $P_{n,d+\ell}^{\beta  - \frac \ell 2}$ have the form $G=[X_{i_1},\ldots, X_{i_{d-2s}}]$ for some indices $1\leq i_1<\ldots < i_{d-2s} \leq n$. By symmetry, the contributions of all these faces are equal, so we may just take $G = [X_1,\ldots,X_{d-2s}]$ (on the event that this is indeed a face) and write
\begin{multline*}
\E \left[f_k(\Pi_d  P_{n,d+\ell}^{\beta - \frac \ell 2})\right]
=
\E \left[\E \left[ f_k(\Pi_d  P_{n,d+\ell}^{\beta - \frac \ell 2}) \Big| P_{n,d+\ell}^{\beta - \frac \ell 2} \right]\right]\\
=
2 \sum_{s=0}^\infty \binom n {d-2s} \E \left[\gamma(G, P_{n,d+\ell}^{\beta - \frac \ell 2})\ind_{\left\{G \in \cF_{d-2s-1}( P_{n,d+\ell}^{\beta - \frac \ell 2})\right\}} \sum_{F\in \cF_k(G)} \beta(F,G) \right].
\end{multline*}
By~\eqref{eq:P_is_projection} and the independence part of Theorem~\ref{theo:external} (which is a crucial step in this proof allowing us to treat external and internal angles separately), we have
\begin{align*}
\E f_k(P_{n,d}^{\beta})
&=
\E \left[f_k(\Pi_d  P_{n,d+\ell}^{\beta - \frac \ell2})\right]\\
&=
2 \sum_{s=0}^\infty \binom n {d-2s} \E \left[\gamma(G, P_{n,d+\ell}^{\beta - \frac \ell 2})
\ind_{\left\{G \in \cF_{d-2s-1}( P_{n,d+\ell}^{\beta - \frac \ell 2})\right\}}\right] \E\left[\sum_{F\in \cF_k(G)} \beta(F,G) \right].
\end{align*}
By Theorem~\ref{theo:external} and recalling the convention that the external angle is $0$ if $G$ is not a face, we obtain
\begin{align*}
\E \left[\gamma(G, P_{n,d+\ell}^{\beta - \frac \ell 2})\ind_{\left\{G \in \cF_{d-2s-1}
(P_{n,d+\ell}^{\beta - \frac \ell 2})\right\}}\right]
&=
I_{n,d-2s}\left(2\left(\beta - \frac \ell 2\right)  + d +  \ell\right)\\
&=
I_{n,d-2s}(2\beta+d).
\end{align*}
Also, recalling that $G$ is the convex hull of i.i.d.\ random points $X_1,\ldots,X_{d-2s}$ in $\R^{d+\ell}$ with density $f_{d+\ell,\beta -\frac \ell 2}$, we apply Theorem~\ref{theo:internal_angle_projection} to deduce that
\begin{align*}
\E\left[\sum_{F\in \cF_k(G)} \beta(F,G) \right]
&=
\binom {d-2s}{k+1}
\E\beta([X_1,\ldots,X_{k+1}], [X_1,\ldots,X_{d-2s}])\\
&=
\binom {d-2s}{k+1}
\E\beta([X_1',\ldots,X_{k+1}'], [X_1',\ldots,X_{d-2s}'])\\
&=
\binom {d-2s}{k+1} J_{d-2s,k+1}\left(\beta + s + \frac 12\right),
\end{align*}
where $X_1',\ldots,X_{d-2s}'$ are i.i.d.\ random points in $\R^{d-2s-1}$ with density $f_{d-2s-1,\beta + s + \frac 12}$.
Taking everything together, we arrive at the final formula
\begin{equation}\label{eq:ProofXXX}
\E f_k(P_{n,d}^{\beta})
=
2 \sum_{s=0}^\infty \binom n {d-2s} \binom {d-2s}{k+1} I_{n,d-2s}(2\beta + d) J_{d-2s,k+1}\left(\beta + s + \frac 12\right).
\end{equation}
For the above argument, the value of $\ell\in\N$ was irrelevant, so that we can take $\ell=1$. Because of the restriction on $\beta$ at the very beginning of the argument, the proof so far only covers the case where $\beta > - \frac 12$.

\vspace*{2mm}
\noindent
\textit{Analytic continuation: Proof for $\beta>-1$.}
To extend the result to all $\beta>-1$ we argue by analytic continuation. For that purpose we first recall that by Corollary~\ref{cor:analytic_f_k}, the function $\beta\mapsto \E f_k(P_{n,d}^{\beta})$ admits an analytic continuation to  $\beta\in H_{-1} = \{z\in\CC: \Re z>-1\}$. On the other hand, also the right-hand side in~\eqref{eq:ProofXXX} admits an analytic extension to $\beta\in H_{-1}$. Indeed, for $J_{d-2s,k+1}(\beta+s+{1\over 2})$ this was observed in Corollary~\ref{cor:analytic_J}. For $I_{n,d-2s}(2\beta+d)$ this follows from the identity
$$
I_{n,d-2s}(2\beta+d) = \E \gamma([X_1,\ldots,X_{d-2s}], [X_1,\ldots,X_n]), \quad \beta>-1,
$$
where $X_1,\ldots,X_n$ are i.i.d.\ with density $f_{d,\beta}$
(see Theorem~\ref{theo:external}) and Lemma~\ref{lem:Analytic} with $\varphi(x_1,\ldots, x_n) = \gamma ([x_1,\ldots,x_{d-2s}], [x_1,\ldots,x_n])$.
Hence, by the uniqueness of analytic continuation (see~\cite[Corollary to Theorem 10.18]{RudinRealComplexAnalysis}), these two expressions must coincide for all $\beta\in(-1,\infty)$, since they already coincide for all $\beta\in(-{\frac 1 2},\infty)$.

\vspace*{2mm}
\noindent
\textit{Continuity: Proof for $\beta=-1$.}
To prove that~\eqref{eq:ProofXXX} also holds in the limiting case $\beta=-1$ corresponding to the uniform distribution on $\SS^{d-1}$, we shall argue that both sides of~\eqref{eq:ProofXXX} are continuous at $\beta=-1$. Regarding the left-hand side, we claim that
\begin{equation}\label{eq:lim_E_f_k_beta_minus_1}
\lim_{\beta \downarrow -1} \E f_k(P_{n,d}^\beta) = \E f_k(P_{n,d}^{-1})
\end{equation}
for all $d\geq 2$, $n\geq d+1$ and $k \in \{0,1,\ldots, d-1\}$.
To prove this, we observe that the mapping $(x_1,\ldots,x_n) \mapsto f_k([x_1,\ldots,x_n])$ from $(\BB^d)^n$ to $\{0,1,2,\ldots\}$ is continuous on the set $\text{GP}_{n,d}$ of all tuples $(x_1,\ldots,x_n)$ that are in general position (meaning that no $d+1$ points are located on a common affine hyperplane); see also~\cite[Lemma 4.1]{beta_polytopes}. Let $X_1^{(\beta)},\ldots,X_n^{(\beta)}\in\BB^d$ be i.i.d.\ random points with the beta density $f_{d,\beta}$ (if $\beta>-1$) or with the uniform distribution on $\bS^{d-1}$ (if $\beta=-1$).
From the proof of \cite[Proposition 3.9]{beta_polytopes} we conclude that we have the convergence in distribution
$$
(X_1^{(\beta)},\ldots,X_n^{(\beta)})\overset{d}{\longrightarrow} (X_1^{(-1)},\ldots,X_n^{(-1)}),
$$
as $\beta\downarrow -1$. Also, almost surely, $(X_1^{(-1)},\ldots,X_{n}^{(-1)}) \in \text{GP}_{n,d}$. The continuous mapping theorem then yields that
$$
f_k(P_{n,d}^\beta) \overset{d}{\longrightarrow} f_k(P_{n,d}^{-1}),
$$
as $\beta\downarrow -1$. Moreover, since almost surely $f_k(P_{n,d}^\beta) \leq \binom n{k+1}$ for all $\beta\geq -1$, we conclude from this  that~\eqref{eq:lim_E_f_k_beta_minus_1} holds.

It remains to prove that the right-hand side of~\eqref{eq:ProofXXX} is also continuous at $\beta = -1$. Indeed, for $J_{d-2s,k+1}\left(\beta + s + \frac 12\right)$ we even proved analyticity since $\beta + s+\frac 12 \geq -\frac 12$, while for $I_{n,d-2s}(2\beta + d)$ the continuity follows from the defining integral representation~\eqref{eq:I_definition} since $2\beta + d\geq d-2\geq 0$ for $d\geq 2$. Having proved that both sides are continuous at $\beta=-1$, we conclude that~\eqref{eq:ProofXXX} indeed holds for $\beta=-1$.
\end{proof}

\begin{proof}[Proof of Theorem~\ref{theo:f_vect_prime}.]
The proof for the beta' case is line by line the same as the one for Theorem~\ref{theo:f_vect} given before.   In addition to the distributional equality
$$
f_k(\Pi_d\tilde{P}_{n,d+\ell}^{\beta+{\ell\over 2}}) \overset{d}{=} f_k(\tilde{P}_{n,d}^\beta)
$$
that follows from Lemma \ref{lem:projection} (b), one now uses Theorem \ref{theo:external_prime} instead of Theorem \ref{theo:external} and the beta' case of Theorem \ref{theo:internal_angle_projection}. No analytic continuation and continuity arguments are needed in this case.
\end{proof}

\subsection{Proof of the monotonicity}
In this section we prove Theorems~\ref{theo:monoton} and~\ref{theo:monoton_prime}.
Fix $d\geq 2$, $n\geq d+1$ and $k\in\{0,1,\ldots,d-1\}$. Our aim is to prove that $\E f_k(P_{n,d}^{\beta}) < \E f_k(P_{n+1,d}^{\beta})$ and $\E f_k(\tilde P_{n,d}^{\beta}) < \E f_k(\tilde P_{n+1,d}^{\beta})$. In view of the formulae
\begin{align*}
\E f_k(P_{n,d}^{\beta})
&=
2 \sum_{s=0}^\infty \binom n {d-2s} \binom {d-2s}{k+1} I_{n,d-2s}(2\beta+d) J_{d-2s,k+1}\left(\beta + s + \frac 12\right),\\
\E f_k(\tilde P_{n,d}^{\beta})
&=
2 \sum_{s=0}^\infty \binom n {d-2s} \binom {d-2s}{k+1} \tilde I_{n,d-2s}(2\beta-d) \tilde J_{d-2s,k+1}\left(\beta - s - \frac 12\right),
\end{align*}
that follow from Theorem \ref{theo:f_vect} and Theorem \ref{theo:f_vect_prime}, respectively, it suffices to show that
\begin{equation}\label{eq:MonoBetaToShow}
\binom n {m} I_{n,m}(\alpha) \leq \binom {n+1} {m} I_{n+1,m}(\alpha)
\end{equation}
and
\begin{equation}\label{eq:MonoBetaPrimeToShow}
\binom n {m} \tilde I_{n,m}(\alpha) \leq \binom {n+1} {m} \tilde I_{n+1,m}(\alpha)
\end{equation}
for all   $\alpha\geq 0$ and $m \in \{1,\ldots,n-1\}$ with strict inequality holding if $m\neq 1$. Note that we do not need to consider the case $m=0$, since the term with $d-2s=0$ vanishes because then $\binom {d-2s}{k+1}=0$.
Recall from Theorems~\ref{theo:f_vect} and~\ref{theo:f_vect_prime} that
\begin{align}
\binom n {m} I_{n,m}(\alpha)
&=
\binom n {m}  \int_{-1}^{+1}  c_{1, \frac {\alpha m -1}{2}}
(1-t^2)^{\frac {\alpha m - 1}{2}}
\left(\int_{-1}^t c_{1, \frac{\alpha - 1}{2}} (1-s^2)^{\frac{\alpha - 1}{2}}\dd s\right)^{n-m} \dd t,
\label{eq:I_n_m_proof_monoton1}\\
\binom n {m} \tilde I_{n,m}(\alpha)
&=
\binom n {m} \int_{-\infty}^{+\infty} \tilde c_{1, \frac {\alpha m + 1}{2}}
(1+t^2)^{-\frac {\alpha m + 1}{2}}
\left(\int_{-\infty}^t \tilde c_{1, \frac{\alpha + 1}{2}} (1+s^2)^{-\frac{\alpha + 1}{2}}\dd s\right)^{n-m} \dd t.
\label{eq:I_n_m_proof_monoton2}
\end{align}
Note that the factors  $c_{1, \frac {\alpha m -1}{2}}$ and $\tilde c_{1, \frac {\alpha m + 1}{2}}$ appearing in the above formulae are strictly positive and do not depend on $n$. For this reason, these factors do not play any essential role in the sequel. To simplify the notation, we introduce the distribution function $F(t) = \int_{-\infty}^t f(s)\, \dd s$, where $f$ is the probability density on $\R$ given by
$$
f(s)
=
\begin{cases}
f_{1,\frac{\alpha-1}{2}}(s) = c_{1, \frac{\alpha - 1}{2}} (1-s^2)^{\frac{\alpha - 1}{2}} \ind_{\{|s|<1\}},
&\text{ in the beta case,}\\
\tilde f_{1,\frac{\alpha+1}2}(s) = \tilde c_{1, \frac{\alpha + 1}{2}} (1+s^2)^{-\frac{\alpha + 1}{2}},
&\text{ in the beta' case}.
\end{cases}
$$

Let first $m=1$. For concreteness, we consider the beta case. From~\eqref{eq:I_n_m_proof_monoton1} we have
$$
\binom n 1 I_{n,1} (\alpha) = n \int_{-1}^{+1}  f(t) F^{n-1}  (t)\dd t = 1.
$$
So, for $m=1$, both sides of~\eqref{eq:MonoBetaToShow} are equal to $1$. The beta' case is similar.

Let in the following  $m\in \{2,\ldots,n-1\}$.
From~\eqref{eq:I_n_m_proof_monoton1} and~\eqref{eq:I_n_m_proof_monoton2} we  see that it is necessary to study monotonicity in $n$ for expressions of the form
$$
g_{n,m} := \binom nm \int_{-\infty}^{+\infty} f^{(m-1) \gamma  + 1}(t) F^{n-m}(t) \dd t,
$$
where $\gamma=\frac{\alpha}{\alpha-1}$ with $\alpha=2\beta+d$ in the beta case and $\gamma=\frac{\alpha}{\alpha+1}$ with $\alpha=2\beta-d$ in the beta' case. Note that $\alpha\geq 0$ in both cases. Below, we shall consider the beta case with $\alpha=1$ separately, so let us assume that $\gamma$ is well-defined. 

\begin{lemma}\label{lem:monotone_g}
Assume that $f$ is a probability density on $\R$ that is strictly positive and continuously differentiable on some non-empty open interval $I\subseteq\R$ (which is allowed to coincide with the whole real line $\R$) and zero on $\R\setminus I$. If $\gamma\in\R$ and the function $\gamma f^{\gamma-2}(t) f'(t)$ is strictly decreasing on $I$,
then
$$
g_{n+1,m} > g_{n,m}
$$
for all $m\in\{2,3,\ldots\}$ and $n \in \{m,m+1,\ldots\}$.
\end{lemma}

For the proof we need the following slightly corrected version of Lemma~5 from~\cite{bonnet_etal}.
\begin{lemma}\label{lem:FromBGTTTW}
Let $h, g,L:(0,1)\to\R$ be three functions such that
\begin{enumerate}
\item $h$ is non-negative, measurable, and   $0 < \int_0^1 h(s)\dint s<\infty$;
\item $g$ is linear, with negative slope and a root at $s^*\in (0,1)$,
\item $L$ is non-negative and strictly concave on $(0,1)$.  
\end{enumerate}
Then, for all $m > 1$,
$$
\int_0^1 h(s)g(s)L^{m-1}(s)\,\dint s > \int_0^1 h(s)g(s)\left(\frac {L(s^*)}{s^*}s\right)^{m-1}\,\dint s.
$$
\end{lemma}

\begin{proof}[Proof of Lemma~\ref{lem:monotone_g}]
Observe that under the assumptions of the lemma the distribution function $F$ is  strictly increasing and continuously differentiable on $I$. The tail function $\bar F(t)= 1-F(t)$ has thus a well-defined inverse $\bar F^{-1}$.  Using the definition of $g_{n,m}$ and then the substitution $\bar F(t) = s$, we arrive at
\begin{align*}
g_{n+1,m}- g_{n,m}
&= \int_{I} f^{(m-1) \gamma+1}(t) \bigg[{n+1\choose m}F(t)-{n\choose m}\bigg]F(t)^{n-m}\,\dd t\\
&= \int_0^1 f^{(m-1) \gamma}(\bar F^{-1}(s)) \bigg[{n+1\choose m}(1-s)-{n\choose m}\bigg](1-s)^{n-m}\,\dd s.
\end{align*}
Now, we define
\begin{align*}
h(s) := (1-s)^{n-m},\quad g(s) := {n+1\choose m}(1-s)-{n\choose m}\quad\text{and}\quad L(s):=f^{\gamma}(\bar F^{-1}(s))
\end{align*}
for $s\in(0,1)$. Clearly, the function $h$ is measurable, strictly positive and bounded,  the function $g$ is linear, has negative slope and root at $s^*=m/(n+1)\in (0,1)$. Moreover, the function $L$ is positive and we shall argue that $L$ is also strictly concave. Indeed, by the chain rule its derivative equals
$$
L'(s) = -\gamma f^{\gamma-2}(\bar F^{-1}(s)) f'(\bar F^{-1}(s)),
$$
which is strictly decreasing because $-\gamma f(t)^{\gamma-2} f'(t)$ is increasing and $\bar F^{-1}(s)$ is decreasing.
Thus, Lemma~\ref{lem:FromBGTTTW} can be applied to deduce that
\begin{align*}
 g_{n+1,m}- g_{n,m}
&=  \int_0^1 L^{m-1}(s)g(s)h(s)\,\dd s\\
&>  \bigg({L(s^*)\over s^*}\bigg)^{m-1}\int_0^1 s^{m-1}g(s)h(s)\,\dd s\\
&=  \bigg({L(s^*)\over s^*}\bigg)^{m-1}\int_0^1 s^{m-1}\bigg[{n+1\choose m}(1-s)-{n\choose m}\bigg](1-s)^{n-m}\,\dd s\\
&=  \bigg({L(s^*)\over s^*}\bigg)^{m-1}{n+1\choose m}\bigg[\int_0^1 s^{m-1}(1-s)^{n+1-m}\,\dd s\\
&\hspace{5cm}-{n-m+1\over n+1}\int_0^1 s^{m-1}(1-s)^{n-m}\,\dd s\bigg]\\
&=  \bigg({L(s^*)\over s^*}\bigg)^{m-1}{n\choose m}\Big[B(m,n-m+2)-{n-m+1\over n+1}B(m,n-m+1)\Big],
\end{align*}
where $B(x,y)=\int_0^1 s^{x-1}(1-s)^{y-1}\,\dd s$, $x,y>0$, is Euler's beta function. Since $B(x,y+1)={y\over x+y}B(x,y)$, the last expression in square brackets is equal to zero. Hence,
$
g_{n+1,m}- g_{n,m} > 0
$,
which is the desired inequality.
\end{proof}

\begin{proof}[Proof of Theorem~\ref{theo:monoton}]
As explained above, we need to prove the strict inequality in~\eqref{eq:MonoBetaToShow} for all $m\in \{2,\ldots,n-1\}$.
Recall that $\alpha = 2\beta +d\geq 0$, and consider first the case when $\alpha\notin\{0,1\}$.  In particular, this means that $\gamma= \frac{\alpha}{\alpha-1}$ is well defined.
To apply Lemma~\ref{lem:monotone_g}, we need to verify that the function $\gamma f^{\gamma-2}(t) f'(t)$ is strictly decreasing in $t\in (-1,1)$, where  $f(t) = c_{1, \frac{\alpha - 1}{2}} (1-t^2)^{\frac{\alpha - 1}{2}}$.
We have
$$
\gamma f^{\gamma-2}(t) f'(t) = -\alpha c_{1,\frac{\alpha-1}{2}}^{1/(\alpha-1)} \frac{t}{\sqrt{1-t^2}}, \qquad t\in (-1,1),
$$
which is strictly decreasing because $\alpha>0$. Lemma~\ref{lem:monotone_g} thus yields $g_{n+1,m} > g_{n,m}$, which can be written as
$$
\binom n {m} I_{n,m}(\alpha) < \binom {n+1} {m} I_{n+1,m}(\alpha).
$$
This establishes~\eqref{eq:MonoBetaToShow} and completes the proof when $\alpha \notin \{0,1\}$.
The case when $\alpha=0$ occurs if $(d,\beta)=(2,-1)$. Note that Theorem~\ref{theo:monoton} becomes trivial in this case, but we prefer to prove~\eqref{eq:MonoBetaToShow} in all cases. Formula~\eqref{eq:I_n_m_proof_monoton1} simplifies as follows:
$$
\binom nm I_{n,m}(0) = \binom nm \int_{-1}^{+1} f(t) F^{n-m}(t) \dint t = \frac 1 {n-m+1} \binom nm.
$$
It follows that for all $m\in \{2,\ldots,n-1\}$,
$$
\frac{\binom {n+1}m I_{n+1,m} (0) }{\binom nm I_{n,m} (0)}
=
\frac{n+1}{n-m+2}>1.
$$

Let finally $\alpha=1$, which occurs if $(d,\beta)$ is $(3,-1)$ or $(2, -1/2)$. The expression for $I_{n,m}(\alpha)$ given in~\eqref{eq:I_n_m_proof_monoton1} simplifies as follows:
\begin{align*}
\binom nm I_{n,m} (1)
&=
\binom nm  \int_{-1}^{+1} c_{1,\frac{m-1}{2}} (1-t^2)^{\frac{m-1}{2}} \left(\frac {1+t}{2}\right)^{n-m} \dint t\\
&=c_{1,{m-1\over 2}}2^m{n\choose m}\int_0^1 u^{n-{m\over 2}+{1\over 2}-1}(1-u)^{{m\over 2}+{1\over 2}-1}\,dint u\\
&={c_{1,{m-1\over 2}}2^m\over m!(n-m)!}\Gamma\Big(n-{m-1\over 2}\Big)\Gamma\Big({m+1\over 2}\Big),
\end{align*}
where we computed the integral by using the substitution $u:= (1+t)/2$ and the properties of the beta and the gamma function. It follows that for all $m\in\{2,\ldots,n-1\}$,
$$
\frac{\binom {n+1}m I_{n+1,m} (1) }{\binom nm I_{n,m} (1) } = {(n-m)!\over(n-m+1)!}{\Gamma(n-{m-1\over 2}+1)\over\Gamma(n-{m-1\over 2})} = \frac{n-\frac{m-1}{2}}{n-m+1} > 1.
$$
This completes the proof of Theorem~\ref{theo:monoton}.
\end{proof}

\begin{proof}[Proof of Theorem~\ref{theo:monoton_prime}]
Observe that $\alpha=2\beta-d >0$.
We take $\gamma = \frac{\alpha}{\alpha+1}$ and $f(t) = \tilde c_{1, \frac{\alpha + 1}{2}} (1+t^2)^{-\frac{\alpha + 1}{2}}$, $t\in\R$. Then,
$$
\gamma f^{\gamma-2}(t) f'(t) = -\alpha \tilde c_{1, \frac{\alpha + 1}{2}}^{-1/(\alpha+1)} \frac{t}{\sqrt{1+t^2}}, \qquad t\in \R,
$$
which is strictly decreasing in $t$. An application of Lemma~\ref{lem:monotone_g} yields $g_{n+1,m} > g_{n,m}$ for all $m\in \{2,\ldots,n-1\}$ and thus
$$
\binom n {m} \tilde I_{n,m}(\alpha) < \binom {n+1} {m} \tilde I_{n+1,m}(\alpha),
$$
which establishes~\eqref{eq:MonoBetaPrimeToShow} and  completes the argument.
\end{proof}

\subsection{Expected intrinsic volumes of tangent cones}

In this section we give proofs of Theorems~\ref{theo:expected_conic_tangent}, \ref{theo:internal}, \ref{theo:expected_conic_tangent_prime} and~\ref{theo:internal_prime}.

\begin{proof}[Proof of Theorem~\ref{theo:expected_conic_tangent}]
We shall use the indices $k\in\{1,\ldots,d\}$ and $\ell := j+1 \in \{k,\ldots,d\}$.
Then, by  the property of the Grassmann angles \eqref{eq:ConicalIntVolGrassmannAngle}, we have
\begin{align*}
2\E\left[h_{\ell+1}(T(G,P_{n,d}^\beta))\ind_{\{G\in\cF_{k-1}(P_{n,d}^\beta)\}}\right]
=
\P\left[T(G,P_{n,d}^\beta) \cap L_{d-\ell} \neq \{0\}\text{ and }G\in\cF_{k-1}(P_{n,d}^\beta)\right],
\end{align*}
where $L_{d-\ell}\in G(d, d-\ell)$ is a uniformly distributed linear subspace that is independent of everything else. Equivalently,
\begin{align*}
&\P\left[ G\in\cF_{k-1}(P_{n,d}^\beta)\right] -2\E\left[h_{\ell+1}(T(G,P_{n,d}^\beta))\ind_{\{G\in\cF_{k-1}(P_{n,d}^\beta)\}}\right] \\
&= \P\left[T(G,P_{n,d}^\beta) \cap L_{d-\ell} = \{0\}\text{ and }G\in\cF_{k-1}(P_{n,d}^\beta)\right].
\end{align*}
Let $\Pi_\ell$ denote the orthogonal projection on the $\ell$-dimensional linear subspace $L_{d-\ell}^\bot$.
It is known, see the proof of~(3.1) on p.~298 of~\cite{GruenbaumGA}, that on the event $G\in\cF_{k-1}(P_{n,d}^\beta)$, the probability that the intersection of $T(G,P_{n,d}^\beta)$ and $L_{d-\ell}$ is the null space $\{0\}$ is the same as the probability that $\Pi_\ell G$ is a $(k-1)$-face of the projected polytope $\Pi_\ell P_{n,d}^\beta$. Moreover, if $G\notin\cF_{k-1}(P_{n,d}^\beta)$, then $G$ a.s.\ contains an interior point of $P_{n,d}^\beta$, which under the projection $\Pi_\ell$ is mapped to a relative interior point of $\Pi_\ell P_{n,d}^\beta$, implying that $\Pi_\ell G$ cannot be a $(k-1)$-face in this case.
 It follows that
\begin{equation*}
\begin{split}
&\P\left[ G\in\cF_{k-1}(P_{n,d}^\beta)\right] -2\E\left[h_{\ell+1}(T(G,P_{n,d}^\beta))\ind_{\{G\in\cF_{k-1}(P_{n,d}^\beta)\}}\right]\\
&= \P\left[\Pi_\ell G \in \cF_{k-1}(\Pi_\ell P_{n,d}^\beta)\text{ and }G\in\cF_{k-1}(P_{n,d}^\beta)\right]\\
&= \P\left[\Pi_\ell G \in \cF_{k-1}(\Pi_\ell P_{n,d}^\beta)\right]\\
&=
\binom{n}{k}^{-1}\E f_{k-1} (\Pi_\ell P_{n,d}^\beta).
\end{split}
\end{equation*}
Since the probability law of $P_{n,d}^\beta$ is rotationally invariant, we can replace $L_{d-\ell}$ by any deterministic linear subspace of the same dimension. For example, we can take $L_{d-\ell} =  \lin (e_{\ell+1},\ldots,e_d)$, where $e_1,\ldots,e_d$ is the standard orthonormal basis in $\R^d$. In this case, $\Pi_\ell:\R^d\to\R^\ell$ becomes the orthogonal projection from $\R^d$ to $\R^\ell$ (which is identified with $L_{d-\ell}^\bot = \lin (e_1,\ldots, e_\ell))$ given by
$$
\Pi_\ell(x_1,\ldots,x_d) := (x_1,\ldots,x_\ell).
$$
By Lemma~\ref{lem:projection} (a), the random polytope $\Pi_\ell P_{n,d}^\beta$ has the same distribution as $P_{n,\ell}^{\beta+\frac{d-\ell}{2}}$. It follows that
\begin{equation}\label{eq:P_minus_2h}
\P\left[ G\in\cF_{k-1}(P_{n,d}^\beta)\right] -2\E\left[h_{\ell+1}(T(G,P_{n,d}^\beta))\ind_{\{G\in\cF_{k-1}(P_{n,d}^\beta)\}}\right]
=
\binom{n}{k}^{-1}\E f_{k-1} \big(P_{n,\ell}^{\beta+\frac{d-\ell}{2}}\big).
\end{equation}
These considerations are valid for $\ell  \in \{k,\ldots,d\}$, where we recall the convention $h_{d+1}(T(G,P_{n,d}^\beta))=0$.
Applying Theorem~\ref{theo:f_vect} to the right-hand side of \eqref{eq:P_minus_2h}, we can write
\begin{multline}\label{eq:eq_1}
\E\left[h_{\ell+1}(T(G,P_{n,d}^\beta))\ind_{\{G\in\cF_{k-1}(P_{n,d}^\beta)\}}\right]
=
{1\over 2} \P\left[ G\in\cF_{k-1}(P_{n,d}^\beta)\right]
\\
-\frac 1 {\binom{n}{k}} \sum_{s=0}^\infty \binom n {\ell-2s} \binom {\ell-2s}{k} I_{n,\ell-2s}(2\beta+d) J_{\ell-2s,k}\left(\beta + s+\frac{d - \ell + 1}2\right),
\end{multline}
for all $\ell \in \{k,\ldots,d\}$.  Let us argue that~\eqref{eq:eq_1} holds true in the slightly larger range $\ell \in \{k-2,\ldots,d\}$.
On the event $G\in\cF_{k-1}(P_{n,d}^\beta)$ we have
$$
\upsilon_{m}(T(G,P_{n,d}^\beta)) = 0 \quad\text{ for }\quad m \in \{0,\ldots, k-2\},
$$
because the cone $T(G,P_{n,d}^\beta)$ has a $(k-1)$-dimensional lineality space. It follows from~\eqref{eq:h_def} and the Gauss--Bonnet relation~\eqref{eq:gauss_bonnet} that on the event $G\in\cF_{k-1}(P_{n,d}^\beta)$,
$$
h_{k-1}(T(G,P_{n,d}^\beta)) = h_{k}(T(G,P_{n,d}^\beta)) = \frac 12.
$$
Hence, \eqref{eq:eq_1} is true for $\ell\in \{k-1,k-2\}$ because then the sum in~\eqref{eq:eq_1} is empty and both sides of~\eqref{eq:eq_1} are equal to $\frac 1 2 \P[ G\in\cF_{k-1}(P_{n,d}^\beta)]$.

Altogether, we have shown that~\eqref{eq:eq_1} holds true for $\ell \in \{k-2,\ldots,d\}$. Inserting $\ell-2$ in place of $\ell$ and shifting the summation index $s$ yields the identity
\begin{multline}\label{eq:eq_2}
\E\left[h_{\ell-1}(T(G,P_{n,d}^\beta))\ind_{\{G\in\cF_{k-1}(P_{n,d}^\beta)\}}\right]
=
{1\over 2} \P\left[ G\in\cF_{k-1}(P_{n,d}^\beta)\right]
\\-
\frac 1 {\binom{n}{k}} \sum_{s=1}^\infty \binom n {\ell-2s} \binom {\ell-2s}{k} I_{n,\ell-2s}(2\beta+d) J_{\ell-2s,k}\left(\beta + s+\frac{d - \ell + 1}2\right),
\end{multline}
which is valid for $\ell  \in \{k,\ldots,d+2\}$.

For the rest of the argument, let $\ell \in \{k,\ldots,d\}$. It follows from~\eqref{eq:h_def} that
$$
\upsilon_{\ell-1}(T(G,P_{n,d}^\beta)) = h_{\ell-1}(T(G,P_{n,d}^\beta)) - h_{\ell+1}(T(G,P_{n,d}^\beta)),
$$
with the usual convention that $h_{d+1}(C) = 0$ if $C\subset \R^d$.
Subtracting~\eqref{eq:eq_1} from~\eqref{eq:eq_2}, we see that on the right-hand side only the term with $s=0$ remains, while the left-hand side reduces to
\begin{multline*}
\E\left[h_{\ell-1}(T(G,P_{n,d}^\beta))\ind_{\{G\in\cF_{k-1}(P_{n,d}^\beta)\}}\right]
-
\E\left[h_{\ell+1}(T(G,P_{n,d}^\beta))\ind_{\{G\in\cF_{k-1}(P_{n,d}^\beta)\}}\right]
\\
=
\E\left[\upsilon_{\ell-1}(T(G,P_{n,d}^\beta))\ind_{\{G\in\cF_{k-1}(P_{n,d}^\beta)\}}\right].
\end{multline*}
We thus arrive at
\begin{align*}
\E\left[\upsilon_{\ell-1}(T(G,P_{n,d}^\beta))\ind_{\{G\in\cF_{k-1}(P_{n,d}^\beta)\}}\right]
&=
\frac {\binom n {\ell} \binom {\ell}{k}} {\binom{n}{k}}  I_{n,\ell}(2\beta+d) J_{\ell,k}\left(\beta + \frac{d - \ell + 1}2\right)\\
&={n-k\choose \ell-k}I_{n,\ell}(2\beta+d) J_{\ell,k}\left(\beta + \frac{d - \ell + 1}2\right).
\end{align*}
It remains to recall our convention that $T(G,P_{n,d}^\beta)=\R^d$ if $G\notin\cF_{k-1}(P_{n,d}^\beta)$ and to note that, by definition of the conic intrinsic volumes, $\upsilon_{\ell-1}(\R^d)$ vanishes as long as $\ell\leq d$.  This implies that
\begin{align*}
\E \upsilon_{\ell-1}(T(G,P_{n,d}^\beta)) &= \E\left[\upsilon_{\ell-1}(T(G,P_{n,d}^\beta))\ind_{\{G\in\cF_{k-1}(P_{n,d}^\beta)\}}\right]+\E\left[\upsilon_{\ell-1}(T(G,P_{n,d}^\beta))\ind_{\{G\notin\cF_{k-1}(P_{n,d}^\beta)\}}\right]\\
&={n-k\choose \ell-k}I_{n,\ell}(2\beta+d) J_{\ell,k}\left(\beta + \frac{d - \ell + 1}2\right),
\end{align*}
for all $\ell \in \{k,\ldots,d\}$.
Recalling that $\ell = j+1$, we arrive at the required formula.
\end{proof}

\begin{proof}[Proof of Theorem~\ref{theo:internal}]
Let first $k\in\{1,\ldots,d-1\}$.
Taking $\ell = d-1$ in~\eqref{eq:P_minus_2h}, we obtain
\begin{align*}
\E\left[h_{d}(T(G,P_{n,d}^\beta))\ind_{\{G\in\cF_{k-1}(P_{n,d}^\beta)\}}\right]
&=
\frac 12 \P\left[G\in\cF_{k-1}(P_{n,d}^\beta)\right]
-
\frac{1}{2\binom{n}{k}} \E f_{k-1} \big(P_{n,d-1}^{\beta+\frac{1}{2}}\big)\\
&=
\frac{1}{2\binom{n}{k}} \left(
\E f_{k-1} \big(P_{n,d}^{\beta}\big)
-
\E f_{k-1} \big(P_{n,d-1}^{\beta+\frac{1}{2}}\big)\right).
\end{align*}
Recall that $h_d(C) = \nu_d(C)$ is the internal angle of a $d$-dimensional cone $C$.
Applying Theorem~\ref{theo:f_vect} two times (which would fail in the case $k=d$) and introducing the summation indices $m:=d-2s$ and $m:=d-2s-1$ in the resulting sums, we arrive at
\begin{align*}
\E \left[\beta(G, P_{n,d}^\beta) \ind_{\{G\in\cF_{k-1}(P_{n,d}^\beta)\}} \right]
&=
\frac 1 {\binom nk} \sum_{m=k}^d (-1)^{d-m} \binom nm \binom mk I_{n,m}(2\beta+d) J_{m,k}\left(\beta +\frac{d-m+1}{2}\right)\\
&=
\sum_{m=k}^d (-1)^{d-m} \binom {n-k}{m-k} I_{n,m}(2\beta+d) J_{m,k}\left(\beta +\frac{d-m+1}{2}\right),
\end{align*}
which proves the first identity of the theorem. To prove the second identity, also assuming that $k\in\{1,\ldots,d-1\}$, we recall that on the event $G\notin\cF_{k-1}(P_{n,d}^\beta)$, the internal angle $\beta(G, P_{n,d}^\beta)$ is defined to be $1$, hence
\begin{align*}
\E \beta(G, P_{n,d}^\beta)
&=
\E \left[\beta(G, P_{n,d}^\beta) \ind_{\{G\in\cF_{k-1}(P_{n,d}^\beta)\}} \right]
+
\E \left[\ind_{\{G\notin\cF_{k-1}(P_{n,d}^\beta)\}} \right]\\
&=
\frac 12 \P\left[G\in\cF_{k-1}(P_{n,d}^\beta)\right]
-
\frac{1}{2\binom{n}{k}} \E f_{k-1} \big(P_{n,d-1}^{\beta+\frac{1}{2}}\big)
+
1 - \P\left[G\in\cF_{k-1}(P_{n,d}^\beta)\right]\\
&=
1- \frac 1{2\binom{n}{k}} \left(
\E f_{k-1} \big(P_{n,d}^{\beta}\big)
+
\E f_{k-1} \big(P_{n,d-1}^{\beta+\frac{1}{2}}\big)\right).
\end{align*}
Applying Theorem~\ref{theo:f_vect} twice, we obtain
$$
\E \beta(G, P_{n,d}^\beta)
=
1- \frac 1 {\binom nk} \sum_{m=k}^d \binom nm \binom mk I_{n,m}(2\beta+d) J_{m,k}\left(\beta +\frac{d-m+1}{2}\right),
$$
which proves the second identity of the theorem.

The remaining case $k=d$ is easy to treat because then we have $\beta(G, P_{n,d}^\beta) = 1/2$ on the event $G\in\cF_{d-1}(P_{n,d}^\beta)$ and  $\beta(G, P_{n,d}^\beta) = 1$ otherwise. It follows that
$$
\E \left[\beta(G, P_{n,d}^\beta) \ind_{\{G\in\cF_{d-1}(P_{n,d}^\beta)\}} \right]
=
\frac 12 \P\left[G\in\cF_{d-1}(P_{n,d}^\beta)\right]
=
\frac {\E f_{d-1} \big(P_{n,d}^{\beta}\big)}{2\binom{n}{d}}
=
I_{n,d}(2\beta+d)
$$
by Remark~\ref{rem:SpecialCasesMainResultBeta}. This proves the first identity. The second one can be established analogously.
\end{proof}

\begin{proof}[Proof of Theorem \ref{theo:expected_conic_tangent_prime}]
As in the beta case, see \eqref{eq:P_minus_2h}, one shows that
$$
\E\left[h_{\ell+1}(T(G,\tilde P_{n,d}^\beta))\ind_{\{G\in\cF_{k-1}(\tilde P_{n,d}^\beta)\}}\right]
=
\frac 12 \P\left[ G\in\cF_{k-1}(P_{n,d}^\beta)\right]  - \frac1{2\binom{n}{k}} \E f_{k-1} \big(\tilde P_{n,\ell}^{\beta-\frac{d-\ell}{2}}\big),
$$
for all $\ell \in \{k,\ldots,d\}$,
where we used Lemma \ref{lem:projection} (b) instead of part (a). Applying Theorem \ref{theo:f_vect_prime}  to the right-hand side, we can write
\begin{multline}\label{eq:eq_1_prime}
\E\left[h_{\ell+1}(T(G,\tilde P_{n,d}^\beta))\ind_{\{G\in\cF_{k-1}(\tilde P_{n,d}^\beta)\}}\right]
=
{1\over 2} \P\left[ G\in\cF_{k-1}(\tilde P_{n,d}^\beta)\right]
\\
-\frac 1 {\binom{n}{k}} \sum_{s=0}^\infty \binom n {\ell-2s} \binom {\ell-2s}{k} \tilde I_{n,\ell-2s}(2\beta - d) \tilde J_{\ell-2s,k}\left(\beta  - s - \frac{d - \ell + 1}2\right),
\end{multline}
for all $\ell \in \{k,\ldots,d\}$. From this point the proof can be completed as the one of Theorem~\ref{theo:expected_conic_tangent}.
\end{proof}

\begin{proof}[Proof of Theorem~\ref{theo:internal_prime}]
Since this is analogous to the proof of Theorem~\ref{theo:internal}, we refrain from presenting the details.
\end{proof}

\subsection{The Poisson limit for beta' polytopes}\label{sec:proof_poisson_limit}
In this section we prove Theorem~\ref{theo:f_vect_poisson}.
\begin{lemma}\label{lem:asymptotics_I}
Fix some $\alpha>0$ and $\beta>0$.  As $n\to\infty$, we have
\begin{align*}
A_n
:=
\int_{-\infty}^{+\infty}
(1+t^2)^{-\frac {\beta + 1}{2}}
\left(\int_{-\infty}^t \tilde c_{1, \frac{\alpha + 1}{2}} (1+s^2)^{-\frac{\alpha + 1}{2}}\dd s\right)^{n} \dd t
\sim
\frac {\Gamma(\beta/\alpha)}{\alpha} \left(\frac{\alpha}{\tilde c_{1, \frac{\alpha + 1}{2}}}\right)^{\beta/\alpha} n^{-\beta/\alpha}.
\end{align*}
\end{lemma}
\begin{proof}
To simplify the notation, we shall write $C_\alpha$ for $\tilde c_{1, \frac{\alpha + 1}{2}}$.
Using the change of variables $t= n^{1/\alpha} u$, we have
\begin{align*}
A_n
=
 n^{1/\alpha} \int_{-\infty}^{+\infty}
(1 + n^{2/\alpha} u^2)^{-\frac {\beta + 1}{2}}
\left(1 - \int_{n^{1/\alpha}u}^{+\infty} C_\alpha (1+s^2)^{-\frac{\alpha + 1}{2}}\dd s\right)^{n} \dd u
=
n^{-\frac{\beta}{\alpha}}
\int_{-\infty}^{+\infty}
g_n(u) \dd u,
\end{align*}
where
$$
g_n(u) =
(n^{-2/\alpha} +  u^2)^{-\frac {\beta + 1}{2}}
\left(1 - \int_{n^{1/\alpha}u}^{+\infty} C_\alpha (1+s^2)^{-\frac{\alpha + 1}{2}}\dd s\right)^{n}.
$$
Applying the L'Hospital rule, it is easy to check that for every positive $u>0$,
\begin{equation}\label{eq:asympt_tail_integral}
\int_{n^{1/\alpha}u}^{+\infty} C_\alpha (1+s^2)^{-\frac{\alpha + 1}{2}}\dd s
\sim
\alpha^{-1} C_\alpha (n^{1/\alpha} u)^{-\alpha}
=
\alpha^{-1} C_\alpha u^{-\alpha} n^{-1},
\end{equation}
as $n\to\infty$. It follows that for all $u>0$,
$$
\lim_{n\to\infty} g_n(u)
=
\begin{cases}
u^{-(\beta + 1)}
\eee^{-\alpha^{-1}C_\alpha u^{-\alpha}}, & \text{ if } u>0,\\
0, &\text{ if } u\leq 0.
\end{cases}
$$
In fact, the case $u \leq 0$ follows from the observation that
\begin{equation}\label{eq:tech}
\int_{n^{1/\alpha}u}^{+\infty} C_\alpha (1+s^2)^{-\frac{\alpha + 1}{2}}\dd s \geq 1/2, \quad u\leq 0.
\end{equation}
Assuming that we can apply the dominated convergence theorem, we arrive at
$$
A_n
=
n^{-\frac{\beta}{\alpha}}
\int_{-\infty}^{+\infty}
g_n(u) \dd u
\sim
n^{-\frac{\beta}{\alpha}}
\int_0^{+\infty} u^{-(\beta + 1)}
\eee^{-\alpha^{-1} C_\alpha u^{-\alpha}} \dd u
=
\frac {\Gamma(\beta/\alpha)}{\alpha} \left(\frac{\alpha}{C_\alpha}\right)^{\beta/\alpha} n^{-\beta/\alpha},
$$
which is the required claim.

Let us justify the use of the dominated convergence theorem above. First of all, observe that $g_n(u)\geq 0$ by definition. Further, we have
$
g_n(u) \leq  |u|^{-(\beta + 1)}
$,
with the right-hand side being integrable over $\{|u|\geq 1\}$. To construct an integrable bound for $u\in (-1,0)$, observe that according to \eqref{eq:tech}, in this range we have  $g_n(u) \leq n^{\frac{\beta+1}{\alpha}} 2^{-n}$, which in turn is bounded by a constant.
Finally, in the case when $u\in (0,1)$, we use the estimate
\begin{equation}\label{eq:est_tail111}
\int_{n^{1/\alpha}u}^{+\infty} C_\alpha (1+s^2)^{-\frac{\alpha + 1}{2}}\dd s \geq c_1(1+n^{2/\alpha}u^2)^{-\alpha/2}, \qquad u\geq 0,
\end{equation}
valid for some constant $c_1>0$. To prove this estimate, note that as functions of $n^{1/\alpha}u$, both expressions are continuous and non-zero on $[0,\infty)$. Since the quotient of both expressions tends to a non-zero constant as $n^{1/\alpha}u\to \infty$, see the asymptotic equivalence~\eqref{eq:asympt_tail_integral}, we can conclude~\eqref{eq:est_tail111}.  An estimate similar to~\eqref{eq:est_tail111}  was used in~\cite[Equation (1)]{BonnetEtAlThresholds}.
Now, we distinguish the two cases $u^2>n^{-2/\alpha}$ and $0< u^2\leq n^{-2/\alpha}$. In the first case, that is, if $u^2>n^{-2/\alpha}$, we  use the inequality $(1-x)^n \leq \eee^{-nx}$, $0\leq x <1$, to deduce that
\begin{multline*}
g_n(u)
\leq
u^{-(\beta+1)}\exp\left\{-n \int_{n^{1/\alpha}u}^{+\infty} C_\alpha (1+s^2)^{-\frac{\alpha + 1}{2}}\dd s \right\}
\leq
u^{-(\beta+1)}\exp\{-c_1 n (1+n^{2/\alpha}u^2)^{-\alpha/2}\}
\\
\leq
u^{-(\beta+1)}\exp\{-c_1 n (2 n^{2/\alpha}u^2)^{-\alpha/2}\}
=
u^{-(\beta+1)}\exp\{-c_2 u^{-\alpha}\},
\end{multline*}
where $c_2>0$ is another constant. On the other hand, if $0 < u^2\leq n^{-2/\alpha}$, then, again using the inequality $(1-x)^n \leq \eee^{-nx}$, $0\leq x <1$, we have that
$$
g_n(u)
\leq
n^{\beta+1\over \alpha}\exp\{-c_1 n (1+n^{2/\alpha}u^2)^{-\alpha/2}\}
\leq
n^{\beta+1\over \alpha}\exp\{-c_1 n 2^{-\alpha/2}\}
=
n^{\beta+1\over \alpha}\exp\{-c_3 n\} \leq c_4
$$
with suitable constants $c_3,c_4>0$. Altogether this shows that for $u\in(0,1)$, we have the upper bound
$$
g_n(u) \leq \max\{c_4,u^{-(\beta+1)}\exp\{-c_2u^{-\alpha}\}\} \leq c_5
$$
with some constant $c_5>0$. The proof is thus complete.
\end{proof}

\begin{proof}[Proof of Theorem~\ref{theo:f_vect_poisson}]
It was shown in~\cite{convex_hull_sphere}  that
\begin{equation}\label{eq:E_f_k_to_E_f_k_Poi}
\E f_k(\conv \Pi_{d,\alpha})
=
\lim_{n\to\infty}\E f_k\left(\tilde P_{n,d}^{\frac{d+\alpha}{2}}\right).
\end{equation}
In fact, only the case $\alpha=1$ was considered in~\cite{convex_hull_sphere}, but as we explained at the end of Section~\ref{subsec:convex_hulls_half_sphere},  the same proof applies to any $\alpha>0$.
So, we have to compute the limit on the right-hand side of~\eqref{eq:E_f_k_to_E_f_k_Poi}. It follows from Lemma~\ref{lem:asymptotics_I} with $\beta = \alpha m$ that for every fixed $m\in\N$, the quantity $\tilde I_{n,m}(\alpha)$ defined in~\eqref{eq:I_definition_prime} satisfies
\begin{align*}
\tilde I_{n,m}(\alpha)
&=
\int_{-\infty}^{+\infty} \tilde c_{1, \frac {\alpha m + 1}{2}}
(1+t^2)^{-\frac {\alpha m + 1}{2}}
\left(\int_{-\infty}^t \tilde c_{1, \frac{\alpha + 1}{2}} (1+s^2)^{-\frac{\alpha + 1}{2}}\dd s\right)^{n-m} \dd t\\
&\sim
\tilde c_{1, \frac {\alpha m + 1}{2}} \frac {\Gamma(m)}{\alpha} \left(\frac{\alpha}{\tilde c_{1, \frac{\alpha + 1}{2}}}\right)^{m} n^{-m},
\end{align*}
as $n\to\infty$.
By Theorem~\ref{theo:f_vect_prime} and the above asymptotics with $m=d-2s$ for all $s\in \N_0$ with $d-2s\geq k+1$, we arrive at
\begin{align*}
\E  f_k\left(\tilde P_{n,d}^{\frac{d+\alpha}{2}}\right)
&=
2 \sum_{s=0}^\infty \binom n {d-2s} \binom {d-2s}{k+1} \tilde I_{n,d-2s}(\alpha) \tilde J_{d-2s,k+1}\left(\frac{d-2s-1+\alpha}{2}\right)\\
&=
2 \sum_{\substack{m\in \{k+1,\ldots,d\}\\ m\equiv d \Mod{2}}} \binom n {m} \binom {m}{k+1} \tilde I_{n,m}(\alpha) \tilde J_{m,k+1}\left(\frac{m-1+\alpha}{2}\right)\\
&\sim
2 \sum_{\substack{m\in \{k+1,\ldots,d\}\\ m\equiv d \Mod{2}}} \frac{\tilde c_{1, \frac{\alpha m +1}{2}}}{(\tilde c_{1, \frac{\alpha+1}{2}})^{m}} \cdot \frac{\alpha^{m-1}}{m} \cdot \binom {m}{k+1}\tilde J_{m,k+1}\left(\frac{m-1+\alpha}{2}\right),
\end{align*}
as $n\to\infty$.
Note that we restricted the summation to the range  $m\in \{k+1,\ldots,d\}$ because terms with $m\leq k$ vanish.
To complete the proof of the theorem, recall that $\tilde c_{1,\gamma} = \frac{\Gamma (\gamma) }{\sqrt{\pi} \Gamma( \gamma - \frac{1}{2})}$ by~\eqref{eq:def_f_beta_prime}.
\end{proof}

\subsection{Asymptotics for the \texorpdfstring{$f$}{f}-vector of beta polytopes}
In this section we prove Theorem~\ref{theo:f_vector_asympt_beta}. The proof is prepared with the following auxiliary estimate.

\begin{lemma}\label{lem:asympt_beta}
Fix some $\alpha>-1$ and $\beta>-1$. As $n\to\infty$, we have
\begin{align*}
B_n:= \int_{-1}^{+1} (1-t^2)^{\frac{\beta-1}{2}} \left(\int_{-1}^t c_{1,\frac{\alpha-1}{2}} (1-s^2)^{\frac{\alpha-1}{2}} \dd s\right)^{n} \dd t
\sim \frac{n^{-\frac{\beta+1}{\alpha+1}}}{1 + \alpha}
\left(\frac{1+\alpha}{c_{1,\frac{\alpha-1}{2}}}\right)^{\frac{\beta+1}{\alpha+1}}
\Gamma\left(\frac{1 + \beta}{1 + \alpha}\right).
\end{align*}
\end{lemma}
\begin{proof}
Write $C_\alpha:= c_{1,\frac{\alpha-1}{2}}$. Using the change of variables $1-t = u n^{-\frac{2}{\alpha+1}}$, we obtain
$$
B_n = n^{-\frac{\beta+1}{\alpha+1}} \int_0^{2n^{\frac{2}{\alpha+1}}} g_n(u) \dd u,
$$
where $g_n$ is given by
$$
g_n(u) =
n^{\frac{\beta-1}{\alpha+1}} \left(1-\left(1- u n^{-\frac{2}{\alpha+1}}\right)^2\right)^{\frac{\beta-1}{2}} \left(1-\int_{1 - un^{-\frac{2}{\alpha+1}}}^1 C_\alpha (1-s^2)^{\frac{\alpha-1}{2}} \dd s \right)^n.
$$
With the rule of L'Hospital one easily checks that
\begin{equation}\label{eq:asympt_proof_beta}
\int_{1- un^{-\frac{2}{\alpha+1}}}^1 C_\alpha (1-s^2)^{\frac{\alpha-1}{2}} \dd s
\sim
C_\alpha 2^{\frac{\alpha+1}{2}} (\alpha+1)^{-1} (un^{-\frac{2}{\alpha+1}})^{\frac{\alpha+1}{2}}={C_\alpha\over\alpha+1}(2u)^{\alpha+1\over 2}n^{-1}.
\end{equation}
It follows that for every $u>0$ we have
$$
\lim_{n\to\infty} g_n(u) = (2u)^{\frac{\beta-1}{2}} \exp\left\{-{C_\alpha\over\alpha+1} (2u)^{\frac{\alpha+1}{2}} \right\}.
$$
Assuming that the dominated convergence theorem is applicable, we arrive at
\begin{align*}
B_n = n^{-\frac{\beta+1}{\alpha+1}} \int_0^{\infty} g_n(u) \ind_{\big(0,2n^{\frac{2}{\alpha+1}}\big)}(u)\, \dd u
\sim n^{-\frac{\beta+1}{\alpha+1}} \int_0^{\infty}(2u)^{\frac{\beta-1}{2}} \exp\left\{-{C_\alpha\over\alpha+1} (2u)^{\frac{\alpha+1}{2}} \right\}\dd u.
\end{align*}
Evaluation of the integral yields
$$
\int_0^{\infty}(2u)^{\frac{\beta-1}{2}} \exp\left\{-{C_\alpha\over\alpha+1} (2u)^{\frac{\alpha+1}{2}} \right\}\dd u = {1\over\alpha+1}\left({\alpha+1\over C_\alpha}\right)^{\beta+1\over \alpha+1}\Gamma\left({\beta+1\over \alpha+1}\right)
$$
and thus the desired asymptotic formula.

To justify the interchanging of the integral and the limit, it suffices to show that there is a  sufficiently small $\delta>0$ such that
\begin{equation}\label{eq:est_dominated1}
0\leq g_n(u) \leq h(u), \quad  \text{ for all } u\in (0,(2-\delta) n^{\frac 2 {\alpha+1}}),
\end{equation}
where $h(u)$ is integrable, and that
\begin{equation}\label{eq:est_dominated2}
\lim_{n\to\infty} \int_{(2-\delta)n^{\frac 2 {\alpha+1}}}^{2n^{\frac 2 {\alpha+1}}} g_n(u) \dd u = 0.
\end{equation}

Clearly, $g_n(u)\geq 0$. To prove the upper estimate in~\eqref{eq:est_dominated1}, observe first that there is $c_1>0$ such that
$$
n^{\frac{\beta-1}{\alpha+1}}
\left(1-\left(1- u n^{-\frac{2}{\alpha+1}}\right)^2\right)^{\frac{\beta-1}{2}}
=
u^{\frac {\beta-1}2} \left(2 - u n^{-\frac{2}{\alpha+1}}\right)^{\frac{\beta-1}{2}}
\leq
c_1 u^{\frac{\beta-1}{2}}
$$
for all $u\in (0,(2-\delta) n^{\frac 2 {\alpha+1}})$.
Namely, we can take $c_1 = 2^{\frac{\beta-1}{2}}$ if $\beta\geq 1$ and $c_1 =\delta^{\frac{\beta-1}{2}}$ if $\beta \leq 1$.
Further, there exists a constant $c_2>0$  such that
$$
\int_{1 - un^{-\frac{2}{\alpha+1}}}^1 C_\alpha (1-s^2)^{\frac{\alpha-1}{2}} \dd s
\geq
c_2\Big(1-(1-un^{-{2\over \alpha+1}})\Big)^{\alpha+1\over 2}
=
c_2 u^{\frac{\alpha+1}{2}} n^{-1},
$$
for all $u\in (0,2 n^{\frac 2 {\alpha+1}}]$.  Indeed, the quotient of both expressions converges to a non-zero constant as $un^{-\frac{2}{\alpha+1}}\to 0$; see Relation~\eqref{eq:asympt_proof_beta}. Furthermore, both expressions are continuous, non-vanishing functions of the argument $un^{-\frac{2}{\alpha+1}} \in (0,2]$. This implies the required bound.   A similar bound was also used in~\cite[Lemma 2.2]{BonnetEtAlThresholds}. Now, if $u\in (0,(2-\delta) n^{\frac 2 {\alpha+1}})$, then taking the above estimates together and using the elementary inequality $(1-x)^n\leq \eee^{-nx}$, $0\leq x<1$, we arrive at
$$
g_n(u)
\leq
c_1 u^{\frac{\beta-1}{2}} \exp\left\{- n\int_{1 - un^{-\frac{2}{\alpha+1}}}^1 C_\alpha (1-s^2)^{\frac{\alpha-1}{2}} \dd s \right\}
\leq
c_1 u^{\frac{\beta-1}{2}}\exp\{-c_2 u^{\frac{\alpha+1}{2}}\},
$$
which proves the integrable bound stated in~\eqref{eq:est_dominated1} for every $\delta\in (0,2)$.

Let us prove~\eqref{eq:est_dominated2}. First of all, we have
$$
n^{\frac{\beta-1}{\alpha+1}}
\left(1-\left(1- u n^{-\frac{2}{\alpha+1}}\right)^2\right)^{\frac{\beta-1}{2}}
=
u^{\frac {\beta-1}2} \left(2 - u n^{-\frac{2}{\alpha+1}}\right)^{\frac{\beta-1}{2}}.
$$
Unfortunately, this becomes infinite at $un^{-\frac{2}{\alpha+1}} =2$ if $\beta<1$.  Let us choose $\delta>0$ so small that for all  $u\in ((2-\delta) n^{\frac 2 {\alpha+1}}, 2n^{\frac 2 {\alpha+1}})$,
$$
1-\int_{1 - un^{-\frac{2}{\alpha+1}}}^1 C_\alpha (1-s^2)^{\frac{\alpha-1}{2}} \dd s
=
\int_{-1}^{-1+ (2 - un^{-\frac{2}{\alpha+1}})} C_\alpha (1-s^2)^{\frac{\alpha-1}{2}} \dd s
\leq \frac 12.
$$
This is possible because the integral converges to $0$ as $(2 - un^{-\frac{2}{\alpha+1}})\to 0$.  Recalling the definition of $g_n$ and taking the above estimates together, we arrive at
$$
\int_{(2-\delta)n^{\frac 2 {\alpha+1}}}^{2n^{\frac 2 {\alpha+1}}} g_n(u) \dd u
\leq
\int_{(2-\delta)n^{\frac 2 {\alpha+1}}}^{2n^{\frac 2 {\alpha+1}}} u^{\frac{\beta-1}{2}}  \left(2 - u n^{-\frac{2}{\alpha+1}}\right)^{\frac{\beta-1}{2}} 2^{-n} \dd u
=
n^{\frac{\beta+1}{\alpha+1}} 2^{-n} \int_{2-\delta}^2 v^{\frac{\beta-1}{2}} (2-v)^{\frac{\beta-1}{2}} \dd v,
$$
which converges to $0$, as $n\to\infty$. This completes the proof of~\eqref{eq:est_dominated2}.
\end{proof}

\begin{proof}[Proof of Theorem~\ref{theo:f_vector_asympt_beta}]
It follows from Lemma~\ref{lem:asympt_beta} with $\beta=\alpha k$ that
\begin{align*}
I_{n,k}(\alpha)
&=
\int_{-1}^{+1}  c_{1, \frac {\alpha k -1}{2}}
(1-t^2)^{\frac {\alpha k - 1}{2}} \left(\int_{-1}^t c_{1, \frac{\alpha - 1}{2}} (1-s^2)^{\frac{\alpha - 1}{2}}\dd s\right)^{n-k} \dd t\\
&\sim
n^{-\frac{\alpha k + 1}{\alpha+1}}
\frac{c_{1, \frac {\alpha k -1}{2}}}{1 + \alpha}
\left(\frac{1+\alpha}{c_{1,\frac{\alpha-1}{2}}}\right)^{\frac{\alpha k + 1}{\alpha+1}}
\Gamma\left(\frac{1 + \alpha k}{1 + \alpha}\right).
\end{align*}
From Theorem~\ref{theo:f_vect} we recall the formula
$$
\E f_k(P_{n,d}^{\beta})
=
2 \sum_{s=0}^\infty \binom n {d-2s} \binom {d-2s}{k+1} I_{n,d-2s}(2\beta+d) J_{d-2s,k+1}\left(\beta + s + \frac 12\right).
$$
It follows from the above that the $s$-th term of the sum behaves like a constant multiple of $n^{d-2s-1\over 2\beta+d+1}$, as $n\to\infty$. Consequently, as $n\to\infty$, the term with $s=0$ dominates all other terms and we arrive at
\begin{align*}
\E f_k(P_{n,d}^{\beta}) &\sim
n^{\frac{d - 1}{2\beta+d+1}}\frac{2}{d!} \binom {d}{k+1}  J_{d,k+1}\left(\beta + \frac 12\right)
\frac{c_{1, \frac {(2\beta+d) d -1}{2}}}{2\beta+d+1}\\
&\qquad\qquad\qquad\times\left(\frac{2\beta+d+1}{c_{1,\frac{2\beta+d -1}{2}}}\right)^{\frac{(2\beta+d)  d + 1}{2\beta+d +1}}
\Gamma\left(\frac{(2\beta+d)  d+1}{2\beta+d+1}\right).
\end{align*}
This completes the proof.
\end{proof}

\subsection{Poisson hyperplane tessellations}

Recall the definitions of the Poisson point process $\Pi_{d,\alpha}$ and the zero cell $Z_{d,\alpha}=Z_{d,\alpha,1}$ given in Sections~\ref{subsec:PPP} and~\ref{subsec:PHT}.

\begin{proof}[Proof of Theorem \ref{thm:PoissonHyperplanes}]
Let us define the (measurable) mapping
$$
T:\R^d\setminus\{0\}\to A(d,d-1),\qquad
x\mapsto H(x):=\{y\in\R^d:\langle x,y\rangle=1\}.
$$
The well-known mapping property of Poisson processes (see, for example, \cite[Theorem 5.1]{LastPenrosePPPBook}) implies that the image process $T\Pi_{d,\alpha}$ is a Poisson process on the space $A(d,d-1)$. Its probability law is rotationally invariant, since $\Pi_{d,\alpha}$ has the same property. Next, we consider the distance distribution. For $s>0$ we first compute, by transformation into spherical coordinates, that
\begin{align}\label{eq:DistanceDistributionComp1}
\int_{\{x\in\R^d:\|x\|>s\}}{\dint x\over\|x\|^{d+\alpha}} = \omega_d\int_s^\infty {\dint r\over r^{\alpha+1}} = \omega_d{s^{-\alpha}\over \alpha}.
\end{align}
On the other hand, writing $d(0,H)$ for the distance of a hyperplane $H\in A(d,d-1)$ to the origin, we have that $|\{H\in\eta_{d,\alpha,\gamma}:d(0,H)\leq s\}|$ ($|\,\cdot\,|$ denotes the cardinality of a set) is Poisson distributed with mean
\begin{align*}
\Theta_{d,\alpha,\gamma}(\{H\in A(d,d-1):d(0,H)\leq s\})
=
2\gamma\int_{0}^s|t|^{\alpha-1}\,\dint t
={2\gamma s^\alpha\over\alpha},
\end{align*}
where we used the definition of $\Theta_{d,\alpha,\gamma}$ given in~\eqref{eq:Theta_def}.
Thus,
\begin{align}\label{eq:DistanceDistributionComp2}
\E|\{H\in\eta_{d,\alpha,\gamma}:d(0,H)^{-1}\geq s\}| = \E|\{H\in\eta_{d,\alpha,\gamma}:d(0,H)\leq s^{-1}\}| = 2\gamma\,{s^{-\alpha}\over\alpha}.
\end{align}
So, a comparison of \eqref{eq:DistanceDistributionComp1} with \eqref{eq:DistanceDistributionComp2} shows that the Poisson processes $T\Pi_{d,\alpha}$ and $\eta_{d,\alpha,\gamma}$ with $\gamma=\omega_d/2$ have the same distribution. In view of the definition of the mapping $T$ and the definition of the dual of a convex body this implies that the random polytopes $(\conv\Pi_{d,\alpha})^\circ$ and $Z_{d,\alpha,\omega_d/2}$ (or, equivalently, $\conv\Pi_{d,\alpha}$ and $Z_{d,\alpha,\omega_d/2}^\circ$) are identically distributed. The claim for the expected $f$-vectors follows from~\eqref{eq:DualityFvector}, the fact that a change in the parameter $\gamma$ is equivalent to a rescaling of the zero cell and that the $f$-vector is invariant under such scalings.
\end{proof}

\begin{proof}[Proof of Theorem \ref{thm:PoissonHyperplanesDtoInfinity}]
Fix some $k\in \{0,1,2,\ldots\}$.
Theorem \ref{thm:PoissonHyperplanes} and Theorem \ref{theo:f_vect_poisson} imply that
\begin{align*}
 \E f_k(Z_{d,\alpha})
&= \E f_{d-k-1}(\conv\Pi_{d,\alpha})\\
&= 2\sum_{s=0}^{\lfloor{\frac k2}\rfloor}{\Gamma({(d-2s)\alpha+1\over 2})\over\Gamma({(d-2s)\alpha\over 2})}\left({\Gamma({\alpha\over 2})\over\Gamma({\alpha+1\over 2})}\right)^{d-2s}{(\sqrt{\pi}\alpha)^{d-2s-1}\over d-2s}{d-2s\choose d-k}\tilde{J}_{d-2s,d-k-1}\left({d-2s-1+\alpha\over 2}\right)\\
&=: \sum_{s=0}^{\lfloor{\frac k2}\rfloor} T_{d,\alpha}(s).
\end{align*}
Recall that we assume that $\alpha=\alpha(d)>0$ is bounded away from $0$.  For fixed $s\in \{0,1,\ldots,\lfloor \frac k2 \rfloor\}$, Stirling's formula and the definition of the binomial coefficients yield
$$
{\Gamma({(d-2s)\alpha+1\over 2})\over\Gamma({(d-2s)\alpha\over 2})}\sim \sqrt{(d-2s)\alpha\over 2}
\qquad\text{and}\qquad
{d-2s\choose d-k}={d-2s\choose k-2s} \sim {d^{k-2s}\over(k-2s)!},
$$
as $d\to\infty$.
Moreover, we shall argue below that
\begin{equation}\label{eq:JTildeAsymptotic}
\lim_{d\to\infty} \tilde{J}_{d-2s,d-k-1}\left({d-2s-1+\alpha\over 2}\right)  = {1\over 2^{k-2s}}.
\end{equation}
This implies that the $s$th term in the above sum (together with the prefactor $2$) is equal to
\begin{align*}
T_{d,\alpha}(s)
&\sim 2\sqrt{(d-2s)\alpha\over 2}\left({\Gamma({\alpha\over 2})\over\Gamma({\alpha+1\over 2})}\right)^{d-2s}{(\sqrt{\pi}\,\alpha)^{d-2s-1}\over d-2s}{d^{k-2s}\over(k-2s)!}{1\over 2^{k-2s}}\\
&\sim {\sqrt{\alpha}\over 2^{k-2s-{1\over 2}}}\left({\Gamma({\alpha\over 2})\over\Gamma({\alpha+1\over 2})}\right)^{d-2s}{(\sqrt{\pi}\,\alpha)^{d-2s-1}\over (k-2s)!}d^{k - 2s-{1\over 2}},
\end{align*}
as $d\to\infty$. Next we show that the term $T_{d,\alpha}(s)$ with $s=0$ is asymptotically dominating the terms with $s\neq 0$. For every $s\in \{0,1,\ldots,\lfloor \frac k2 \rfloor\}$, we have
$$
\frac{T_{d,\alpha}(0)}{T_{d,\alpha}(s)} \sim  \left({\sqrt \pi \alpha d\over 2}  \,  {\Gamma({\alpha\over 2})\over\Gamma({\alpha+1\over 2})}\right)^{2s}  \frac{(k-2s)!} {k!}.
$$
If $s\neq 0$, then the right-hand side goes to $+\infty$, as $d\to\infty$, since the function $\Gamma({\alpha\over 2}+1)/\Gamma({\alpha+1\over 2})$ is bounded away from $0$ for $\alpha>0$.
Thus, $T_{d,\alpha} (s) = o(T_{d,\alpha} (s))$ for every $s\in \{1,2,\ldots, \lfloor \frac k2 \rfloor\}$ and hence
$$
\E f_k(Z_{d,\alpha}) \sim T_{d,\alpha}(0) \sim {\sqrt{\alpha}\over 2^{k-{1\over 2}}}\left({\Gamma({\alpha\over 2})\over\Gamma({\alpha+1\over 2})}\right)^{d}{(\sqrt{\pi}\,\alpha)^{d-1}\over k!}d^{k-{1\over 2}}.
$$
It remains to prove \eqref{eq:JTildeAsymptotic}. For $\ell,m\in\N$ consider the quantity
$$
\tilde{J}_{m,\ell-1}(\beta)=\E\beta([Z_1,\ldots,Z_\ell],[Z_{1},\ldots,Z_m]),
$$
where the points $Z_1,\ldots,Z_m\in\R^{m-1}$ are i.i.d.\ with beta' density $\tilde{f}_{m-1,\beta}$. In the proof of Theorem \ref{theo:f_vect} we have seen that $\beta([Z_1,\ldots,Z_\ell],[Z_{1},\ldots,Z_m])$ has the same law as $\beta(\pos(V_1-V,\ldots,V_{m-\ell}-V))$, where
\begin{itemize}
\item[(a)] $V,V_1,\ldots,V_{m-\ell}\in\R^{m-\ell}$ are independent and such that
\item[(b)] $V_1,\ldots,V_{m-\ell}$ have density $\tilde{f}_{m-\ell,{2\beta-\ell+1\over 2}}$ and
\item[(c)] $V$ has density $\tilde{f}_{m-\ell,{(2\beta-m)\ell+m\over 2}}$.
\end{itemize}
Now, we substitute $m=d-2s$, $\ell=d-k$ and $\beta={d-2s-1+\alpha\over 2}$ and notice that the relevant beta'-parameters are
$$
m-\ell\quad\text{and}\quad\kappa(d) := {2\beta-\ell+1\over 2}={\alpha+k-2s\over 2}
$$
for the random variables $V_1,\ldots,V_{m-\ell}$ in (b) and
$$
m-\ell\quad \text{and}\quad
\eta(d) := {(2\beta-m)\ell+m\over 2}={\alpha(d-k)+k-2s\over 2}
$$
for the random variable $V$ in (c).
We are interested in the large $d$ behavior of
$$
\tilde{J}_{d-2s,d-k-1}\left({d-2s-1+\alpha\over 2}\right) =
\E \beta \pos(V_1(d)-V(d),\ldots, V_{k-2s}(d)-V(d)),
$$
where  $V_1(d),\ldots,V_{k-2s}(d)$ with density $\tilde{f}_{k-2s,\kappa(d)}$ and $V(d)$ with density $\tilde{f}_{k-2s,\eta(d)}$ are independent.

\vspace*{2mm}
\noindent
\textit{Case 1.}
Assume first that $\kappa(d)$ converges to some finite $\kappa \in (\frac{k-2s}{2}, \infty)$, as $d\to\infty$. Note that the value $\frac{k-2s}{2}$ can be excluded by the that assumption $\inf_{d\in\N} \alpha(d) > 0$. For the same reason,  we have $\eta(d)\to \infty$, as $d\to\infty$.  Observe that the beta' distribution with a second parameter going to $\infty$ weakly converges to the Dirac measure at $0$. It follows that, as $d\to\infty$, the collection of random points $(V_1(d),\ldots,V_{k-2s}(d),V)$ converges in distribution to $(W_1,\ldots,W_{k-2s},0)$, where $W_1,\ldots,W_{k-2s}$ are i.i.d.\ with density $\tilde{f}_{k-2s,\kappa}$. Consequently, we have
\begin{align*}
\lim_{d\to\infty} \tilde{J}_{d-2s,d-k-1}\left({d-2s-1+\alpha\over 2}\right)
&=
\lim_{d\to\infty} \E \beta \pos(V_1(d)-V(d),\ldots, V_{k-2s}(d)-V(d))\\
&=
\E \beta \pos(W_1,\ldots, W_{k-2s}).
\end{align*}
However, since the distribution of $W_i$ is the same as that of $-W_i$ for every $i=1,\ldots,k-2s$, we must have
$$
\E\beta(\pos(W_1,\ldots,W_{k-2s})) = {1\over 2^{k-2s}},
$$
for symmetry reasons. This proves~\eqref{eq:JTildeAsymptotic} in Case~1.

\vspace*{2mm}
\noindent
\textit{Case 2.}
Assume now that $\kappa(d)$ diverges to $+\infty$, as $d\to\infty$. By the definition of $\kappa(d)$ and $\eta(d)$ we have $\eta(d) \to\infty$ and moreover $\kappa(d) = o(\eta(d))$, as $d\to\infty$. By Lemma~\ref{lem:gauss_limit}, the random points
$$
\sqrt{2 \kappa(d)} V_1(d), \ldots, \sqrt{2 \kappa(d)} V_{k-2s}(d), \sqrt{2 \eta(d)} V(d)
$$
thus converge in distribution to independent random points $W_1,\ldots, W_{k-2s}, W$ with standard normal distribution on $\R^{k-2s}$.
Combining this with $\kappa(d) = o(\eta(d))$, we obtain the convergence
$$
\sqrt{2 \kappa(d)} \, (V_1(d), \ldots, V_{k-2s}(d),  V(d)) \overset{d}{\longrightarrow}  (W_1,\ldots,W_{k-2s}, 0),
$$
as $d\to\infty$. Since the standard normal distribution is centrally symmetric with respect to the origin, the same symmetry argument as in Case~1 proves the validity of~\eqref{eq:JTildeAsymptotic}, that is
$$
\lim_{d\to\infty} \tilde{J}_{d-2s,d-k-1}\left({d-2s-1+\alpha\over 2}\right)  = {1\over 2^{k-2s}}.
$$

\vspace*{2mm}
\noindent
\textit{Case 3.} In general, $\kappa(d)$ need not converge to any finite or infinite limit. However, any subsequence of $\kappa(d)$ has a subsubsequence to which either Case~1 or Case~2 can be applied, thus showing that~\eqref{eq:JTildeAsymptotic} holds without additional assumptions.  This completes the proof of Theorem~\ref{thm:PoissonHyperplanesDtoInfinity}.
\end{proof}

\section*{Acknowledgement}
We are grateful to an anonymous referee for his/her report. The comments and suggestions were very helpful for us to improve the style and the presentation of the paper. ZK and CT were supported by the DFG Scientific Network {\it Cumulants, Concentration and Superconcentration}. ZK was supported by the German Research Foundation under Germany's Excellence Strategy  EXC 2044 -- 390685587, Mathematics M\"unster: Dynamics - Geometry - Structure.

\addcontentsline{toc}{section}{References}
\bibliography{beta_poly_bib}
\bibliographystyle{plainnat}

\vspace{1cm}

\footnotesize

\textsc{Zakhar Kabluchko:} Institut f\"ur Mathematische Stochastik, Westf\"alische Wilhelms-Universit\"at M\"unster\\
\textit{E-mail}: \texttt{zakhar.kabluchko@uni-muenster.de}

\bigskip

\textsc{Christoph Th\"ale:} Fakult\"at f\"ur Mathematik, Ruhr-Universit\"at Bochum\\
\textit{E-mail}: \texttt{christoph.thaele@rub.de}

\bigskip

\textsc{Dmitry Zaporozhets:}  St.\ Petersburg Department of Steklov Mathematical Institute\\
\textit{E-mail}: \texttt{zap1979@gmail.com}

\end{document}